\documentclass[a4paper]{article}

\usepackage{amssymb}
\usepackage{polski}
\usepackage[english, polish]{babel}
\usepackage[utf8]{inputenc}
\usepackage{amsthm}
\usepackage{amsmath}
\usepackage{amsfonts}
\usepackage{enumerate}
\usepackage[alphabetic]{amsrefs}
\usepackage{array}

\newcolumntype{L}[1]{>{\raggedright\let\newline\\\arraybackslash\hspace{0pt}}m{#1}}
\newcolumntype{C}[1]{>{\centering\let\newline\\\arraybackslash\hspace{0pt}}m{#1}}
\newcolumntype{R}[1]{>{\raggedleft\let\newline\\\arraybackslash\hspace{0pt}}m{#1}}

\usepackage{mathrsfs}
\usepackage{url}
\usepackage[all]{xy}

\newtheorem{theorem}{Theorem}[section]
\newtheorem{prop}[theorem]{Proposition}
\newtheorem{lemma}[theorem]{Lemma}

\newtheorem{conj}[theorem]{Conjecture}
\newtheorem{hyp}{Hypothesis}
\newtheorem*{theorem*}{Theorem}
\newtheorem*{conj*}{Conjecture}
\newtheorem*{qn*}{Question}

\theoremstyle{definition}

\newtheorem{defi}{Definition}
\newtheorem{remark}[defi]{Remark}

\theoremstyle{remark}

\numberwithin{equation}{section}

\title{Almost primes in various settings}
\date{}

\author{Pawe\l  ~Lewulis\thanks{Supported by NCN Preludium 11, 2016/21/N/ST1/02599.}}

\begin{document}

\maketitle

\renewcommand{\thepage}{\roman{page}}

\selectlanguage{english}
\begin{abstract}
Let $k \geq 3$ and let $L_i(n) = A_in + B_i$ be some linear forms such that $A_i$ and $B_i$ are integers. Define ${\mathcal{P}(n) = \prod_{i=1}^k L_i(n)}$. For each $k$ it is known that $\Omega (\mathcal{P} (n) ) \leq \rho_k$ infinitely often for some integer $\rho_k$. We improve the possible values of $\rho_k$ for $4 \leq k \leq 10$ assuming $GEH$. We also show that we can take $\rho_5=14$ unconditionally. As a by-product of our approach we reprove the $\rho_3=7$ result which was previously obtained by Maynard who used techniques specifically designed for this case. 
\end{abstract}

\renewcommand{\thepage}{\roman{page}}

\selectlanguage{english}

\renewcommand{\thetable}{\Alph{table}} 


\section{Introduction}
\renewcommand{\thepage}{\arabic{page}}
\setcounter{page}{1}

\subsection{State of the art}

One of the most famous problems concerning prime numbers is the so called twin prime conjecture, which can be stated in the following form:
\begin{conj}\label{TPconj}
There are infinitely many primes $p$ such that $p+2$ is also a prime. 
\end{conj}
This statement seems to be unachievable by current techniques and can be considered as a major target of studies in analytic number theory. In \cite{Chen} J. Chen proved the following result which can be seen as a weak variation of Conjecture \ref{TPconj}.
\begin{theorem}\label{chenthm}
There are infinitely many primes $p$ such that $p+2$ has at most two prime factors, each greater than $p^{1/10}$. 
\end{theorem}
Unfortunately, neither techniques developed in order to prove Theorem \ref{chenthm} nor any other like the circle method are able to find infinitely many twin primes. 

Let $p_n$ denotes the $n$-th consecutive prime number. D. Goldston, J. Pintz and C. Y. {Y{\i}ld{\i}r{\i}m in \cite{GYPprimestuple} proved the following revolutionary result.
\begin{theorem} We have 
\[ \liminf_{n \rightarrow \infty} (p_{n+1}-p_n) \leq 16\]
under the Elliott--Halberstam conjecture (cf. Conjecture \ref{EH}) and 
\[  \liminf_{n \rightarrow \infty} \frac{p_{n+1}-p_n}{\log p_n} = 0\]
unconditionally.
\end{theorem}

In \cite{GYPprimestuple} it was also proven that $EH[\frac{1}{2} + \eta]$ for any $\eta>0$ implies
\begin{equation}\label{Zhang}
 \liminf_{n \rightarrow \infty} (p_{n+1}-p_n) \leq C,
\end{equation}
for some explicit constant $C$ depending on $\eta$. On the other hand, proving $EH[\frac{1}{2} + \eta]$ seems to be unreachable by currently known techniques. 

The next giant step towards twin primes came with the work of Y. Zhang \cite{Zhang}. He showed that the inequality (\ref{Zhang}) is unconditionally true for $C=70~000~000$. His main idea was to get around the troubles with the Elliott--Halberstam conjecture by restricting attention only to sufficiently smooth numbers (i.e. numbers $n$ of the same order of magnitude as $x$ which have only prime divisors smaller than $x^\delta$ for some small, fixed $\delta$ and some large $x$). In such a case, Zhang was able to use the well developed theory of bilinear sums in order to cross the $1/2$ limit induced by the Bombieri--Vinogradov theorem.

The result of Zhang was later dramatically improved by J. Maynard in \cite{Maynard}, where he was able to take $C=600$.  For $i=1,\dots, k$ put $L_i(n)= A_i n + B_i$, where $A_i \in \mathbf{N}$, $B_i \in \mathbf{Z}$ and let $\mathcal{H}=\{L_1 , \dots , L_k \}$. Define
\begin{equation}\label{P(n)}
\mathcal{P}(n) = \prod_{i=1}^k L_i (n),
\end{equation}
\begin{equation}
\nu_p (\mathcal{H} )= \#\{ 1 \leq n \leq p \colon \mathcal{P} (n) \equiv 0 \bmod p \}.
\end{equation}
We call a tuple $\mathcal{H}$ \textit{admissible} if $\nu_p (\mathcal{H})<p$ for all primes and each two of the $L_i$ are distinct. We can also assume that each of the coefficients $A_i$ is composed of the same primes, none of which divides any of the $B_i$ (which can be done without loss of generality due to the nature of the problems studied here; cf. Conjecture \ref{HL} and Theorem \ref{MAIN}). Maynard's idea was to use a multidimensional variant of the Selberg sieve, i.e. he considered the expression
\begin{equation}
\sum_{N < n \leq 2N} \left(  \sum_{i=1}^k \mathbf{1}_{\mathbf{P}} (n+h_i) -1 \right) 
\left(   \sum_{\substack{ d_1, \dots , d_k \\ \forall i ~ d_i|L_i(n)}} \lambda_{d_1,\dots, d_k}  \right)^2
\end{equation}
for $\mathbf{1}_{\mathbf{P}}$ being the indicator function of primes, an admissible tuple $\{ n+h_1, \dots , n+h_k \}$ and the weights $\lambda_{d_1, \dots , d_k}$ suitably chosen to make the whole sum greater than $0$ for sufficiently large $N$. In the same paper Maynard also proved that 
\begin{equation}
\liminf_{n \rightarrow \infty} (p_{n+m} - p_n) < \infty
\end{equation}
for every positive integer $m$. Suprisingly, his proof of these two facts does not rely on any distributional claims stronger than the Bombieri--Vinogradov theorem.

Methods developed by Maynard were further investigated and up to this moment the best achievements concerning small gaps between primes were obtained in the Polymath Project \cite{Polymath8}. 

\begin{theorem} We have 
\begin{enumerate}
\item $\liminf_{n \rightarrow \infty} \left( p_{n+1}-p_n \right) \leq 246$ unconditionally,
\item $\liminf_{n \rightarrow \infty} \left( p_{n+1}-p_n \right) \leq 12$ on $EH$,
\item $\liminf_{n \rightarrow \infty} \left( p_{n+1}-p_n \right) \leq 6$ on $GEH$  (cf. Conjecture \ref{GEH}). 
\end{enumerate}
\end{theorem}
The second statement is actually due to Maynard \cite{Maynard}. There is a hope that some advances on the numerics will allow us to change $12$ into $8$ in the nearest future. The third result can be considered to be the most interesting because it is already shown that it reaches a limit of what is potentially provable by the sieve techniques due to the parity problem (see \cite{Polymath8} for details).

\subsection{Distributional claims on arithmetic functions}

The question how well the prime numbers are distributed among arithmetic progressions is one of the major problems in analytic number theory. We state a famous conjecture which seems to be true due to heuristic reasoning.
\begin{conj}[Elliott--Halberstam conjecture]\label{EH} Let $\theta \in (0,1)$, $A \geq 1$ be fixed. For every $Q \ll x^\theta$,  we have 
\[ \sum_{q \leqslant  Q} \max_{\substack{a \\ (a,q)=1 }} \left| \pi (x; a, q) - \frac{\pi (x)}{\varphi (q)} \right| \ll  x\log^{-A} x \]
(cf. Subsection 1.4 for the notation).
\end{conj} 
 We will refer to this conjecture for some specific exponent $\theta$ by $EH[\theta ]$. The best known result of this kind is due to Bombieri and Vinogradov.

\begin{theorem}\label{BV} $EH[ \theta ]$ holds for every $\theta \in (0,1/2)$.
\end{theorem}
One can view the Bombieri--Vinogradov theorem as a `generalised Riemann hypothesis on average'. It is also a very powerful substitute for $GRH$ in the sieve-theoretical context. 

There exists also a conjecture much more general than $EH$. It asserts that not only the prime counting function $\pi$ or von Mangoldt function $\Lambda$ but all functions equipped with sufficiently strong bilinear structure and having a good correlation with arithmetic sequences do satisfy the Elliott--Halberstam conjecture. 

\begin{conj}[Generalised Elliott--Halberstam conjecture]\label{GEH}
Let $\theta \in (0,1)$, $\varepsilon>0$, $A \geq 1$ be fixed. Let $N,M$ be quantities such that $x^\varepsilon \ll N$, $M \ll x^{1-\varepsilon}$ with $x \ll NM \ll x$, and let $\alpha,\beta \colon \mathbf{N} \rightarrow \mathbf{R}$ be sequences supported on $[N,2N]$ and $[M,2M]$, respectively, such that one has the pointwise bound 
\[ |\alpha (n) | \ll \tau (n)^{O(1)} \log^{O(1)} x;~~~~|\beta (m)| \ll \tau (m)^{O(1)} \log^{O(1)} x \]
for all natural numbes $n,m$. Suppose also that $\beta$ obeys the Siegel--Walfisz type bound
\[ \left| ~\sum_{\substack{ (n,r)=1 \\ n \equiv a \bmod q }} \beta (n) - \frac{1}{\varphi (q)} \sum_{\substack{ (n,qr)=1}} \beta (n) \right| \ll \tau (qr)^{O(1)} M \log^{-A} x \]
for any $q,r \geq 1$, any fixed $A$, fixed $C \geq 0$, and any primitive residue class $a~(q)$. Then for any $Q \ll x^{\theta}$, we have
\[ \sum_{q \leq Q} \tau(q)^C \max_{y \leq x} \max_{\substack{a \\ (a,q)=1} }  \left| ~\sum_{\substack{ n \leq y  \\ n \equiv a \bmod q }} (\alpha * \beta ) (n) - \frac{1}{\varphi (q)} \sum_{\substack{ n \leq y \\ (n,q)=1 }} (\alpha * \beta ) (n) \right|
 \ll x \log^{-A} x. \]
\end{conj} 

The broad generalization of this kind first appeared in \cite{GEH}. The best known result in this direction is currently proven by Motohashi \cite{Motohashi}.
\begin{theorem}\label{GBV} $GEH[ \theta ]$ holds for every $\theta \in (0,1/2)$.
\end{theorem}
It is possible to get $EH[\theta ]$ easily from $GEH [\theta]$ by Vaughan's identity as shown in \cite{Polymath8}. The importance of $GEH [\theta ]$ stands on the fact that it allows us to obtain important results of $EH$ type concerning almost primes as in the following result almost completely analogous to what is proven in \cite[Subsection `The generalized Elliott--Halberstam case']{Polymath8}. 

\begin{theorem}\label{GEHwniosek} Assume $GEH[\theta]$. Let $r \geq 1$, $\epsilon > 0$ and $A\geq 1$ be fixed, let 
\[ \Delta_{r,\epsilon} = \{ (t_1,\dots,t_r) \in [\epsilon,1]^r \colon ~ t_1 \leq \dots \leq t_r;~ t_1+\dots+t_r=1\}, \]
and let $F \colon \Delta_{r,\epsilon} \rightarrow {\bf R}$ be a fixed smooth function. Let $\widetilde F \colon {\bf N} \rightarrow {\bf R}$ be the function defined by setting
\[\displaystyle  \widetilde F(n) = F \left( \frac{\log p_1}{\log n}, \dots, \frac{\log p_r}{\log n} \right) \]
whenever $n=p_1 \dots p_r$ is the product of $r$ distinct primes $p_1 < \dots < p_r$ with $p_1 \geq x^\epsilon$ for some fixed $\epsilon>0$, fixed $C \geq 0$, and $\widetilde F(n)=0$ otherwise. Then for every $Q \ll x^\theta$, we have
\[ \displaystyle \sum_{q \leq Q}  \tau(q)^C \max_{y \leq x} \max_{\substack{a \\ (a,q)=1}}
 \left| ~\sum_{\substack{ n \leq x  \\ n \equiv a \bmod q }} \widetilde{F} (n) - \frac{1}{\varphi (q)} \sum_{\substack{ n \leq x \\ (n,q)=1}} \widetilde{F} (n) \right|
  \ll x \log^{-A} x.\]
\end{theorem}
This work is focused on almost primes, so that is the reason, why in Theorem \ref{MAIN} we have results achieved under $GEH$, but not under $EH$.

\subsection{The main result and key ingredients of its proof} 

The twin prime conjecture is equivalent to
\begin{equation}
 \liminf_{n \rightarrow \infty} \Omega (n(n+2)) = 2.
\end{equation}
The above statement with $3$ instead of $2$ is a slightly weaker version of Chen's theorem. We can also formulate the Dickson--Hardy--Littlewood $k$--tuple conjecture this way.
\begin{conj}\label{HL} For every admissible $k$--tuple $\mathcal{H}$, we have
\[ \liminf_{n \rightarrow \infty} \Omega ( \mathcal{P}(n)) = k, \]
where $\mathcal{P}(n)$ is defined as in (\ref{P(n)}). 
\end{conj}
Maynard \cite{3-tuples, MaynardK} proved that 
\begin{equation}\label{glowne_szacowanie}
\liminf_{n \rightarrow \infty} \Omega ( \mathcal{P}(n))  \leq \rho_k,
\end{equation}
where $\rho_k$ are given as in a table below. 
\begin{table}[h!]
\centering
\text{Table A.}
\vspace{1mm}
\\
  \begin{tabular}{ | C{0.7cm} |  C{1cm}  | C{1cm}  | C{1cm}  | C{1cm} |  C{1cm} | C{1cm} | C{1cm} |  C{1cm} | }
    \hline
    $k$ & 3 & 4 & 5 & 6 & 7 & 8 & 9 & 10 \\ \hline
    $\rho_k$ & 7 & 11 & 15 & 18 & 22 & 26 & 30 & 34 \\
      \hline
  \end{tabular}
\end{table}
\\ It was an improvement upon the results of Diamond and Halberstam \cite{DH} for ${3 \leq k \leq 6}$, and Ho and Tsang \cite{HT} for $7\leq k \leq 10$. The proof employed the following equality valid for any $n$ square-free and any $y \geq n^{1/2}$:
\begin{equation}\label{1.8}
\Omega (n) = \sum_{\substack{ p|n \\ p \leq y}} \left( 1 - \frac{\log p}{\log y} \right) + \frac{\log n}{\log y} + \sum_{r=1}^\infty \chi_r (n),
\end{equation}
where
\begin{equation}
 \chi_r (n) = 
\begin{cases}
- \left( \frac{\log n}{\log y} - 1 - \sum_{i=1}^{r-1} \frac{\log p_i}{\log y} \right), & \text{if~} n=p_1 \dots p_r \text{~with~ }  \\
& ~~  p_1 < \dots < p_{r-1} \leq y< p_r,\\
\\
0, & \text{otherwise. }
\end{cases}
\end{equation}
The idea was to consider the sum
\begin{equation}\label{sum}
\sum_{N < n \leq 2N} (c - \Omega(\mathcal{P} (n) )) \text{weight}(n),
\end{equation}
with  the classical Selberg sieve in the place of weights and use the identity (\ref{1.8}) to prove that (\ref{sum}) is positive for sufficiently large $N$ and $c$ suitably chosen.

The objective of the present work is to improve upon the result of Maynard in the case $k=5$ and to draw stronger conclusions for $4 \leq k \leq 10$ assuming $GEH$. 
\begin{theorem}[Main Theorem]\label{MAIN} Given an admissible $5$--tuple $\mathcal{H}$ the inequality (\ref{glowne_szacowanie}) holds with $\rho_5=14$.
Moreover, assuming $GEH[2/3]$, for any admissible $k$--tuple $\mathcal{H}$, the same inequality holds with $\rho_k$ given in a table below:

\begin{table}[h!]
\centering
\text{Table B.}
\vspace{1mm}
\\
  \begin{tabular}{ | C{0.7cm} |  C{1cm}  | C{1cm}  | C{1cm}  | C{1cm} |  C{1cm} | C{1cm} | C{1cm} | }
    \hline
    $k$ & 4 & 5 & 6 & 7 & 8 & 9 & 10 \\ \hline
    $\rho_k$ & 10 & 13 & 17 & 20 & 24 & 28 & 32 \\
      \hline
  \end{tabular}
\end{table}
\end{theorem}

As a by-product we are also able to give an alternative proof of Maynard's result for $k=3$ (it is worth mentioning that previously this case required a specifically devised approach using the Diamond--Halberstam sieve).

Motivated by Maynard's successful proof concerning small gaps between consecutive primes, we apply the multidimensional sieve to the problem of $k$--tuples of almost primes. The main difficulty in this approach is that the variation of the sieve proposed by Maynard in \cite{Maynard} combined with techniques developed in \cite{MaynardK} is not strong enough for any point of Theorem \ref{MAIN} to be proven. 

The main parameters for our set-up are denoted $\vartheta$ and $\vartheta_0$. The former is related to the level of distribution in $GEH[ \theta ]$ by $\theta = 2 \vartheta $. The latter can be viewed as related to $y$ from the identity (\ref{1.11}) by $N^{\vartheta_0} \approx y$. In order to produce non-trivial results relying on $GEH$, we need to work with the parameter $\vartheta$ greater than $1/4$. Then, the constraint $\vartheta_0 + 2 \vartheta < 1$ in Proposition \ref{5.1} forces us to take $\vartheta_0 < 1/2$. As a result, we need a variation of identity (\ref{1.8}) which works also for $0 \leq y <n^{1/2}$. We can use the following simple equation 
\begin{equation}
1 = \sum_{r=1}^\infty \sum_{s=0}^r \mathbf{1}_{n=p_1 \dots p_r \text{ for some } p_1 < \dots < p_{r-s} \leq y < p_{r-s+1} < \dots < p_r},
\end{equation}
valid for square-free $n$ and get 
\begin{equation}\label{1.11}
\Omega (n) = \sum_{\substack{ p|n \\ p \leq y}} \left( 1 - \frac{\log p}{\log y} \right) + \frac{\log n}{\log y} + \sum_{r=1}^\infty \sum_{s=1}^r  \chi_{r,s} (n),
\end{equation}
where
\begin{equation}
 \chi_{r,s} (n) = 
\begin{cases}
- \left( \frac{\log n}{\log y} - s - \sum_{i=1}^{r-s} \frac{\log p_i}{\log y} \right), & \text{if~} n=p_1 \dots p_r \text{~with~ }  \\
& ~~  p_1 < \dots < p_{r-s} \leq y < p_{r-s+1} < \dots < p_r,\\
\\
0, & \text{otherwise. }
\end{cases}
\end{equation}
With this modification, we can use Propositions \ref{5.1}--\ref{5.4} with $\vartheta_0$ arbitrarily small and the final results will not be damaged in any way. In \cite{MaynardK} the author was forced to take $\vartheta_0$ not smaller than $1/2$ (in his work it was $r_1$ instead of $\vartheta_0$ and $r_2$ instead of $\vartheta$) and also $\vartheta$ not greater than $1/4$ which strictly blocked any possibility of using $GEH$ to amplify the existing results. 

The unconditional part of Theorem \ref{MAIN} is yet harder. The reason for that probably hides behind Conjecture \ref{CONJ}, supported by some numerical experiments, and the calculations mentioned in Table F. We can enhance our sieve by expanding its support as described in (\ref{SelSieve}) which is basically an idea inspired by \cite{Polymath8}. This prodecure does not allow us to take $\vartheta_0$ as large as $1/2$ any longer, but thanks to the identity (\ref{1.11}) this shall not be a major obstacle. 

The general strategy of our proof is fairly simple. We consider the sum from (\ref{sum}) and just like Maynard we wish to prove that it is greater than $0$ for sufficiently large $N$. We need some function $\text{weight}(n)$ which gives preference to the ``good candidates'' for almost primes. The multidimensional Selberg sieve used in this role allows us to transform the problem about $k$--tuples of almost primes into estimating four sums which can be viewed as a multidimensional analogues of sums $(5.8)$--$(5.11)$ from \cite{MaynardK}. Then, we have to recover the results about the sums $T_\delta$ and $T_\delta^{*}$, playing the crucial role in \cite{MaynardK}, in the multidimensional context. Appropriate lemmas are described in Chapter 2. In the end we are left with an optimization problem depending on various complicated integrals. This problem is studied in the last section.

\subsection{Notation}

The letter $p$ always denotes a prime number and ${\log}$ denotes the natural logarithm. We consider $N$ as a number close to infinity. To avoid any problems with the domains of logarithmic functions, we assume that $N>16$. We also use the notation $\mathbf{N}=\{1,2,3,\dots \}$.  By $\mathcal{H}$ we always denote an admissible $k$--tuple. The letter $k$ always denotes a positive integer greater than or equal to $3$.\\

We use the following functions which are common in analytic number theory:

\begin{itemize}
\item $\varphi (n) := \left| \left( \mathbf{Z} / n \mathbf{Z} \right)^\times \right|$ denotes Euler totient function; 
\item $\tau (n)  := \sum_{d|n} 1 $ denotes the divisor function; 
\item $\Omega (n)$ denotes the number of prime factors of $n$; 
\item $\pi (x) := \left\{ n \in \mathbf{N}: n \leq x, ~n \text{ is prime} \right\}$;
\item $\pi (x;q,a) := \left\{ n \in \mathbf{N}: n \leq x,~n \equiv a \bmod q, ~n \text{ is prime} \right\}$;
\item $(n_1, \dots , n_r)$ and $[n_1, \dots ,n_r]$ denote the greatest common divisor and the lowest common multiple, respectively;
\item For a logical formula $\phi$ we define the indicator function $\mathbf{1}_{\phi (x)}$ which equals $1$ when $\phi (x)$ is true and $0$ otherwise;
\item For two arithemtic functions $\alpha, \beta \colon \mathbf{N} \rightarrow \mathbf{C}$ we define their Dirichlet convolution by the formula $(\alpha * \beta )(n) = \sum_{d|n} \alpha (d) \beta (n/d)$.
\end{itemize}

We use the `big $O$' and the `small $o$' notation. The formula $f=O(g)$ or $f \ll g$ means that there exists a constant $C>0$ such that $|f(x)| \leq Cg(x)$ on the domain of $f$. In the second case, $f=o(g)$ means simply $\lim_{x \rightarrow \infty} f(x)/g(x) = 0 $. We see that the big $O$ makes sense also in the case when the considered functions are multivariate. In the `big $O$' or `$\ll$' notation the dependence on the variables $k, A_i, B_i, \vartheta, \vartheta_0$, the functions $W_0, W_r$,  and any other parameters declared as fixed, will not be mentioned explicitly. The notation $O_\epsilon (f(x))$ or $\ll_\epsilon$ means that the considered constant depends on the variable in the lower index.

\subsection{Constructing the sieve}

In order to deal with some small problematic primes, we use a device called the $W$-trick. For some $D_0$ we put 
\begin{equation}
W = \prod_{p < D_0} p
\end{equation}
and we further demand $n$ to lie in some residue class $\nu_0 \bmod W$ such that $(\mathcal{P}(\nu_0),W)=1$. We take
\begin{equation}
D_0 = \log \log \log N
\end{equation}
and without loss of generality we can assume that $N$ is so large that $D_0 > A_iB_j - A_jB_i$ and $D_0>A_i$  for all $i,j=1,\dots,k$. 

We construct the expanded multidimensional Selberg sieve in the following way:
\begin{equation}\label{SelSieve}
\lambda_{d_1, \dots, d_k} = 0 ~~~~\mbox{if}~~ \left( d , W \right) > 1  ~~\mbox{or}~~\mu \left( d \right)^2 \not= 1 ~~\mbox{or}~~  \exists_{j}~d/d_j \geq R ,
\end{equation}
where $d:= \prod_{i=1}^k d_i$ and $R$ to be chosen later. In the remaining cases we may choose the weights arbitrarily. 

The proper choice of the sieve weights should maximise the value of the sum (\ref{sum}). In the one-dimensional case the standard form of a weight is something similar to 
\begin{equation}
 \lambda_d \approx \mu (d) G \left(\frac{\log d}{\log R} \right)
 \end{equation}
for some smooth function $G$. In the multidimensional case we use weights of the form: 
\begin{equation}
 \lambda_{d_1, \dots , d_k} \approx \left( \prod_{i=1}^k \mu (d_i) \right) G \left(\frac{\log d_1}{\log R}, \dots , \frac{\log d_k}{\log R} \right).
 \end{equation}
The correlation with the M\"{o}bius function makes the sums containing sieve weights difficult to evaluate, since there are many positive terms and many negative ones, so we have a lot of cancellation to deal with. To overcome this obstacle we apply the \textit{reciprocity law} which transfers the $\lambda_{d_1,\dots ,d_k}$ into a set of new variables which are positive. Specifically, in the multidimensional case, we use the following lemma from \cite{Granville} which follows easily from the M\"{o}bius inversion. 
\begin{lemma}\label{RECIP}
Suppose that $L(d)$ and $Y(r)$ are two sequences of complex numbers, indexed by $d,r \in \mathbf{N}^k$, and supported on the $d_i$ and $r_i$ which are square-free, coprime to $W$ and satisfy $\prod_{i=1}^k d_i,~\prod_{i=1}^k r_i < R$. Then
\[ L(d_1, \dots , d_k) = \prod_{i=1}^k \mu (d_i)   \sum_{\substack{ r_1, \dots , r_k \\ \forall i~ d_i | r_i }} Y(r_1, \dots , r_k) \]
if and only if
\[ Y(r_1, \dots , r_k) = \prod_{i=1}^k \mu (r_i)   \sum_{\substack{ d_1, \dots , d_k \\ \forall i~ r_i | d_i }} L(d_1, \dots , d_k). \]
\end{lemma}
This lemma lets us define the Selberg weights indirectly in a convenient way. 
\begin{defi}[Selberg weights]\label{lambda_def}
We put
\begin{equation}\label{opis_lambda}
\lambda_{d_1,\dots ,d_k} =  \left( \prod_{i=1}^k \mu(d_i) d_i \right) \sum_{\substack{ r_1, \dots , r_k \\ \forall i~d_i|r_i }} \frac{ y_{r_1, \dots , r_k} }{ \prod_{i=1}^k \varphi( r_i )}
\end{equation}
for some $y_{r_1, \dots , r_k}$ being non-negative, square-free, coprime to $W$, and supported only on the $r_i$ satisfying $r/r_j < R$ for each $j=1,\dots , k$, where $r= \prod_{i=1}^k r_i$. By Lemma \ref{RECIP} the equality (\ref{opis_lambda}) is equivalent to
\begin{equation}\label{opis_y}
y_{r_1, \dots , r_k} = \left( \prod_{i=1}^k \mu(r_i) \varphi(r_i) \right) \sum_{\substack{ d_1, \dots , d_k \\ \forall i~r_i|d_i }} \frac{ \lambda_{d_1, \dots ,d_k} }{ \prod_{i=1}^k d_i }.
\end{equation}
For Lemma \ref{Ttilde} and further we shall take
\begin{equation}\label{opis_y2}
y_{r_1, \dots , r_k} =  \begin{cases}
        F \left( \frac{\log r_1}{\log R} , \dots ,  \frac{\log r_k}{\log R} \right) & \text{if } \mu \left( r \right)^2=1,~ (r,W)=1,~ \text{and}~  \forall_j~  r/r_j < R,  \\
        0 & \text{otherwise,}
        \end{cases}
\end{equation}
where $F$ is a nonzero function $F \colon \mathbf{R}^k \rightarrow \mathbf{R}_{\geq 0}$, supported on 
\[\mathcal{R}_k' = \{ (t_1, \dots , t_k) \in [0,1]^k \colon \sum_{\substack{ i=1 \\ i \not= j }}^k t_i \leq 1  \mbox{ for each $j=1, \dots , k$} \} \] 
and differentiable in the interior of this region. 
\end{defi}

\begin{remark} We may also notice that if the sieve weights $\lambda_{d_1 , \dots , d_k}$ are supported on the $d_i$ such that $\prod_{i=1}^k d_i < R $, then by the reciprocity law it is true that $\mbox{supp} \,(F) \subset \mathcal{R}_k$, where
\[ \mathcal{R}_k = \{ (t_1, \dots , t_k) \in [0,1]^k \colon \sum_{\substack{ i=1}}^k t_i \leq 1  \}.   \]
In such a case we assume only the differentiability of $F$ inside the interior of $\mathcal{R}_k$. We use this specific support while proving the $GEH$ case of Theorem \ref{MAIN}. It is less powerful than $\mathcal{R}_k' $ (it is obvious, because it is strictly smaller) but leads to much simpler calculations.
\end{remark}

It is also worth mentioning that this is also the support which was originally used by Maynard in \cite{Maynard}, where the multidimensional Selberg sieve appeared for the first time. 
We also define
\begin{align}
\lambda_{\max} &= \sup_{d_1, \dots, d_k} \left| \lambda_{d_1, \dots , d_k} \right|, ~~~~~~
y_{\max} = \sup_{r_1, \dots, r_k} \left| y_{r_1, \dots , r_k} \right|, \\
  F_{\max} &= \sup_{(t_1,\dots,t_k) \in [0,1]^k}  \left(  \left| F(t_1, \dots , t_k) \right| + \sum_{i=1}^k \left| \frac{\partial F}{\partial t_i} (t_1 , \dots , t_k) \right| \right) . \nonumber
\end{align}


\subsection{Sketch of the proof}

We start from 
\begin{equation}\label{1.22}
\mathcal{S}(\sigma; N, R_0, R, \mathcal{H}, \nu_0) = \sum_{\substack{ N < n \leq 2N \\ n \equiv \nu_0 \bmod W \\ \mu(\mathcal{P}(n))^2 =1}} \left( \sigma - \sum_{p| \mathcal{P}(n)} \left( 1 - \frac{\log p}{\log R_0} \right) \right) \Lambda_{\text{Sel}}^2 (n),
\end{equation}
where
\begin{equation}
 \Lambda_{\text{Sel}}^2 (n) = \left( \sum_{\substack{ d_1, \dots , d_k \\ \forall i ~ d_i | L_i (n)}} \lambda_{d_1, \dots , d_k} \right)^2
\end{equation}
and we choose some $\nu_0$ coprime to $W$. We also choose some positive constants $\vartheta$ and $\vartheta_0$ to be fixed later and put
\begin{equation}
R_0 = N^{\vartheta_0},~~~~~~~   R=N^{\vartheta}.
\end{equation}

We note that for $\mathcal{P}$ square-free we have
\begin{equation}\label{1.26}
\sum_{p| \mathcal{P}(n)} \left( 1 - \frac{\log p}{\log R_0} \right) = \Omega( \mathcal{P} (n)) - \frac{\log \mathcal{P}(n)}{\log R_0}.
\end{equation}
Now, we follow the reasoning of Maynard described in equations \cite[(5.4)--(5.11)]{MaynardK}. Notice that the precise shape of the sieve, the $W$-trick, and the usage of identity (\ref{1.11}) do not affect these equations so we may just rewrite the results with our choice of sieve weights:
\begin{equation}
\mathcal{S} \geq \sigma \mathcal{S}_0 - \mathcal{S}' - T_0 + \sum_{j=1}^k \sum_{r=1}^h \sum_{s=1}^r T_{r,s}^{(j)},
\end{equation}
for any $h \in \mathbf{N}$, where
\begin{align} 
\begin{split}
\mathcal{S}_0 &= \sum_{\substack {N < n \leq 2N \\ n \equiv \nu_0 \bmod W}} \Lambda_{\text{Sel}}^2 (n), \\
\mathcal{S}' &= \sum_{\substack {N < n \leq 2N \\ n \equiv \nu_0 \bmod W \\ \mu \left( \mathcal{P}(n) \right)^2\not=1 }} \left( \sigma - \sum_{p| \mathcal{P}(n)} \left( 1 - \frac{\log p}{\log R_0} \right) \right) \Lambda_{\text{Sel}}^2 (n), \\ 
T_0 &= \sum_{\substack {N < n \leq 2N \\ n \equiv \nu_0 \bmod W}} \sum_{\substack{ p| \mathcal{P}(n)  \\ p \leq R_0}}  \left( 1 - \frac{\log p}{\log R_0} \right)  \Lambda_{\text{Sel}}^2 (n), \\
 T_{r,s}^{(j)} &=  \sum_{\substack {N < n \leq 2N \\ n \equiv \nu_0 \bmod W}} \chi_{r,s} (L_j(n)) \Lambda_{\text{Sel}}^2 (n),
\end{split}
\end{align}
and
\begin{equation}\label{1.31}
 \chi_{r,s} (n) = 
\begin{cases}
  s -\frac{\log N}{\log R_0} + \sum_{i=1}^{r-s} \frac{\log p_i}{\log R_0} , & \text{if~} n=p_1 \dots p_r \text{~with~ }  \\
& ~~  n^\epsilon <  p_1 < \dots < p_{r-s} \leq n^{\vartheta_0} < p_{r-s+1} < \dots < p_r,\\
\\
0, & \text{otherwise. }
\end{cases}
\end{equation}

In order to calculate $\mathcal{S}_0, \mathcal{S}', T_0$ and $T_{r,s}^{(j)}$ we need multidimensional analogues of Propositions $5.1$$-$$5.4$ from \cite{MaynardK}.

\begin{prop}[Analogue of Proposition 5.1 from \cite{MaynardK}]\label{5.1} Let $W_0 \colon [0, \frac{\vartheta_0}{\vartheta}] \rightarrow \mathbf{R}_{\geq 0}$ be a piecewise smooth non-negative function. Assume one of the following hypotheses:
\begin{enumerate}
\item $\vartheta_0 + 2 \vartheta < 1$ and $\emph{supp} \,(F) \subset \mathcal{R}_k,$
\item $\vartheta_0 + \frac{2k}{k-1} \vartheta < 1$ and $\emph{supp} \,(F) \subset \mathcal{R}_k',$
\end{enumerate}
Then, for every $0 < \epsilon <  \vartheta_0/\vartheta  $, we have
\begin{align*} 
\Sigma_0 :&= \sum_{\substack {N < n \leq 2N \\ n \equiv \nu_0 \bmod W}}\left(  \sum_{\substack{ p| \mathcal{P}(n)  \\ p \leq R_0}} W_0 \left( \frac{\log p}{\log R} \right) \right) \Lambda_{\emph{Sel}}^2 (n)  \\
&= ~ \frac{\varphi(W)^k N (\log R)^{k}   }{W^{k+1}} \sum_{j=1}^k J_0^{(j)}  + O\left( \frac{F_{\max}^2 \varphi(W)^k N (\log R)^k (\epsilon + \frac{|\log \epsilon |}{D_0})}{W^{k+1}} \right),
 \end{align*}
 where
 \[ J_0^{(j)} = \int \limits_\epsilon^{\vartheta_0/\vartheta} \frac{W_0 (y)}{y} I_{0}^{(j)}  (y ) \, dy \]
and
 \[  I_{0}^{(j)}  (y) =  \int \limits_0^1 \cdots \int \limits_0^1 \left(F(t_1,\dots,t_k) -  F \left( t_1, \dots , t_{i-1} , 
t_j + y, t_{i+1}, \dots , t_k \right) \right)^2 \, dt_1 \, \dots  \, dt_k.  \] 
\end{prop}

\begin{prop}[Analogue of Proposition 5.2 from \cite{MaynardK}]\label{5.2}
Given $\epsilon>0$ and $r \in \mathbf{N}$, let
\[ \mathcal{A}_r:= \left\{  x\in [0,1]^{r-1} : \epsilon < x_1 < \dots < x_{r-1}, ~\sum_{i=1}^{r-1} x_i < \min( 1 - \vartheta, 1- x_{r-1})     \right\}. \]
Let $W_r \colon [0,1]^{r-1} \rightarrow \mathbf{R}_{\geq 0}$ be a piecewise smooth function supported on $\mathcal{A}_r$, such that
\[ \frac{\partial}{ \partial x_j}W_r (x) \ll W_r (x)~~~~\text{uniformly for $x \in \mathcal{A}_r$}. \]
Let
\[ {\beta}_r (n) = 
 \begin{cases}
        {W}_r \left(\frac{\log p_1}{\log n}, \dots , \frac{\log p_{r-1}}{\log n} \right) & \text{for } n=p_1\dots p_{r}, \mbox{with } p_1 \leq \dots \leq p_{r},\\
        0 & \text{otherwise.}
 \end{cases} \]
 Assume $GEH[2 \vartheta]$. We have
\[ \sum_{\substack {N < n \leq 2N \\ n \equiv \nu_0 \bmod W}} \beta_r (L_j(n))  \Lambda_{\emph{Sel}}^2 (n) =
 \frac{\varphi(W)^k N (\log R)^{k+1}}{W^{k+1} \log N} J_r^{(j)} 
+ O_\epsilon \left(\frac{F_{\max}^2 \varphi(W)^k N (\log R)^{k} }{W^{k+1} D_0} \right), \]
 where
 \begin{align*}
 J_r^{(j)} &= ~ \int \limits_{(x_1, \dots , x_{r-1}) \in \mathcal{A}_r}  \frac{W_r (x_1, \dots , x_{r-1}) I^{(j)}(x_1\vartheta^{-1}, \dots, x_{r-1}\vartheta^{-1})}{\left( \prod_{i=1}^{r-1} x_i \right)  \left( 1 - \sum_{i=1}^{r-1} x_i \right) }, \\
 I^{(j)} (x_1, \dots , x_{r-1}) &= ~   \int \limits_0^1 \dots \int \limits_0^1 \widetilde{I}^{(j)}  (x_1, \dots , x_{r-1}) ^2 \, dt_1 \, \dots \, dt_{j -1} \, dt_{j +1}\, \dots \,dt_k, \\
 \widetilde{I}^{(j)}  (x_1, \dots , x_{r-1}) &= ~ \int \limits_0^1 \sum_{J \subset \{1, \dots , r-1\}}  (-1)^{|J|} F \left( t_1, \dots, t_{j -1}, t_j + \sum_{i \in J} x_j, t_{j +1}, \dots , t_k\right) \, dt_j . 
\end{align*}

\end{prop}

\begin{prop}[Analogue of Proposition 5.3 from \cite{MaynardK}]\label{5.3}
Assume one of the following hypotheses:
\begin{enumerate}
\item $\vartheta < \frac{1}{2} - \eta$ and $\emph{supp} \,(F) \subset \mathcal{R}_k$
\item $\vartheta < \frac{k-1}{2k} - \eta$ and $\emph{supp} \,(F) \subset \mathcal{R}_k'$
\end{enumerate}
 for some positive $\eta$. Then, we have
\[ \sum_{\substack {N < n \leq 2N \\ n \equiv \nu_0 \bmod W \\ \mu \left( \mathcal{P}(n) \right)^2 \not=1 }} \Lambda_{\emph{Sel}}^2 (n) \ll_\eta
\frac{ F_{\max}^2 \varphi (W)^k N (\log R)^{k}  }{W^{k+1}D_0}. \]
\end{prop}

The first case of the next result is a part of \cite[Proposition 4.1]{Maynard}. The proof of the second case is analogous.

\begin{prop}[Analogue of Proposition 5.4 from \cite{MaynardK}]\label{5.4}
Assume one of the following hypotheses:
\begin{enumerate}
\item $\vartheta < \frac{1}{2} - \eta$ and $\emph{supp} \,(F) \subset \mathcal{R}_k$
\item $\vartheta < \frac{k-1}{2k} - \eta$ and $\emph{supp} \,(F) \subset \mathcal{R}_k'$
\end{enumerate}
for some positive $\eta$. Then, we have
\[  \sum_{\substack {N < n \leq 2N \\ n \equiv \nu_0 \bmod W  }} \Lambda_{\emph{Sel}}^2 (n) = \frac{\varphi(W)^k N (\log R)^{k} }{W^{k+1}} J + O\left( \frac{F^2_{\max} \varphi(W)^k N (\log R)^{k} }{W^{k+1}D_0} \right),  \]
where
\[J = \int \limits_0^1 \dots \int \limits_0^1 F(t_1, \dots , t_k)^2 \, dt_1 \, \dots \, dt_k.  \]
\end{prop}

The shape of conditions from Propositions \ref{5.1}--\ref{5.4} lead us to two different cases: $\mbox{supp} \,(F) \subset \mathcal{R}_k$ or $\mbox{supp} \,(F) \subset \mathcal{R}_k'$. These considerations motivate the following hypotheses. The first one is related to the unextended multidimensional sieve and is useful while proving the $GEH$ part of Theorem \ref{MAIN}. 

\begin{hyp}\label{A} Let $\vartheta_0$, $\vartheta$, and $F$ be such that
\begin{enumerate}
\item $\emph{supp} \,(F) \subset \mathcal{R}_k,$ 
\item $\vartheta \leq \vartheta_0$,
\item $\vartheta_0 + 2 \vartheta < 1$, 
\item $\vartheta < \frac{1}{2}$,
\item $GEH[2 \vartheta ]$.
\end{enumerate}
\end{hyp}
The second hypothesis is used in the proof of unconditional part of Theorem \ref{MAIN} and is related to the situation when our sieve has extended support.

\begin{hyp}\label{B}  Let $\vartheta_0$, $\vartheta$, and $F$ be such that
\begin{enumerate}
\item $\emph{supp} \,(F) \subset \mathcal{R}_k',$ 
\item $\vartheta \leq \vartheta_0$,
\item  $\vartheta_0 + \frac{2k}{k-1} \vartheta < 1$, 
\item  $\vartheta < \frac{k-1}{2k}$,
\item $GEH[2 \vartheta ]$.
\end{enumerate}
\end{hyp}
Note that if $\vartheta < 1/4$, then the fifth assumption in Hypothesis 2 holds by Theorem \ref{GBV}.

\subsection*{Acknowledgements}

I would like to thank my advisors, Maciej Radziejewski and Piotr Achinger for many valuable comments, corrections and discussions.  From the strictly mathematical side, I would also like to express my gratitude towards James Maynard whose work was a great inspiration for my thesis and towards Terence Tao whose blog allowed me to grasp many ideas concerning the analytic number theory and mathematics in general.


\section{Lemmata}

\subsection{The quantities $T_{\delta_1, \dots , \delta_k}$ and $T^{(\ell)}_{\delta}$}

In the next chapters we frequently apply a useful lemma proven in \cite[(5.9)]{Maynard} enabling us to estimate the Selberg weights (in \cite{Maynard} the author assumes $\mbox{supp} \, (F) \subset \mathcal{R}_k$, but in the $\mbox{supp} \, (F) \subset \mathcal{R}_k'$ case the proof is analogous).
\begin{lemma}\label{szacowanieSelbergow}
We have
\[ \lambda_{\max} \ll F_{\max} (\log R)^{k} .
\]
\end{lemma}
In \cite{MaynardK} the estimation of the two sums $T_\delta$ and $T^*_\delta $ makes up the crucial part of the proof. We propose their analogues for the multidimensional sieve. 

\begin{equation}
T_{\delta_1, \dots , \delta_k} = \sideset{}{'}\sum_{\substack{d_1, \dots , d_k \\ e_1, \dots , e_k  }} \frac{\lambda_{d_1, \dots , d_k} \lambda_{e_1, \dots , e_k}}{ \prod_{i=1}^k \frac{ [d_i , e_i, \delta_i] }{\delta_i}},
\end{equation}
\begin{equation}\label{2.2}
~~~~T^{(\ell)}_{\delta} = \sideset{}{'}\sum_{\substack{d_1, \dots , d_k \\ e_1, \dots , e_k \\ d_\ell , e_\ell|\delta }} \frac{\lambda_{d_1, \dots , d_k} \lambda_{e_1, \dots , e_k}}{ \prod_{\substack{i=1 \\ i \not= \ell}}^k \varphi \left( [d_i , e_i] \right) },
\end{equation}
where $\Sigma'$ means that $[d_1,e_1], \dots , [d_k,e_k]$ have to be pairwise coprime. 

In the next lemma we find a representation of $T_{\delta_1, \dots , \delta_k}$ in almost diagonal form. To this end, we combine the methods from the proofs of \cite[Lemma 5.1]{Maynard} and \cite[Lemma 6]{GYP}.

\begin{lemma}\label{Lemma2.2}
For $\delta_1, \dots , \delta_k \in \mathbf{N}$ pairwise coprime satisfying $\sum_{i=1}^k \omega (\delta_i) \ll 1$ we have
\begin{multline*} T_{\delta_1, \dots , \delta_k} =  \sum_{\substack{ u_1, \dots , u_k \\ \forall i~ (u_i , \delta_i)=1 }} \frac{1}{\prod_{i=1}^k \varphi( u_i ) }  \left( \sum_{\substack{ s_1, \dots , s_k \\ \forall i~ s_i | \delta_i }} \mu (s_1) \dots \mu (s_k) y_{u_1s_1, \dots , u_ks_k} \right)^2 \\ +~
O \left( \frac{y^2_{\max}\varphi (W)^k (\log R)^k }{ W^k D_0 } \right).
\end{multline*}
\end{lemma}

\begin{proof} Since the $s_i$ which are not all coprime to $W$ give no contribution to the inner sum,
both sides of the asserted equality remain constant if we divide one of the $\delta_i$ by a prime divisor of $(\delta_i, W)$. 
Hence, we can assume that $(\prod_{i=1}^k \delta_i , W)=1$. For each $i=1, \dots ,k$ put
\[ g_i(d) = \frac{d}{(d,\delta_i)}.\]
With this notation, we get
\begin{equation}
 T_{\delta_1, \dots , \delta_k} = \sideset{}{'}\sum_{\substack{d_1, \dots , d_k \\ e_1, \dots , e_k  }} \frac{\lambda_{d_1, \dots , d_k} \lambda_{e_1, \dots , e_k }  \prod_{i=1}^k g_i \left((d_i , e_i )\right) } {{ \left( \prod_{i=1}^k g_i(d_i) \right) \left(\prod_{i=1}^k g_i(e_i) \right) }}.
\end{equation}
Consider the Dirichlet convolutions
\[ \overline{g}_i = g_i * \mu . \]
Note that
\[ \overline{g}_i (p) = 
        \begin{cases}
        \varphi(p) & \text{for } p \nmid  \delta_i, \\
        0 & \text{if } p| \delta_i.
        \end{cases}
         \] 
By M\"{o}bius inversion, we have
\begin{align}\label{2.4}
 T_{\delta_1, \dots , \delta_k} &= \sideset{}{'}\sum_{\substack{d_1, \dots , d_k \\ e_1, \dots , e_k  }} \frac{\lambda_{d_1, \dots , d_k} \lambda_{e_1, \dots , e_k }  } {{ \left( \prod_{i=1}^k g_i(d_i) \right) \left(\prod_{i=1}^k g_i(e_i) \right) }}
 \sum_{\substack{ u_1, \dots , u_k \\ \forall i~ u_i | d_i , e_i}} \prod_{i=1}^k \overline{g}_i (u_i)  \\
&=  \sum_{ \substack{ u_1, \dots , u_k \\ \forall i ~ (u_i, \delta_i)=1 }} \left( \prod_{i=1}^k \varphi (u_i) \right)
\sideset{}{'}\sum_{\substack{d_1, \dots , d_k \\ e_1, \dots , e_k \\ \forall i ~ u_i | d_i , e_i }} \frac{\lambda_{d_1, \dots , d_k} \lambda_{e_1, \dots , e_k }  } {{ \left( \prod_{i=1}^k g_i(d_i) \right) \left(\prod_{i=1}^k g_i(e_i) \right) }}. \nonumber
 \end{align}
 We wish to drop the dependencies between the $d_i$ and the $e_j$ variables in the inner sum. The requirement that $[d_1,e_1], \dots , [d_k,e_k]$ are pairwise coprime is equivalent to $(d_i, e_j)=1$ for all $i \not= j$. This is because $\lambda_{d_1, \dots , d_k}$ is supported only on integers $d_1, \dots , d_k$ satisfying $(d_i, d_j) = 1$ for all $i \not= j$. We can also remove the restriction $(d_i, e_j)=1$ for all $i \not= j$ by noticing that 
\[  \mathbf{1}_{(d_i, e_j)=1} = \sum_{s_{i,j}|d_i, e_j } \mu (s_{i,j}).\]
Hence $ T_{\delta_1, \dots , \delta_k}$ equals
\begin{equation}\label{2.5}
  \sum_{ \substack{ u_1, \dots , u_k \\ \forall i ~ (u_i, \delta_i)=1 }} \left(  \prod_{i=1}^k \varphi (u_i) \right)
\sum_{\substack{ s_{1,2} , \dots s_{k,k-1}}} \left(  \prod_{\substack{ 1 \leq i , j \leq k \\ i \not= j }}\mu (s_{i,j}) \right)
  \sum_{\substack{d_1, \dots , d_k \\ e_1, \dots , e_k \\ \forall i ~ u_i | d_i , e_i \\ \forall i\not=j ~ s_{i,j}|d_i, e_j}} \frac{\lambda_{d_1, \dots , d_k} \lambda_{e_1, \dots , e_k }  } {{ \left( \prod_{i=1}^k g_i (d_i) \right) \left(\prod_{i=1}^k g_i (e_i) \right) }}
 \end{equation}
Note that $\lambda_{d_1, \dots , d_k}=0$, unless $(d_i,d_j)=1$ for all $i\not=j$, so the $s_{i,j}$ that are not coprime to $u_i$ or $u_j$ give no contribution to the sum. We can also impose further constraint that $s_{i_1, j_1}$ has to be coprime to $s_{i_2,j_2}$, if these two share any coordinate. We denote this condition by writing $*$ next to the sum.

We define
\begin{equation}\label{2.6}
 w_{r_1, \dots , r_k} = \left( \prod_{i=1}^k \mu (r_i) \varphi  (r_i) \right) \sum_{\substack{d_1, \dots , d_k \\ \forall i~ r_i|d_i}}  \frac{\lambda_{d_1, \dots , d_k}   } {{ \prod_{i=1}^k g_i (d_i) }}.
\end{equation}
 By Lemma \ref{RECIP} and the definition of $w_{r_1, \dots , r_k}$ we can change the variables as follows:
\begin{equation}\label{2.7}
\lambda_{d_1, \dots , d_k} = \left( \prod_{i=1}^k \mu(d_i) g_i(d_i) \right) 
 \sum_{\substack{r_1 , \dots r_k \\ \forall i ~ d_i | r_i }} \frac{w_{r_1, \dots , r_k}}{ \prod_{i=1}^k \varphi (r_i) }.
\end{equation}

From (\ref{2.7}), we obtain
\begin{equation}\label{2.8}
 \sum_{\substack{d_1, \dots , d_k \\ \forall i ~ u_i | d_i  \\ \forall i\not=j ~ s_{i,j}|d_i}} \frac{\lambda_{d_1, \dots , d_k}  } {{ \prod_{i=1}^k g_i (d_i)  }} =  
 \sum_{\substack{r_1 , \dots r_k}}  \frac{w_{r_1, \dots , r_k}}{ \prod_{i=1}^k \varphi (r_i) }
  \sum_{\substack{d_1, \dots , d_k \\ \forall i ~ u_i | d_i  \\ \forall i\not=j ~ s_{i,j}|d_i \\ \forall i ~ d_i | r_i }}   \left( \prod_{i=1}^k \mu(d_i)  \right)  = w_{a_1, \dots, a_k}  \prod_{i=1}^k \frac{ \mu (a_i) } { \varphi (a_i)},
\end{equation}
where $a_j = u_j \prod_{i \not= j} s_{j,i}$. By performing an analogous calculation, we get 
\begin{equation}\label{2.9} 
\sum_{\substack{e_1, \dots , e_k \\ \forall i ~ u_i | e_i  \\ \forall i\not=j ~ s_{j,i}|e_i}} \frac{\lambda_{e_1, \dots , e_k}  } {{ \prod_{i=1}^k g_i (e_i)  }} =   w_{b_1, \dots, b_k} \prod_{i=1}^k \frac{ \mu (b_i) }{ \varphi (b_i)},
  \end{equation}
where $b_j =  u_j \prod_{i \not= j} s_{i,j}$. We observe that when $a_j$, respectively $b_j$, are not all square-free, the right-hand side of (\ref{2.8}), respectively (\ref{2.9}), vanishes, so we can rewrite $\mu (a_i)$ as $\mu (u_j) \prod_{i\not=j} \mu (s_{j,i})$ and do the same thing with $\mu (b_j)$, $\varphi (a_j)$ and $\varphi (b_j)$. Combining this fact with (\ref{2.5}), (\ref{2.8}), and (\ref{2.9}), we find
 \begin{equation}\label{2.10}
 T_{\delta_1, \dots , \delta_k} = 
 \sum_{ \substack{ u_1, \dots , u_k \\ \forall i ~ (u_i, \delta_i)=1 }} \left(  \prod_{i=1}^k \frac{\mu (u_i)^2}{\varphi (u_i)} \right)
 \sideset{}{^{*}}\sum_{\substack{ s_{1,2} , \dots s_{k,k-1}}} \left(  \prod_{\substack{ 1 \leq i , j \leq k \\ i \not= j }} \frac{\mu (s_{i,j})}{ \varphi (s_{i,j})^2 } \right) w_{a_1 , \dots , a_k} w_{b_1, \dots , b_k}.
 \end{equation}
 Again, note that the summands of the inner sum vanish when the $a_j$ and $b_j$ are not all square-free, by the definition of $w_{r_1, \dots, r_k}$.
 
The relation $s_{i,j} | d_i$ for all $i\not=j$ implies that there is no contribution from $s_{i,j}$ satisfying $(s_{i,j},W)\not=1$. Therefore, we can restrict our considerations to the $s_{i,j}$ such that $s_{i,j}=1$ for all $i \not= j$ or $s_{l,m}>D_0$ for some pair $l \not= m$. 

Consider the case when $s_{i,j}=1$ for all $i\not=j$. We have to calculate the expression
 \begin{equation}\label{formula_T}
\sum_{\substack{ u_1, \dots , u_k \\ \forall i ~ (u_i,\delta_i)=1 }} \frac{w_{u_1,\dots,u_k}^2 }{\prod_{i=1}^k \varphi (u_i) }.
 \end{equation}
  We have
\begin{align}\label{2.12}
w_{r_1, \dots , r_k} &=  \left( \prod_{i=1}^k  \frac{ \mu (r_i)  \varphi (r_i)  }{  g_i (r_i)  } \right)
\sum_{d_1, \dots , d_k} \frac{ \lambda_{d_1r_1, \dots , d_kr_k}}{ \prod_{i=1}^k g_i (d_i) } \nonumber \\
&= \left(  \prod_{i=1}^k   \frac{ \mu (r_i)  \varphi (r_i)  }{   g_i (r_i) } \right)  \sum_{d_1, \dots , d_k} 
\left(  \prod_{i=1}^{k}  \frac{\mu (d_i r_i) d_ir_i}{g_i (d_i) } \right)  \sum_{t_1, \dots , t_k} \frac{ y_{t_1d_1r_1 , \dots , t_kd_kr_k} }{\prod_{i=1}^k \varphi (t_id_ir_i) }  \\
&=  \left( \prod_{i=1}^k \frac{   \mu  (r_i)^2  \varphi (r_i) r_i  }{  g_i (r_i) }  \right) 
 \sum_{m_1, \dots , m_k}  \frac{ y_{m_1r_1 , \dots , m_kr_k} }{ \left(  \prod_{i=1}^k \varphi (m_ir_i) \right) } \sum_{\substack{ d_1, \dots , d_k \\ \forall i ~d_i|m_i }} 
 \prod_{i=1}^{k} \frac{  \mu (d_i) d_i}{  g_i (d_i) } . \nonumber
\end{align}
For $i=1, \dots , k$ and square-free $m$ we consider the sum 
\[ \sum_{d|m} \frac{ \mu (d) d }{ g_i (d)} = \sum_{d|m} \mu (d) (d, \delta_i ). \]
We split $d=d'd''$, where $(d',\delta_i)=1$ and $d''|\delta_i$. Then the sum above equals
\begin{equation}\label{2.13}
\sum_{\substack{ d'|m \\ (d',\delta_i)=1 }} \mu( d' ) \sum_{\substack{d''|m \\ d''|\delta_i }} \mu(d'') (d'', \delta_i) 
\end{equation}
The first factor of (\ref{2.13}) equals
\begin{equation}
\sum_{\substack{ d'|m \\ (d',\delta_i)=1 }} \mu( d' ) = \prod_p \left( 1 - \mathbf{1}_{p \nmid \delta_i } \right) = \mathbf{1}_{m|\delta_i}.
\end{equation}
The second factor of (\ref{2.13}) equals
\begin{equation}
 \sum_{\substack{d''|m \\ d''|\delta_i }} \mu(d'') (d'', \delta_i)  = \prod_{p|(m,\delta_i)} \left( 1 - p \right) = \mu \left( (m, \delta_i) \right) \varphi  \left( (m, \delta_i) \right).
\end{equation}
Therefore, we find that
\begin{equation}\label{formula_w}
w_{r_1, \dots , r_k} = 
\left( \prod_{i=1}^k  (r_i, \delta_i) \right)
 \sum_{\substack{ m_1, \dots , m_k \\ \forall i ~ m_i|\delta_i }}  y_{m_1r_1 , \dots , m_kr_k}
\prod_{i=1}^k  \mu \left( m_i \right) .
\end{equation}

Recall that the number of prime factors of $\delta_1 \cdots \delta_k$ is bounded. In the case when $s_{l,m} > D_0$ for some pair $l \not=m$ the contribution of (\ref{2.10})  is
\begin{multline}\label{total_error_lemma2}
\ll (y_{\max })^2 \left( \prod_{i=1}^k \sum_{\substack{ u<R \\ (u,\delta_i W)=1}} \frac{\mu (u)^2}{\varphi (u)} \right)
\left( \sum_{ \substack{ s_{l,m} > D_0 \\ (s_{l,m},W)=1}}  \frac{\mu (s_{l,m} )^2 (s_{l,m} , \delta_l) (s_{l,m}, \delta_m) }{\varphi (s_{l,m} )^2 }   \right) \\
\times ~ \left( \prod_{\substack{ i,j=1 \\ i \not=j \\ (i,j) \not= (l,m) }}^k
 \sum_{ s=1 }^\infty  \frac{\mu (s )^2 (s , \delta_i) (s, \delta_j)}{\varphi (s )^2 }   \right) ,
\end{multline}
where we used identity (\ref{formula_w}) and $(a_i, \delta_i) = \prod_{i \not= j} ( s_{j,i}, \delta_i)$ and an analogous one for the $b_i$. By factoring $s=s's''$, where $(s',\delta_i)=1$ and $s''|\delta_i$, we find that
\begin{equation}\label{2.18}
 \sum_{ s=1 }^\infty  \frac{\mu (s )^2 (s , \delta_i) (s, \delta_j)}{\varphi (s )^2 }    = 
 \sum_{s''|\delta_i} \frac{s'' \mu (s'' )^2 (s'', \delta_j) }{\varphi (s'' )^2 }   
  \sum_{ \substack{ s'=1 \\ (s',\delta_i)=1 } }^\infty  \frac{\mu (s' )^2  (s', \delta_j)}{\varphi (s' )^2 }.
\end{equation}
A trivial estimation gives
\begin{equation}\label{2.19}
 \sum_{s''|\delta_i}  \frac{s'' \mu (s'' )^2 (s'', \delta_j) }{\varphi (s'' )^2 }    \leq \prod_{p|\delta_i}  \left( 1 + \frac{p^2}{(p-1)^2} \right) \ll 1.
\end{equation}
We repeat this procedure for $(s', \delta_j)$ in order to get
\begin{equation}\label{2.20}
 \sum_{ s=1 }^\infty  \frac{\mu (s )^2 (s , \delta_i) (s, \delta_j)}{\varphi (s )^2 } \ll 1.
\end{equation}
By the same method, we also have
\begin{multline}\label{2.21}
 \sum_{ \substack{ s > D_0 \\ (s ,W)=1}}  \frac{\mu (s )^2 (s , \delta_l) (s , \delta_m) }{\varphi (s )^2 }  =
  \sum_{ s'' | \delta_l}  \frac{\mu (s'' )^2 s'' (s'' , \delta_m) }{\varphi (s'' )^2 } 
   \sum_{ \substack{ s' \\ s's'' > D_0 \\ (s' ,\delta_l W)=1}}  \frac{\mu (s' )^2 (s' , \delta_m) }{\varphi (s' )^2 }  \\
\ll   \sum_{ \substack{ s \\ s > D_0 \\ (s ,\delta_l W)=1}}  \frac{\mu (s )^2 (s , \delta_m) }{\varphi (s )^2 } 
+ \frac{1}{D_0} \sum_{s=1}^\infty  \frac{\mu (s )^2 (s , \delta_m) }{\varphi (s )^2},
\end{multline}
where we used the fact that
\begin{equation}\label{2.22}
 \sum_{ \substack{ s | \delta_l \\ s \not=1 }}  \frac{\mu (s )^2 s }{\varphi (s )^2 } = \left( \prod_{p| \delta_l} \left( 1 + \frac{p}{(p-1)^2} \right) \right) -1 \ll \left( 1 + O \left( \frac{1}{D_0} \right) \right)^{\omega ( \delta_l )} -1 \ll \frac{1}{D_0}.
\end{equation}
We repeat the splitting procedure for $s'$ and combining the result with (\ref{2.21}) we obtain
\begin{equation}\label{2.23}
 \sum_{ \substack{ s_{l,m} > D_0 \\ (s_{l,m},W)=1}}  \frac{\mu (s_{l,m} )^2 (s_{l,m} , \delta_l) (s_{l,m}, \delta_m) }{\varphi (s_{l,m} )^2 }  \ll \frac{1}{D_0}.
\end{equation}

Combining (\ref{2.18})--(\ref{2.23}) we get that the total error term coming from (\ref{total_error_lemma2}) is 
\begin{equation}
\ll 
O \left( \frac{y^2_{\max}\varphi (W)^k (\log R)^k }{ W^k D_0 } \right).
\qedhere
\end{equation}
\end{proof}

We define a function which is going to be useful in the next lemma.
 \begin{equation}\label{opis_w2}
  w^{(\ell )}_{r_1, \dots , r_k} = \left( \prod_{i=1}^k  \mu (r_i) \right) \left( \prod_{\substack{i=1 \\ i \not= \ell}}^k \psi (r_i) \right) \sum_{\substack{ d_1, \dots , d_k \\\forall i~ r_i|d_i \\ d_\ell | \delta }}
  \frac{\lambda_{d_1, \dots , d_k} }{ \prod_{i \not= \ell} \varphi \left( d_i\right)  
 } ,
 \end{equation}
where $\psi$ is the totally multiplicative function definded by $\psi (p)= p-2$ for all primes. Put $ w^{(\ell )}_{\max }  = \sup_{r_1, \dots , r_k} | w_{r_1,\dots,r_k}^{(\ell )}| $. The function $w^{(\ell )}_{r_1, \dots , r_k}$ can be seen as a generalization of  $y^{(\ell )}_{r_1, \dots , r_k}$ from \cite{Maynard}.

We now state a lemma, similar to the previous one, referring to $T_{\delta}^{(\ell)}$ (cf. (\ref{2.2})).

\begin{lemma}\label{Lemma2.3}
For $\delta \in \mathbf{N}$ satisfying $\omega ( \delta ) \ll 1$ we have
\[ T_{\delta}^{(\ell)} =  \sum_{ \substack{ u_1, \dots , u_k }} \frac{ (w^{(\ell )}_{u_1, \dots , u_k})^2 }{ \prod_{\substack{ i\not= \ell}}^k \psi (u_i) } + O \left( \frac{ ( w^{(\ell )}_{\max } )^2\varphi (W)^{k-1} \log^{k-1} R }{ W^{k-1} D_0 } \right)\]
\end{lemma}

\begin{proof}
We proceed very much like in the proof of the previous lemma. Again, we wish to drop the dependencies between the $d_i$ and the $e_j$ variables. This time we proceed by using the following identity valid for square-free numbers:
\begin{equation}
\frac{1}{\varphi ([d_i, e_i])} = \frac{1}{\varphi (d_i) \varphi (e_i)} \sum_{u_i | d_i, e_i} \psi (u_i).
\end{equation}
Therefore,

\begin{equation}
 T_{\delta}^{(\ell )} =
 \sum_{ \substack{ u_1, \dots u_{\ell -1}, u_{\ell +1}, \dots , u_k }} \left( \prod_{\substack{i=1 \\ i\not= \ell}}^k \psi (u_i) \right)
 \sideset{}{'}\sum_{\substack{d_1, \dots , d_k \\ e_1, \dots , e_k \\ d_\ell , e_\ell|\delta \\ \forall i\not=\ell ~ u_i|d_i,e_i }} \frac{\lambda_{d_1, \dots , d_k} \lambda_{e_1, \dots , e_k}}{ \left( \prod_{i \not= \ell} \varphi \left( d_i\right)  
  \right) \left( \prod_{i \not= \ell}  \varphi \left( e_i\right) \right) } 
 \end{equation}
In order to slightly simplify our notation we will introduce an extra variable $u_\ell$ under the exterior sum and we make it equal 1. We rewrite the conditions $(d_i, e_j)=1$ by multiplying the expression by $
\sum_{\substack{ s_{i,j}|d_i,e_j}} \mu (s_{i,j})  $ for each $i\not= j$. We also notice that if $i$ or $j$ equals $\ell$, then $s_{i,j}|\delta$. We again impose the same conditions on $s_{i,j}$ as in Lemma \ref{Lemma2.2}. Hence,
\begin{equation}\label{cos_tam}
T_{\delta}^{(\ell )} =
 \sum_{ \substack{ u_1, \dots , u_k \\ u_\ell =1}} \left( \prod_{\substack{i=1}}^k \psi (u_i)\right)
\sideset{}{^{*}}\sum_{\substack{ s_{1,2} , \dots s_{k,k-1}}} \left(  \prod_{\substack{ 1 \leq i , j \leq k \\ i \not= j }}\mu (s_{i,j}) \right)
   \sum_{\substack{d_1, \dots , d_k \\ e_1, \dots , e_k \\ d_\ell , e_\ell|\delta \\ \forall i ~ u_i|d_i,e_i \\ \forall i\not= j ~ s_{i,j}| d_i, e_j}} \frac{\lambda_{d_1, \dots , d_k} \lambda_{e_1, \dots , e_k}}{ \left( \prod_{i \not= \ell} \varphi \left( d_i\right)  
  \right) \left( \prod_{i \not= \ell}  \varphi \left( e_i\right) \right) }  
 \end{equation}
Transforming (\ref{opis_w2}) a bit, we obtain
 \begin{equation}\label{2.30}
  \sum_{\substack{d_1, \dots , d_k \\d_\ell |\delta \\ \forall i ~ u_i|d_i \\ \forall i\not= j ~ s_{i,j}| d_i }} \frac{\lambda_{d_1, \dots , d_k} }{  \prod_{i \not= \ell} \varphi \left( d_i\right)}  = \frac{w_{a_1, \dots , a_k}^{( \ell )} }{ \prod_{i \not= \ell } \psi (a_i) } \prod_{i=1}^k \mu (a_i)
 \end{equation}
and a similar equation for the $b_i$ with $a_i$ and $b_i$ defined as in Lemma \ref{Lemma2.2}. Substituing this into (\ref{cos_tam}), we obtain
 \begin{equation}\label{2.31}
 T_{\delta}^{(\ell )} =
 \sum_{ \substack{ u_1, \dots , u_k \\ u_\ell = 1 }} \left( \prod_{\substack{i=1 }}^k \frac{\mu (u_i)^2 }{\psi (u_i) }\right) 
\sideset{}{^{*}}\sum_{\substack{ s_{1,2} , \dots s_{k,k-1}}} \left(  \prod_{\substack{ 1 \leq i , j \leq k \\ i \not= j \\i,j\not=\ell }} \frac{\mu (s_{i,j})}{\psi (s_{i,j})^2 } \right) \left( 
 \prod_{\substack{ 1 \leq i , j \leq k \\ i \not= j \\ i=\ell \text{ or } j=\ell}} \frac{\mu (s_{i,j})}{\psi (s_{i,j}) }
  \right)w^{(\ell )}_{a_1, \dots , a_k} w^{(\ell )}_{b_1, \dots , b_k},
 \end{equation}

 We can see that the contribution coming from $s_{i,j} \not= 1$ for $i,j\not=\ell$ is
 \begin{multline}\label{2.32}
 \ll   ( w^{(\ell )}_{\max } )^2   \left( \sum_{ \substack{ u <R \\ (u,W)=1 }}  \frac{\mu  (u)^2 }{\psi (u) }\right)^{k-1}
 \left( \sum_{ s_{i,j} > D_0 }  \frac{\mu (s_{i,j} )^2}{\psi (s_{i,j} )^2 }   \right)
 \left( \sum_{ s=1 }^\infty  \frac{\mu (s )^2}{\psi (s )^2 }   \right)^{k^2-3k+1} \left( \sum_{s|\delta} \frac{\mu (s)^2}{\psi (s)} \right)^{2k-2} \\
\ll  \frac{ ( w^{(\ell )}_{\max } )^2 \varphi (W)^{k-1} \log^{k-1} R }{ W^{k-1} D_0 }.
 \end{multline}
 On the other hand, if $i$ or $j$ equals $\ell$ then the contribution from $s_{i,j}$ is
 \begin{multline}\label{2.33}
  \ll   ( w^{(\ell )}_{\max } )^2   \left( \sum_{ \substack{ u <R \\ (u,W)=1 }}  \frac{\mu (u)^2 }{\psi (u) }\right)^{k-1}
 \left( \sum_{\substack{ s_{i,j}|\delta \\ s_{i,j} > D_0 }}  \frac{\mu (s_{i,j} )^2}{\psi (s_{i,j} ) }   \right)
 \left( \sum_{ s=1 }^\infty  \frac{\mu (s )^2}{\psi (s )^2 }   \right)^{k^2-3k+2} \left( \sum_{s|\delta} \frac{\mu (s)^2}{\psi (s)} \right)^{2k-3} \\
\ll \frac{ ( w^{(\ell )}_{\max } )^2 \varphi (W)^{k-1} \log^{k-1} R }{ W^{k-1} D_0 }. 
 \end{multline} 
 
 Thus, the main term equals
 \begin{equation}\label{2.34}
 \sum_{ \substack{ u_1, \dots , u_k \\ u_\ell =1 }} \frac{ (w^{(\ell )}_{u_1, \dots , u_k})^2 }{ \prod_{\substack{ i\not= \ell}}^k \psi (u_i) } . 
 \qedhere
 \end{equation}
 \end{proof}
 
 By (\ref{2.34}) we need to calculate $w^{(\ell )}_{r_1, \dots , r_k}$ in terms of $y_{r_1, \dots , r_k}$ only when  the $\ell$-th coordinate of $w^{(\ell )}_{r_1, \dots , r_k}$ equals $1$. However, we can just as easily deal with a slightly more general situation.

 \begin{lemma}\label{Lemma2.4}
  For $r_1 , \dots, r_k \in \mathbf{N}$ and $\delta$ satisfying $\omega ( \delta  ) \ll 1$, we have
 \begin{multline*}
 w^{(\ell )}_{r_1, \dots , r_k} = \mu ( r_\ell )
\left(   \sum_{\substack{ r_\ell | s  | \delta }}  \mu (s)  \sum_{\substack{ a \\ (a,\delta)=1 }} 
   \frac{y_{r_1, \dots , r_{\ell -1} , as , r_{\ell +1} , \dots , r_k}}{ \varphi (a) } \right) \left(1 + O \left( \frac{1}{D_0} \right) \right) \\ + ~   O \left( \frac{y_{\max } \varphi(W)  \log R }{WD_0} \right).
\end{multline*}
 \end{lemma}

 \begin{proof}
 We wish to calculate $ w^{(\ell )}_{r_1, \dots , r_k}$ in the terms of $y_{r_1,\dots , r_k}$.  Notice that by the definition of $ w^{(\ell )}_{r_1, \dots , r_k}$ it equals 0 if $r_\ell \nmid \delta$, so we will further assume that $r_\ell | \delta$. We obtain
 \begin{align}\label{liczymy_wl}
 w^{(\ell )}_{r_1, \dots , r_k} &= \left( \prod_{i=1}^k \mu (r_i) \right) \left( \prod_{\substack{i=1 \\ i \not= \ell }}^k \psi (r_i) \right)
 \sum_{\substack{d_1, \dots , d_k \\ \forall i~ r_i | d_i \\ d_\ell | \delta }} \frac{\lambda_{d_1,\dots ,d_k}}{\prod_{i \not= \ell } \varphi (d_i)}  \\
&= \left( \prod_{i=1}^k \mu (r_i) \right) \left( \prod_{\substack{i=1 \\ i \not= \ell }}^k \psi (r_i) \right) 
  \sum_{\substack{a_1, \dots , a_k }} \frac{y_{a_1, \dots , a_k}}{\prod_{i =1 }^k   \varphi (a_i) }
  \sum_{\substack{d_1, \dots , d_k \\ \forall i~ r_i | d_i \\ \forall i~d_i|a_i\\ d_\ell | \delta }} \left( \prod_{ \substack{ i=1 \\ i\not= \ell}}^k  \frac{ \mu (d_i) d_i }{\varphi (d_i)} \right) \mu (d_\ell ) d_\ell ,
 \end{align}
 where the $a_i$ are new variables and are not related to these from equation (\ref{2.31}). We have
 \begin{equation}\label{liczymy_glupiasume}
   \sum_{\substack{d_1, \dots d_{\ell -1}, d_{\ell +1}, \dots , d_k \\ \forall i~ r_i | d_i \\ \forall i~d_i|a_i }} \prod_{ \substack{ i=1 \\ i\not= \ell}}^k  \frac{ \mu (d_i) d_i }{\varphi (d_i)} =  \mathbf{1}_{ \{ \forall i\not= \ell ~r_i | a_i \}}\prod_{\substack{i=1 \\ i \not= \ell }}^k \frac{ \mu (a_i) r_i}{\varphi( a_i)},
 \end{equation}
 and 
 \begin{equation}\label{liczymy_al,delta}
  \sum_{\substack{ d_\ell | (a_\ell , \delta) \\  r_\ell | d_\ell \\  }} \mu (d_\ell )  d_\ell   = \mathbf{1}_{r_\ell | (a_\ell , \delta)}
  \mu(r_\ell ) r_\ell \prod_{p| (a_\ell , \delta)/r_\ell } (1-p) = \mathbf{1}_{r_\ell | a_\ell} \frac{r_\ell}{\varphi (r_\ell )} \mu ( (a_\ell, \delta)) \varphi ((a_\ell , \delta)),
 \end{equation}
 where we used the fact that the $a_i$ are square-free. We put (\ref{liczymy_glupiasume})--(\ref{liczymy_al,delta}) into (\ref{liczymy_wl}) and get
 \begin{equation}
  w^{(\ell )}_{r_1, \dots , r_k} = \left( \prod_{i=1}^k \mu (r_i) \right) \left( \prod_{\substack{i=1 \\ i \not= \ell }}^k \psi (r_i) r_i \right)
    \frac{r_\ell}{\varphi (r_\ell )}  \sum_{\substack{a_1, \dots , a_k \\ \forall i~ r_i | a_i  }} \frac{y_{a_1, \dots , a_k}}{\prod_{i =1 }^k   \varphi (a_i) } \left( \prod_{\substack{i=1 \\ i \not= \ell }}^k \frac{ \mu (a_i) }{\varphi( a_i)} \right) \mu ( (a_\ell, \delta)) \varphi ((a_\ell , \delta)).
\end{equation}
We see that $y_{a_1, \dots , a_k}$ is supported only on the $a_i$ satisfying $(a_i, W)=1$, so either $a_j=r_j$ for all $j \not= \ell$ or $a_j > D_0r_j$ for at least one of them. The total contribution of the latter case is 
\begin{align}
\ll& \, y_{\max }  \left( \prod_{\substack{i=1 \\ i \not= \ell }}^k \psi (r_i) r_i \right) \frac{r_\ell }{\varphi (r_\ell)}
\left( \sum_{\substack{ a_j > D_0 r_j  \\ r_j | a_j }} \frac{ \mu(a_j)^2 }{\varphi (a_j)^2}\right)
\left( \sum_{\substack{ a_\ell <R\\ (a_\ell , W)=1 \\ r_\ell | a_\ell }} \frac{ \mu ( a_\ell)^2 \varphi ((a_\ell , \delta))}{\varphi (a_\ell )} \right)
\prod_{\substack{ i=1 \\ i\not= j,\ell}}^k \left( \sum_{r_i | a_i} \frac{\mu (a_i)^2 }{\varphi (a_i)^2 }\right)  \nonumber \\
\ll& \, \frac{\tau (\delta) r_\ell  }{\varphi (r_\ell)} \left( \prod_{\substack{i=1 \\ i \not= \ell }}^k \frac{\psi (r_i) r_i}{\varphi (r_i)^2} \right) \frac{y_{\max } \varphi(W) \log R }{WD_0} \ll  \frac{y_{\max } \varphi(W)  \log R }{WD_0}, 
\end{align}
where in the summation over $a_{\ell}$ we factorized $a=a's$, where $s|\delta$ and $(a',\delta)=1$, to obtain
\begin{equation}
 \sum_{\substack{ a_\ell <R\\ (a_\ell , W)=1 \\ r_\ell | a_\ell }} \frac{ \mu (a_\ell )^2 \varphi ((a_\ell , \delta))}{\varphi (a_\ell )} =
 \sum_{\substack{  (s , W)=1 \\ r_\ell | s | \delta }} \frac{ \mu ( s)^2 \varphi (s)}{\varphi (s )}  
 \sum_{\substack{ a' <R/s  \\ (a' , \delta W)=1 }} \frac{\mu (a')^2}{\varphi (a' )}  \ll \frac{\tau (\delta ) \varphi (W) \log R}{W}.
\end{equation}

Therefore, we have
\begin{multline}
 w^{(\ell )}_{r_1, \dots , r_k} = 
   \left( \prod_{\substack{i=1 \\ i \not= \ell }}^k \frac{ \psi (r_i) r_i}{\varphi (r_i)^2} \right) 
    \frac{ \mu(r_\ell ) r_\ell}{\varphi (r_\ell )}
    \sum_{\substack{ a \\ r_\ell | a }} \frac{y_{r_1, \dots , r_{\ell -1} , a , r_{\ell +1} , \dots , r_k}}{  \varphi (a) }  \mu ( (a_\ell, \delta)) \varphi ((a_\ell , \delta)) \\\ +  O \left( \frac{y_{\max } \varphi(W)  \log R }{WD_0} \right).
\end{multline}
All the $r_i$ are coprime to $W$, so we note that 
  \begin{equation}
  1 \geq \prod_{\substack{i=1 \\ i \not= \ell }}^k \frac{ \psi (r_i) r_i}{\varphi (r_i)^2} \geq \prod_{p >D_0} \left( 1 - \frac{1}{p^2 - 2p + 1} \right) \geq \exp \left( - O\left( \frac{1}{D_0} \right)\right) = 1 - O \left( \frac{1}{D_0} \right)
  \end{equation}
  and 
  \begin{equation}
    1 \leq \frac{ r_\ell}{\varphi (r_\ell )} \leq \prod_{p | r_\ell } \left( 1 + \frac{1}{D_0-1} \right) \leq 1 + O\left( \frac{1}{D_0}\right).
  \end{equation}
 Decomposing $a=a's$, where $s| \delta$ and $(a',\delta)=1$, we obtain 
 \begin{equation}
  \sum_{\substack{ a \\  r_\ell | a }} \frac{y_{r_1, \dots , r_{\ell -1} , a , r_{\ell +1} , \dots , r_k}}{  \varphi (a) }  \mu ( (a_\ell, \delta)) \varphi ((a_\ell , \delta)) =
  \sum_{\substack{ s| \delta \\ r_\ell | s }} \mu (s)  \sum_{\substack{ a' \\ (a',\delta)=1 }} 
   \frac{y_{r_1, \dots , r_{\ell -1} , a's , r_{\ell +1} , \dots , r_k}}{  \varphi (a') }.
   \qedhere
 \end{equation} 
 \end{proof}

 In order to convert certain sums into integrals, we use the following lemma. 
 \begin{lemma}\label{suma_multi}
 Let $\kappa , A_1, A_2, L >0$. Let $\gamma$ be a multiplicative function satisfying
 \[ 0 \leq \frac{\gamma (p)}{p} \leq 1 - A_1 \]
 and 
 \[ -L \leq \sum_{w \leq p \leq z} \frac{\gamma (p) \log p}{p} - \kappa \log (z/w) \leq A_2 \]
  for any $2 \leq w \leq z$. Let $g$ be the completely multiplicative function defined on primes by $g(p)= \gamma (p) / (p - \gamma (p))$. Finally, let $G \colon [0,1] \rightarrow \mathbf{R}$ be a piecewise differentiable function and let $G_{\max} = \sup_{t \in [0,1]} (|G(t)| + |G'(t)|)$. Then
 \[ \sum_{d <z} \mu (d)^2 g(d) G \left( \frac{ \log d}{\log z} \right) = \mathfrak{S} \frac{ (\log z)^{\kappa }  }{\Gamma (\kappa )}
 \int \limits_0^1 G(x) x^{\kappa -1}\, dx + O_{A_1,A_2,\kappa } ( \mathfrak{S} L G_{\max } \log^{\kappa -1}z ),
 \]
 where
 \[ \mathfrak{S} = \prod_p \left( 1 - \frac{\gamma (p)}{p} \right)^{-1} \left( 1 - \frac{1}{p} \right)^\kappa . \]
 Here the constant implied by the `$O$' term is independent of $G$ and $L$. 
 \end{lemma}
 
 \begin{proof} This is Lemma 4 from \cite{GYP} with the notation from \cite{Maynard}.
 \end{proof}

For the sake of convenience, we denote the main terms from Lemmata \ref{Lemma2.2}, \ref{Lemma2.3}, and \ref{Lemma2.4} as follows:
\begin{align}
 \widetilde{T}_{\delta_1, \dots , \delta_k} =&  \sum_{\substack{ u_1, \dots , u_k \\ \forall i~ (u_i , \delta_i)=1 }} \frac{1}{\prod_{i=1}^k \varphi( u_i ) }  \left( \sum_{\substack{ s_1, \dots , s_k \\ \forall i~ s_i | \delta_i }} \mu (s_1) \dots \mu (s_k) y_{u_1s_1, \dots , u_ks_k} \right)^2, \label{Ttilde} \\
 \widetilde{T}_{\delta}^{(\ell)} =&  \sum_{ \substack{ r_1, \dots , r_k }} \frac{ (w^{(\ell )}_{r_1, \dots , r_k})^2 }{ \prod_{\substack{ i\not= \ell}}^k \psi (r_i) }, \label{T*tilde} \\
\widetilde{w}^{(\ell )}_{r_1, \dots , r_k} =& 
  \sum_{\substack{ r_\ell | s  | \delta }}  \mu (s)  \sum_{\substack{ u \\ (u,\delta)=1 }} 
   \frac{y_{r_1, \dots , r_{\ell -1} , us , r_{\ell +1} , \dots , r_k}}{   \varphi (u) } . \label{wtilde}
\end{align}

\begin{lemma}\label{Ttilde}
 Fix a positive integer $r$. Let $\delta = p_1 \cdots  p_{r-1}$ for some distinct primes $p_1 , \dots , p_{r-1}$ greater than $D_0$. Then we have
\[  \widetilde{T}_{1, \dots , 1, \delta , 1, \dots , 1} =  \frac{ \varphi (W)^k (\log R)^{k} }{W^k} I_0 \left( \frac{ \log p_1}{\log R}, \dots , \frac{\log p_{r-1}}{\log R} \right) + O \left( \frac{ F_{\max }^2 \varphi (W)^k (\log R)^{k} }{W^k D_0} \right),\]
where $\delta$ appears on the $m$-th coordinate and 
\begin{multline*} I_0 (x_1, \dots , x_{r-1} ) \\
= \int \limits_0^1 \dots \int \limits_0^1 \left( \sum_{J \subset \{1, \dots , r-1 \}} (-1)^{|J|} F \left( t_1, \dots , t_{m-1} , 
t_m + \sum_{j \in J} x_j , t_{m+1}, \dots , t_k \right) \right)^2\, dt_1 \, \dots \, dt_k.
\end{multline*}
\end{lemma}
The expression from (\ref{Ttilde}) simplifies to 
\begin{equation}
 \widetilde{T}_{1, \dots , 1, \delta , 1, \dots , 1} =   \sum_{\substack{ u_1, \dots , u_k \\   (u_m , \delta)=1 }} \frac{1}{\prod_{i=1}^k \varphi( u_i ) }  \left( \sum_{\substack{ s| \delta }} \mu (s) y_{u_1, \dots , u_{m-1}, su_m, u_{m+1}, \dots , u_ks_k} \right)^2.
\end{equation}

\begin{proof}

In order to replace the weights $y$ by a function $F$ we need the following conditions to be satisfied: $\mu( sU)^2=1$ and $(sU,W)=1$, where $U=\prod_{i=1}^k u_i$. To achieve such a situation, we wish to impose the condition $(U,\delta)=1$. We can do this at the cost of the error of size
\begin{equation}
\ll F_{\max }^2 \sum_{p| \delta} 
  \sum_{\substack{ u_1, \dots , u_k <R \\  (U, W)=1 \\ p|u_l }} \prod_{i=1}^k \frac{\mu (u_i)^2}{ \varphi( u_i ) }  \ll 
  F_{\max }^2 \sum_{p| \delta} \frac{1}{p-1}\left( \sum_{ \substack{ u<R \\ (u,W)=1}}  \frac{\mu (u)^2}{ \varphi ( u )}  \right)^k \ll \frac{ F_{\max }^2 \varphi (W)^k (\log R)^{k} }{W^k D_0}.
\end{equation}
On the other hand, the support of $y_{u_1, \dots , u_r}$ forces $(u_i , u_j)=1$ for each $i \not= j$. Note that if $(u_i, u_j) \not=1$, then the common factor of $u_i$ and $u_j$ has to be greather than $D_0$. Thus, the discussed requirement can be dropped at the cost of 
\begin{multline}
\ll  F_{\max }^2 \sum_{p > D_0} 
  \sum_{\substack{ u_1, \dots , u_k <R \\  (U, W)=1 \\ p|u_l, u_h }} \prod_{i=1}^k \frac{\mu (u_i)^2}{ \varphi( u_i ) }  \ll
   F_{\max }^2 \sum_{p > D_0}  \frac{1}{(p-1)^2} \left( \sum_{ \substack{ u<R \\ (u,W)=1}}  \frac{\mu (u)^2}{ \varphi ( u )}  \right)^k \\ 
\ll    \frac{ F_{\max }^2 \varphi (W)^k (\log R)^{k} }{W^k D_0}. 
   \end{multline}   
Hence, we can rewrite $ \widetilde{T}_{1, \dots , 1, \delta , 1, \dots , 1}$ as
\begin{multline}\label{suma_Ttilde}
  \sum_{\substack{ u_1, \dots , u_k \\   (U , \delta W)=1 }} \prod_{i=1}^k \frac{\mu (u_i)^2}{ \varphi( u_i ) }  \left( \sum_{\substack{ s| \delta }} \mu (s)
 F \left( \frac{\log u_1}{\log R} , \dots , \frac{\log u_{m-1}}{\log R} ,\frac{\log su_m}{\log R} , \frac{\log u_{m+1}}{\log R} , \dots , \frac{\log u_k}{\log R}  \right)
  \right)^2 \\
 + O \left( \frac{ F_{\max }^2 \varphi (W)^k (\log R)^{k} }{W^k D_0} \right).
\end{multline}
 For the sake of convenience, we put
 \begin{equation}
F^{(m)}_{\delta }  \left( t_1 , \dots , t_k \right)
=   \left( \sum_{\substack{ s| \delta }} \mu (s)
 F \left( t_1 , \dots , t_{m-1}  , t_m + \frac{\log s}{\log R} ,t_{m+1} , \dots ,t_k  \right) \right)^2 . 
  \end{equation}
We apply Lemma $\ref{suma_multi}$ to the sum in (\ref{suma_Ttilde}) with $ \kappa =1$ and 
\begin{align}
\gamma (p) &= 
\begin{cases}
       1, &p \nmid \delta W, \\
        0, & \text{otherwise,}
        \end{cases}
\\
L &\ll 1 + \sum_{p | \delta W} \frac{ \log p }{p} \ll \log D_0,
\end{align}
and with $A_1, A_2$ suitable and fixed. Dealing with the sum over each $u_i$ in turn gives
\begin{multline}
\sum_{\substack{ u_1, \dots , u_k \\   (U , \delta W)=1 }} \prod_{i=1}^k \frac{\mu (u_i)^2}{ \varphi( u_i ) } 
F^{(m)}_{\delta }  \left( \frac{\log u_1}{\log R} , \dots , \frac{\log u_k}{\log R}  \right)   \\
=  \frac{ \varphi (W)^k (\log R)^{k} }{W^k}
  I_0 \left( \frac{ \log p_1}{\log R}, \dots , \frac{\log p_{r-1}}{\log R} \right)
  + O \left(  \frac{ F_{\max }^2 \varphi (W)^k (\log R)^{k}  }{W^k D_0 }
 \right),
\end{multline}
 where we used the fact that
 \begin{equation}
 1 \geq \frac{\varphi (\delta )}{\delta } = \prod_{p | \delta } \left( 1 - \frac{1}{p} \right) \geq \left( 1 - \frac{1}{D_0} \right)^{\omega (\delta )} \geq 1 - O \left( \frac{1}{D_0} \right). 
 \qedhere
 \end{equation}
\end{proof}

\begin{lemma}\label{lemma_wtilde} For a square-free $\delta$ coprime to $W$ and pairwise coprime square-free $r_1 , \dots , r_k$ such that $\left( \prod_{i =1}^k  r_i , W\right) =1$, $r_\ell | \delta$ and $(\prod_{i \not= \ell } r_i , \delta ) = 1$ we have
\begin{equation*}\label{wtilde}
\widetilde{w}^{(\ell )}_{r_1, \dots , r_k} =  \frac{ \varphi (W)}{  W} \prod_{\substack{i=1 \\ i \not= \ell}}^k  \left( \frac{\varphi (r_i)}{r_i}\right) (\log R)  I_{\delta; r_1,\dots, r_k}^{(\ell)} 
+ O \left( \frac{F_{\max} \varphi(W) \log R }{WD_0} \right),
\end{equation*}
where 
\[  I_{\delta; r_1,\dots, r_k}^{(\ell)} = \int \limits_0^1  \sum_{\substack{ r_\ell | s  | \delta }}  \mu (s) 
F \left( \frac{\log r_1}{\log R} , \dots , \frac{\log r_{\ell -1}}{\log R}, t_\ell + \frac{\log s}{\log R} , \frac{\log r_{\ell +1}}{\log R} , \dots , \frac{\log r_k}{\log R} \right)\, dt_\ell.\]
\end{lemma}

\begin{proof}
Again, we wish to substitute a weight $y_{r_1, \dots , r_k}$ by a function $F$. We are allowed to do so under the conditions $\mu \left( us \prod_{i \not = \ell} r_i  \right)^2 = 1$ and $(u,W)=1$, otherwise the weight vanishes. We have
\begin{equation} \label{wtilde}
\widetilde{w}^{(\ell )}_{r_1, \dots , r_k} = 
 \sum_{\substack{ u \\ (u,\delta W \prod_{i \not= \ell} r_i)=1 }} \frac{ \mu (u)^2}{\varphi (u)} \sum_{\substack{ r_\ell | s  | \delta }}  \mu (s)  
F \left( \frac{\log r_1}{\log R} , \dots , \frac{\log r_{m-1}}{\log R} ,\frac{\log su}{\log R} , \frac{\log r_{m+1}}{\log R} , \dots , \frac{\log r_k}{\log R}  \right).
 \end{equation}
We put
\begin{equation}
F^{(\ell)}_{\delta , r_\ell}  \left( t_1 , \dots , t_k \right)
=  \sum_{\substack{ r_\ell | s  | \delta }}  \mu (s) 
 F \left( t_1 , \dots , t_{m-1}  , t_m + \frac{\log s}{\log R} ,t_{m+1} , \dots ,t_k  \right) . 
  \end{equation}
  By Lemma \ref{suma_multi} applied with $\kappa =1$,
\begin{equation}
\gamma (p) = 
\begin{cases}
       1, &p \nmid \delta W  \prod_{i \not= \ell} r_i,\\
        0, & \text{otherwise,}
        \end{cases}
\end{equation}
\begin{equation}
L \ll 1 + \sum_{p | \delta W  \prod_{i \not= \ell} r_i} \frac{ \log p }{p} \ll 
\sum_{p < \log R} \frac{ \log p }{p} + \sum_{\substack{ p | \delta W  \prod_{i \not= \ell} r_i \\ p> \log R}} \frac{\log \log R}{\log R} \ll \log \log N,
\end{equation}
and $A_1,A_2$ suitable and fixed, we get
\begin{multline}\label{wykonczeniowka}
\widetilde{w}^{(\ell ) }_{r_1, \dots , r_k}  = 
 \sum_{\substack{ u \\ (u,\delta W \prod_{i \not= \ell} r_i)=1 }} \frac{ \mu (u)^2}{\varphi (u)} F^{(\ell)}_{\delta , r_\ell}  \left( \frac{\log r_1}{\log R} , \dots , \frac{\log r_{m-1}}{\log R} ,\frac{\log u}{\log R} , \frac{\log r_{m+1}}{\log R} , \dots , \frac{\log r_k}{\log R}  \right)  \\
=  \frac{\varphi (\delta) \varphi (W)}{ \delta W} \prod_{\substack{i=1 \\ i \not= \ell}}^k  \left( \frac{\varphi (r_i)}{r_i}\right) \log R
\int \limits_0^1 F^{(\ell)}_{\delta , r_\ell}  \left( \frac{\log r_1}{\log R} , \dots , \frac{\log r_{\ell -1}}{\log R}, t_\ell  , \frac{\log r_{\ell +1}}{\log R} , \dots , \frac{\log r_k}{\log R} \right)\, dt_\ell \\
+ \, O \left( \frac{F_{\max} \varphi(W) \log \log N }{W} \right).
 \end{multline}
Again, we have to use the inequality $\varphi (\delta ) / \delta  \geq 1 + O(D^{-1})$ to complete the proof.  
\end{proof}

Note that (\ref{wykonczeniowka}) implies that 
\begin{equation}\label{2.61}
w_{\max }^{(\ell )} \ll \frac{ F_{\max } \varphi (W) \log R}{W}
\end{equation}

\begin{lemma}\label{calkaI1} Fix a positive integer $r$. Let $\delta = p_1 \cdots p_{r-1}$ for some distinct primes $p_1 , \dots , p_{r-1}$ greater than $D_0$. Then we have

\[
\widetilde{T}_{\delta}^{(\ell)} = \frac{\varphi (W)^{k+1} (\log R)^{k+1} }{W^{k+1}}I^{(\ell )} \left( \frac{\log p_1}{\log R}, \dots , \frac{\log p_{r-1}}{\log R} \right) + O \left( \frac{ F_{\max}^2 \varphi (W)^{k+1} (\log R)^{k+1}}{W^{k+1} D_0} \right),
\]
where
\begin{align*}
I^{(\ell )} (x_1, \dots , x_{r-1}) &=  \int \limits_0^1 \dots \int \limits_0^1 \widetilde{I}^{(\ell )}  (x_1, \dots , x_{r-1})^2\, dt_1 \, \dots\, dt_{\ell -1}\, dt_{\ell +1} \dots\, dt_k, \\
 \widetilde{I}^{(\ell )}  (x_1, \dots , x_{r-1}) &= \int \limits_0^1 \sum_{  J \subset \{1, \dots , r-1\}}  (-1)^{|J|} F \left( t_1, \dots, t_{\ell -1}, t_\ell + \sum_{j \in J} x_j, t_{\ell +1}, \dots , t_k\right)\, dt_\ell .
\end{align*}

\end{lemma}

\begin{proof}

Recall 
\begin{equation}\label{T*tilde}
 \widetilde{T}_{\delta}^{(\ell)} =  \sum_{ \substack{ r_1, \dots , r_k }} \frac{ (w^{(\ell )}_{r_1, \dots , r_k})^2 }{ \prod_{\substack{ i\not= \ell}}^k \psi (r_i) }.
\end{equation}
We need to impose the condition $( \prod_{i =1}^k r_i , \delta )=1$ under the sum. It creates the error term of the following size
\begin{equation}\label{2.63}
\ll \frac{ F_{\max }^2 \varphi (W)^2 \log^2 R}{W^2}
  \left(  \sum_{p|\delta } \sum_{\substack{t<R \\ (t,W)=1 \\ p|t}} \frac{1 }{\psi (t) } \right)
  \left( \sum_{\substack{ r < R \\ (r,W)=1} } 
\frac{1}{\psi (r)}\right)^{k-2} \ll \frac{ F_{\max}^2 \varphi (W)^{k+1} (\log R)^{k+1}}{W^{k+1} D_0}.
\end{equation}

By Lemma \ref{lemma_wtilde} and (\ref{T*tilde})--(\ref{2.63}) we have 
\begin{equation}
 \widetilde{T}_{\delta}^{(\ell)} = 
 \frac{ \varphi (W)^2 \log^2 R }{  W^2 } \sum_{ \substack{ r_1, \dots , r_k \\ \forall i \not= j ~(r_i,r_j)=1 \\ \forall i ~(r_i, \delta)=1 }} 
   \prod_{\substack{i=1 }}^k  \left( \frac{\varphi (r_i)^2}{ \psi (r_i) r_i^2}\right) \left( {I_{\delta; r_1,\dots, r_k}^{(\ell)}}\right)^2 + 
    O \left( \frac{F_{\max}^2 \varphi(W)^{k+1} (\log R)^{k+1} }{W^{k+1} D_0} \right).
\end{equation}
We can restrict our attention to the sum
\begin{equation}
 \sum_{ \substack{ r_1, \dots , r_k \\ \forall i \not= j ~(r_i,r_j)=1 \\ \forall i  ~(r_i, \delta)=1 }} 
   \prod_{\substack{i=1 }}^k  \left( \frac{\varphi (r_i)^2}{ \psi (r_i) r_i^2}\right)    \left(  I_{\delta; r_1,\dots, r_k}^{(\ell)}\right)^2.
\end{equation}

We wish to drop the condition $(r_i , r_j)=1$ for all $i \not= j$ in the summation in order to use Lemma \ref{lemma_wtilde}. We can do this at the cost of introducing an error which is of size
\begin{equation}
\ll F_{\max }^2  \left( \sum_{p > D_0} \frac{\varphi (p)^4 }{\psi (p)^2 p^4}\right)  \left( \sum_{\substack{ r < R \\ (r,W)=1} } 
\frac{\varphi (r)^2 }{\psi (r) r^2}\right)^{k-1} \ll \frac{ F_{\max}^2 \varphi (W)^{k-1} \log^{k-1} R}{W^{k-1} D_0}.
\end{equation}

We apply Lemma \ref{suma_multi} individually to each summand of
\begin{equation}
\sum_{ \substack{ r_1, \dots , r_k \\ \forall i ~(r_i, \delta)=1  }}  
   \prod_{\substack{i=1 }}^k  \left( \frac{\varphi (r_i)^2}{ \psi (r_i) r_i^2}\right)
 \left(    I_{\delta; r_1,\dots, r_k}^{(\ell)} \right)^2.
\end{equation}
 We always take $\kappa =1$ and 
\begin{align}
\gamma (p) &= 
\begin{cases}
       1 + \frac{p^2 - p +1}{p^3 - 3p^2 + 2p -1}, &p \nmid \delta W, \\
        0, & \text{otherwise,}
        \end{cases} \\
L &\ll 1 + \sum_{p| \delta W} \frac{\log p}{p} \ll \log D_0,
\end{align}
and $A_1, A_2$ being suitable fixed constants. We calculate
\begin{equation}
\widetilde{T}_{\delta}^{(\ell)} = \frac{\varphi (W)^{k+1} (\log R)^{k+1} }{W^{k+1}}I^{(\ell)} \left( \frac{\log p_1}{\log R}, \dots , \frac{\log p_{r-1}}{\log R} \right) + O \left( \frac{ F_{\max}^2 \varphi (W)^{k+1} (\log R)^{k+1} }{W^{k+1} D_0} \right).
\qedhere
\end{equation}
\end{proof}

\begin{lemma}\label{2.9}
 Fix $\epsilon>0$. For $p > D_0$ we have
\[ T_{1, \dots 1 , p  , 1 , \dots , 1} \ll \frac{ F_{\max}^2\varphi (W)^k  (\log p ) (\log R)^{k-1}} {W^k} ,\]
where $p$ appears on the $m$-th coordinate of $T$.
\end{lemma}

\begin{proof} We have
\begin{multline}
T_{1, \dots ,1,p,1, \dots , 1} = \sideset{}{'}\sum_{\substack{d_1, \dots , d_k \\ e_1, \dots , e_k  }} \frac{\lambda_{d_1, \dots , d_k} \lambda_{e_1, \dots , e_k}}{  \frac{ [d_m , e_m, p] }{p} \prod_{i\not= j} [d_i , e_i] } \\
 = \sideset{}{'}\sum_{\substack{d_1, \dots , d_k \\ e_1, \dots , e_k \\ p \nmid d_m,e_m }} \frac{ \left( \lambda_{d_1, \dots , d_k} + \lambda_{d_1, \dots, d_{m-1}, pd_m , d_{m+1}, \dots , d_k} \right) \left(\lambda_{e_1, \dots , e_k}+ 
 \lambda_{e_1, \dots, e_{m-1}, pe_m , e_{m+1}, \dots , e_k}    \right) }{  \prod_{i=1}^k [d_i , e_i] } .
\end{multline}
We put 
\begin{equation}
\breve{\lambda}_{d_1,\dots , d_k} =  \left( \lambda_{d_1, \dots , d_k} + \lambda_{d_1, \dots, d_{m-1}, pd_m , d_{m+1}, \dots , d_k} \right)\mathbf{1}_{p \nmid d_m}
\end{equation}
and get
\begin{equation}
T_{1, \dots ,1,p,1, \dots , 1} =  \sideset{}{'}\sum_{\substack{d_1, \dots , d_k \\ e_1, \dots , e_k  }} \frac{\breve\lambda_{d_1, \dots , d_k} \breve\lambda_{e_1, \dots , e_k}}{  \prod_{i= 1}^k [d_i , e_i] }.
\end{equation}
After sequence of operations analogous to what was presented in Lemma \ref{Lemma2.2} we get

\begin{multline}\label{2.74}
 T_{1, \dots ,1,p,1, \dots , 1} =
 \sideset{}{^{*}}\sum_{\substack{ s_{1,2} , \dots s_{k,k-1}}} \left(  \prod_{\substack{ 1 \leq i , j \leq k \\ i \not= j }} \frac{\mu (s_{i,j})}{ \varphi (s_{i,j})^2 } \right)  \sum_{ \substack{ u_1, \dots , u_k  }} \left(  \prod_{i=1}^k \frac{\mu (u_i)^2}{\varphi (u_i)} \right)\breve{y}_{a_1 , \dots , a_k} \breve{y}_{b_1, \dots , b_k} \\
\ll  F_{\max}
 \sideset{}{^{*}}\sum_{\substack{ s_{1,2} , \dots s_{k,k-1}}} 
  \left(  \prod_{\substack{ 1 \leq i , j \leq k \\ i \not= j }} \frac{1}{ \varphi (s_{i,j})^2 } \right) 
  \sum_{ \substack{ u_1, \dots , u_k  }} \left(  \prod_{i=1}^k \frac{\mu (u_i)^2}{\varphi (u_i)} \right) \left|  \breve{y}_{a_1 , \dots , a_k}  \right|,
 \end{multline}
where the $a_i$, $b_i$ and $\Sigma^{*}$ are the same as in Lemma \ref{Lemma2.2} and 
\begin{align}\label{2.75}
\breve{y}_{r_1,\dots, r_k} &= \left( \prod_{i=1}^k \mu(r_i) \varphi(r_i) \right) \sum_{\substack{ d_1, \dots , d_k \\ \forall i~r_i|d_i }} \frac{ \breve\lambda_{d_1, \dots ,d_k} }{ \prod_{i=1}^k d_i } \\
&= \left( \prod_{i=1}^k \mu(r_i) \varphi(r_i) \right) \sum_{\substack{ d_1, \dots , d_k \\ \forall i~r_i|d_i \\ p \nmid d_m}} \frac{ \lambda_{d_1, \dots ,d_k} }{ \prod_{i=1}^k d_i } +
  p \left( \prod_{i=1}^k \mu(r_i) \varphi(r_i) \right) \sum_{\substack{ d_1, \dots , d_k \\ \forall i\not=m~r_i|d_i \\ pr_m | d_m}} \frac{   \lambda_{d_1, \dots ,d_k}     }{ \prod_{i=1}^k d_i } \nonumber
\\
&=  \left( \prod_{i=1}^k \mu(r_i) \varphi(r_i) \right) \sum_{\substack{ d_1, \dots , d_k \\ \forall i~r_i|d_i }} \frac{ \lambda_{d_1, \dots ,d_k} }{ \prod_{i=1}^k d_i } +
  (p-1) \left( \prod_{i=1}^k \mu(r_i) \varphi(r_i) \right) \sum_{\substack{ d_1, \dots , d_k \\ \forall i\not=m~r_i|d_i \\ pr_m | d_m}} \frac{   \lambda_{d_1, \dots ,d_k}     }{ \prod_{i=1}^k d_i } \nonumber
\\
&=~  y_{r_1, \dots , r_k} - y_{r_1, \dots , r_{m-1}, pr_m , r_{m+1}, \dots , r_k}. \nonumber
\end{align}
From (\ref{opis_y2}) we have (if we assume $\mbox{supp} \, (F) \subset \mathcal{R}_k$, then we perform analogously)
\begin{equation}\label{2.76}
y_{r_1, \dots , r_k} - y_{r_1, \dots , r_{m-1}, pr_m , r_{m+1}, \dots , r_k} \ll
\begin{cases}
 \frac{F_{\max} \log p }{\log R} , & \text{if~} \forall_i ~ r_1\cdots r_k < Rr_i/p ,\\
\\  F_{\max} , & \text{otherwise. }
\end{cases}
\end{equation}

We split the inner sum from (\ref{2.74}) in the following manner:
\begin{equation}
\sum_{ \substack{ u_1, \dots , u_k  }} = \sum_{ \substack{ u_1, \dots , u_k  \\ \forall i~ a_1\cdots a_k < Ra_i / p}}  + \sum_{ \substack{ u_1, \dots , u_k \\  \exists j~ Ra_j / p \leq a_1\cdots a_k < Ra_j } } .
\end{equation}
The contribution of the first sum to the bottom expression in (\ref{2.74}) can be easily estimated to be 
\begin{equation}
 \ll   \frac{ F_{\max}^2 \varphi (W)^k (\log p) (\log R)^{k-1} }{W^k}. 
 \end{equation}
In the second case we fix some index $j$ and obtain that the analogous contribution is
\begin{multline}
\ll F_{\max} \sideset{}{^{*}}\sum_{\substack{ s_{1,2} , \dots s_{k,k-1}}} 
  \left(  \prod_{\substack{ 1 \leq i , j \leq k \\ i \not= j }} \frac{1}{ \varphi (s_{i,j})^2 } \right) 
 \sum_{ \substack{ u_1, \dots , u_k \\  Ra_j / p \leq a_1\cdots a_k < Ra_j  }} \left(  \prod_{i=1}^k \frac{\mu (u_i)^2}{\varphi (u_i)} \right)   \left| \breve{y}_{a_1 , \dots , a_k}  \right| \\
\ll  \frac{ F_{\max}^2 \varphi (W)^k (\log p) (\log R)^{k-1} }{W^k}.   \qedhere  
    \end{multline}

\end{proof}


\section{Proofs of Propositions \ref{5.1}, \ref{5.2}, and \ref{5.3}}

\subsection{Proof of Proposition \ref{5.1}}

\begin{proof}

We decompose $\Sigma_0$ into $k$ sums according to which factor $L_j (n)$ of $\mathcal{P} (n)$ is divisible by $p$. We also notice that $(p,W)=1$. Therefore, for each $j=1,\dots , k$ we have to calculate

\begin{align}\label{3.1}
 \Sigma_0^{(j)} :&= \sum_{\substack {N < n \leq 2N \\ n \equiv \nu_0 \bmod W}}\left(  \sum_{\substack{ p| L_j (n)  \\ D_0 \leq p \leq R_0}} W_0 \left( \frac{\log p}{\log R} \right) \right) \left(   \sum_{\substack{ d_1, \dots , d_k \\ \forall i ~ d_i|L_i(n)}} \lambda_{d_1,\dots, d_k}  \right)^2 
 \\
 &= \sum_{\substack{ D_0 \leq p \leq R_0}} W_0 \left( \frac{\log p}{\log R} \right) 
   \sum_{\substack{ d_1, \dots , d_k \\ e_1, \dots , e_k }} \lambda_{d_1,\dots, d_k}  \lambda_{e_1,\dots ,e_k}  \sum_{\substack {N < n \leq 2N \\ n \equiv \nu_0 \bmod W \\  \forall i ~ [d_i,e_i]|L_i(n) \\  p| L_j(n) }} 1 \nonumber
\\
&= ~ \frac{N}{W} \sum_{\substack{ D_0 \leq p \leq R_0}} W_0 \left( \frac{\log p}{\log R} \right) 
  \sideset{}{'}\sum_{\substack{ d_1, \dots , d_k \\ e_1, \dots , e_k \\ }} 
  \frac{ \lambda_{d_1,\dots, d_k}  \lambda_{e_1,\dots ,e_k}  }{ [d_j,e_j,p]\prod_{i\not= j} [d_i,e_i] }  \nonumber
\\
&+ O \left( \sum_{\substack{ D_0 \leq p \leq R_0}}  \sideset{}{'}\sum_{\substack{ d_1, \dots , d_k \\ e_1, \dots , e_k \\ }}  |\lambda_{d_1,\dots, d_k}  \lambda_{e_1,\dots ,e_k}  | \right). \nonumber
 \end{align}  
We may add the $'$-condition without losing anything, because the sum equals 0 otherwise. It can be explained by the following short argument. If $p|[d_i,e_i], \, [d_j,e_j]$ for some $p,i,j$, then also $p|L_i (n), \, L_j (n)$. Given that, one may also deduce $p|A_iB_j-A_jB_i$, which implies $p<D_0$. Consequently, $(d_i,W)>1$, so the appropriate Selberg weight has to vanish.

  Let us estimate the error term of (\ref{3.1}) first. We notice that the value of the product $d_1 \dots  d_k e_1 \dots e_k$ is never greater than $R^2$. On the other hand, every value $r\leq R^2$ can be obtained in no more than $\tau_{2k}(r)$ different ways. Thus, by Lemma \ref{szacowanieSelbergow}, we get
  \begin{equation}\label{3.2}
   \sum_{\substack{ D_0 \leq p \leq R_0}}  \sideset{}{'}\sum_{\substack{ d_1, \dots , d_k \\ e_1, \dots , e_k \\ }}  |\lambda_{d_1,\dots, d_k}  \lambda_{e_1,\dots ,e_k}  | \ll F_{\max}^2 \log^{2k}N \sum_{r \leq R_0R^2} \tau_{2k} (r) \ll F_{\max}^2 \left( \log^{4k} N \right)  R_0R^2. 
  \end{equation}
  We can assume that $R_0R^2 \ll N \log^{-10k}N$ to obtain a satisfactory bound. 
  
 The main term from (\ref{3.1}) equals
  \begin{equation}\label{3.3}
   \frac{N}{W} \sum_{\substack{ D_0 \leq p \leq R_0}} W_0 \left( \frac{\log p}{\log R} \right) \frac{T_{1,\dots,1,p,1,\dots,1}}{p}.
  \end{equation}
   We decompose the sum above as follows:
   \begin{equation}
    \sum_{\substack{ D_0 \leq p \leq R_0}} = 
     \sum_{\substack{ D_0 \leq p \leq N^\epsilon}} +  \sum_{\substack{ N^\epsilon \leq p \leq R_0}}.
   \end{equation}
   By Lemma \ref{2.9} the contribution of the first sum to (\ref{3.3}) is 
\begin{equation}
\ll   \frac{ F^2_{\max} \varphi (W)^k N (\log R)^{k-1} }{W^{k+1}}  \sum_{\substack{ D_0 < p \leq N^\epsilon }} \frac{\log p}{p} \ll 
 \frac{ \epsilon F^2_{\max} \varphi (W)^k N (\log R)^{k} }{W^{k+1}}.
   \end{equation}
   By Lemma \ref{Ttilde} the contribution of the second sum to (\ref{3.3}) equals
\begin{equation}
   \frac{ \varphi (W)^k N (\log R)^{k} }{W^{k+1}}  \sum_{\substack{ N^{\epsilon} < p \leq R_0}} \frac{1}{p} W_0 \left( \frac{\log p}{\log R} \right)  I_0 \left( \frac{ \log p}{\log R} \right) +O \left(  \frac{ F_{\max }^2 \varphi (W)^k N (\log R)^{k}  |\log \epsilon | }{W^{k+1} D_0 }  \right),
  \end{equation}
where we estimated the error by the inequality
\begin{equation}
\sum_{N^\epsilon < p \leq R_0} \frac{1}{p} \ll \left| \log \epsilon \right| .
\end{equation}
We use summation by parts (for example \cite[Lemma 6.8]{MaynardK}) and get
\begin{equation}\label{3.5}
 \sum_{\substack{ N^\epsilon < p \leq R_0}} \frac{1}{p} W_0 \left( \frac{\log p}{\log R} \right)  I_0 \left( \frac{ \log p}{\log R} \right) 
 = \int \limits_{\epsilon/\vartheta}^{\vartheta_0/\vartheta} \frac{W_0 (u) I_0 (u)}{u} du + O \left( \frac{\log \log N}{\log N} \right).
\end{equation}
Combining (\ref{3.1})--(\ref{3.5}) and rescaling $\epsilon$, we get the result. 
\end{proof}

\subsection{Proof of Proposition \ref{5.2}} 

We use a Bombieri--Vinogradov style lemma for numbers with exactly $r$ prime factors weighted by $W_r$ which is a minor variation of \cite[Lemma 8.1]{MaynardK}.

\begin{lemma}\label{BV} Assume $GEH [ \theta ]$. Let $A \geq 1$, $\epsilon>0$, and $h>0$ be fixed. Let 
\[ {\beta}_r (n) = 
 \begin{cases}
        {W}_r \left(\frac{\log p_1}{\log n}, \dots , \frac{\log p_{r-1}}{\log n} \right) & \text{for } n=p_1\dots p_{r}, \mbox{with } \epsilon < p_1 \leq \dots \leq p_{r},\\
        0 & \text{otherwise}
 \end{cases} \]
 for some piecewise smooth function $W_r \colon [0,1]^{r-1} \rightarrow \mathbf{R}_{\geq 0}$. Put
 \begin{equation}
 \Delta_{\beta, r} (x; q) = \max_{y \leq x} \max_{\substack{a \\ (a,q)=1}} \left| \sum_{\substack{ y < n \leq 2y \\ n \equiv a \bmod q}} \beta_r (n) - \frac{1}{\varphi (q)}  \sum_{\substack{ y < n \leq 2y \\ (n,q)=1 }} \beta_r (n) \right|.
 \end{equation}
 Then, for $Q \ll x^{\theta} $ we have 
 \begin{equation}
 \sum_{q \leq Q} \mu (q)^2 h ^{\omega (q)} \Delta_{\beta , r} (x;q) \ll \left( \max_{[0,1]^{r-1}} |W_r| \right) x \log^{-A}x.
 \end{equation}
\end{lemma}

\begin{proof} Follows from Theorem \ref{GEHwniosek}. The rest works exaclty the same way as in \cite[Lemma 8.1]{MaynardK}. 
\end{proof}

To prove the next lemma we proceed similarily to Thorne \cite{Thorne} with a few minor differences. The main disparity is that we are able to focus on the prime divisors of $L_i (n)$ directly, so we do not have to `search' for them among $[d,e]$ (see (4.3) in \cite{Thorne}). \\
\begin{lemma}
Assume $GEH[2 \vartheta ]$. Fix $\epsilon>0$. We require $W_r$ to be fixed and supported on 
\[ \mathcal{A}_r = \left\{ x \in [0,1]^{r-1} \colon \epsilon < x_1 < \dots < x_{r-1},~ \sum_{i=1}^{r-1} x_i < \min ( 1 -  \vartheta, 1 - x_{r-1} ) \right\}. \] 
For every $U \geq 1$ we have

\begin{multline*}
 \mathcal{S}_j = \sum_{\substack{ N < n \leq 2N \\ n \equiv \nu_0 \bmod W }} \beta_r \left( L_j (n) \right)  \left( \sum_{\substack{d_1, \dots , d_k \\ \forall i ~ d_i | L_i(n) }} \lambda_{d_1,\dots,d_k} \right)^2 \\
=  \frac{ N}{\varphi(W) \log N}   \sum_{\substack{ p_1, \dots ,p_{r-1} \\ N^\epsilon < p_1 < \dots < p_{r-1} \\ q<\min ( N/R, N/{p_{r-1}})  }} \frac{T^{(j)}_q}{q} 
  \alpha \left( \frac{\log p_1}{\log R} , \dots , \frac{\log p_{r-1}}{\log R} \right) \left( 1 + O\left( \frac{1}{\log N} \right) \right) 
+~  O_U\left(\frac{N}{ \log^{U} N} \right), 
\end{multline*}
  where $q$ is the product of primes appearing under the summation and 
\[ \alpha (x_1, \dots , x_{r-1}) = \frac{W_r (\vartheta x_1, \dots \vartheta x_{r-1})}{1 - \vartheta \sum_{i=1}^{r-1} x_i}. \]
\end{lemma}

\begin{proof}
Switching the order of summation, we get
\[ \mathcal{S}_j =     \sideset{}{'}\sum_{\substack{d_1, \dots , d_k \\  e_1, \dots,  e_k}} \lambda_{d_1,\dots,d_k}  \lambda_{e_1,\dots,e_k}
 \sum_{\substack{ N < n \leq 2N \\ n \equiv \nu_0 \bmod W \\ \forall i ~ [d_i,e_i] | L_i(n) }} \beta_r \left( L_j (n) \right). \]
Similarily to (\ref{3.1}), we may add the $'$-condition here freely. In the first step, we apply the trick devised by Thorne \cite{Thorne}, i.e. we decompose the inner sum 
according to how many prime factors of $ L_j(n)$ divide $[d_j, e_j]$:
\begin{equation}
 \mathcal{S}_j = S_0 + \dots + S_{r},
\end{equation}
where
\[ S_h =   \sideset{}{'}\sum_{\substack{d_1, \dots , d_k \\  e_1, \dots,  e_k}} \lambda_{d_1,\dots,d_k}  \lambda_{e_1,\dots,e_k}
\sum_{\substack{ p_1<\dots<p_{r-h} \\ q=[d_j, e_j]   }} \sum_{\substack{ N < n \leq 2N \\  n \equiv \nu_0 \bmod W \\ \forall i\not=j ~ [d_i,e_i] | L_i(n) \\ q|L_j (n) }} \beta_r \left( L_j (n) \right)   \]
We substitute $L_j (n) = m$. Observe that for any $i \not=j$ the condition $[d_i,e_i]|L_i (n)$ is equivalent to $m \equiv (A_i B_j - A_j B_i) A^{-1} \bmod [d_i,e_i]$, where $A^{-1}$ denotes the inverse element in the multiplicative residue class group (which in this case actually exists because $(d_ie_i,W)=1$). Moreover, $ (A_i B_j - A_jB_i)$ is also coprime to $d_ie_i$, so we can combine all these congruences into one: $m \equiv m_0 \bmod \prod_{i\not=j} [d_i,e_i]$, where $m_0$ is coprime to the modulus. We also have to transform the congruence $n \equiv \nu_0 \bmod W$. To do so, we split it into two: $n \equiv \nu_0 \bmod [A_j,W]/A_j$ and $n \equiv \nu_0 \bmod \mbox{rad} \,A_j$, where $\mbox{rad} \,A_j$ denotes the square-free part of $A_j$. The latter congruence is equivalent to $m \equiv A_j \nu_0 + B_j \bmod A_j \text{rad} \,A_j$ and we can see that  it is compatibile with $m \equiv B_j \bmod A_j$, which is forced by our substitution. This gives
\begin{equation}
S_h =   \sideset{}{'}\sum_{\substack{d_1, \dots , d_k \\  e_1, \dots,  e_k}} \lambda_{d_1,\dots,d_k}  \lambda_{e_1,\dots,e_k}
\sum_{\substack{ p_1<\dots<p_{r-h} \\ q=[d_j, e_j]   }} \sum_{\substack{ A_jN +B_j< m \leq 2A_jN+B_j \\ m \equiv m_0 \bmod   \prod_{i\not=j} [d_i,e_i] \\ q|m \\ m \equiv \mu_0 \bmod [A_j,W]/A_j \\ m \equiv A_j \nu_0 + B_j \bmod A_j \text{rad} \,A_j}} \beta_r \left( m \right),
\end{equation}
where $( \mu_0, [A_j,W]/A_j)=1$. We can simplify the situation further by substituting $m=qt$. The inner sum from the equation above satisfies 
\begin{equation}\label{3.10}
\sum_{\substack{ A_jN +B_j< m \leq 2A_jN+B_j \\ m \equiv m_0 \bmod  \prod_{i\not=j} [d_i,e_i] \\ q|m \\  m \equiv \mu_0 \bmod [A_j,W]/A_j  \\  m \equiv A_j \nu_0 + B_j \bmod A_j\text{rad} \,A_j }} \beta_r \left( m \right) = 
\sum_{\substack{ (A_jN +B_j)/q < t \leq (2A_jN+B_j)/q \\ t \equiv t_0 \bmod u}} \beta_r \left( qt \right)
\end{equation}
for some residue $t_0$ coprime to its modulus $u:=  [A_j,W] \,  \text{rad} \,A_j \prod_{i\not=j} [d_i,e_i] $. We also put $\beta_r (qt) = \widetilde{\beta}_h (t)$ for some function $\widetilde{\beta}_h$ that satisfies
\begin{equation}
 \widetilde{\beta}_h (t) = 
 \begin{cases}
        \widetilde{W}_h \left(\frac{\log p_1}{\log t}, \dots , \frac{\log p_{h-1}}{\log t} \right) & \text{for } t=p_1\cdots p_{h}, \mbox{with } \epsilon < p_1 \leq \dots \leq p_{h},\\
        0 & \text{otherwise}
 \end{cases} 
\end{equation}
for some piecewise smooth function $\widetilde{W}_h \colon [0,1]^{h-1} \longrightarrow \mathbf{R}_{\geq 0}$. It can be done because the largest prime divisor of $qt$ does not divide $q$ as long as it is greater than $R$. This is implied by the fact that $W_r$ is supported on $(x_1, \dots , x_{r-1})$ satisfying $\sum_{i=1}^{r-1} x_i < 1 - \vartheta$ together with $n>N$. Notice that $\widetilde{\beta}_h$ is still depending on $q$, so we have to be careful. The sum on the left-hand of the equation (\ref{3.10}) has the form required by Lemma \ref{BV}. We have
\begin{multline}
\sum_{\substack{ (A_jN +B_j)/q < t \leq (2A_jN+B_j)/q \\ t \equiv t_0 \bmod  u }} \beta_r \left( qt \right) ~=~
\sum_{\substack{ A_jN/q < t \leq 2A_jN/q \\ t \equiv t_0 \bmod u }} \widetilde{\beta}_h \left( t \right) +O (1) \\
= ~ \frac{1}{\varphi([A_j,W] \text{rad} \,A_j) \prod_{i\not=j} \varphi ([d_i,e_i])}  \sum_{\substack{ A_jN/q < t \leq 2A_jN/q }}  \widetilde{\beta}_h \left( t \right) + O \left( \Delta_{\widetilde{\beta},h} \left( A_j N ;  u \right) \,+1 \right).
\end{multline}

We write 
\begin{equation}
S_h =  M_h + E_h,
\end{equation}
where 
\begin{equation}
 M_h = \frac{1}{\varphi([A_j,W] \text{rad} \,A_j)} \sideset{}{'}\sum_{\substack{d_1, \dots , d_k \\  e_1, \dots,  e_k}}
 \frac{ \lambda_{d_1,\dots,d_k}  \lambda_{e_1,\dots,e_k}}{ \prod_{i\not=j} \varphi ([d_i,e_i])} \sum_{\substack{ p_1<\dots<p_{r-h} \\ q=[d_j, e_j]   }}   \sum_{\substack{ A_jN/q < t \leq 2A_jN/q }}  \widetilde{\beta}_h \left( t \right).
\end{equation}
By Lemma \ref{BV} we may obtain
\begin{align}\label{E_h_error}
E_h &\ll
\sideset{}{'}\sum_{\substack{d_1, \dots , d_k \\  e_1, \dots,  e_k}}
|  \lambda_{d_1,\dots,d_k}  \lambda_{e_1,\dots,e_k} | \sum_{\substack{ p_1<\dots<p_{r-h} \\ q=[d_j, e_j]   }}  
 \left( \Delta_{\widetilde{\beta},h} \left( A_j N ;  u \right)  +O(1)  \right)\nonumber \\
&\ll F_{\max}^2 \left( \log^{2k}R  \right)
\sum_{d_j, e_j}
 \sum_{\substack{ p_1<\dots<p_{r-h} \\ q=[d_j, e_j]   }}  
\,  \sideset{}{'}\sum_{\substack{d_1, \dots ,d_{j-1}, d_{j+1}, \dots d_k \\  e_1,\dots ,e_{j-1}, e_{j+1}, \dots  e_k \\ [d_1,e_1] \cdots [d_k,e_k] \leq R^2}}
 \left( \Delta_{\widetilde{\beta},h} \left( A_j N ;  u \right)  +O(1)  \right) \\
 &\ll_U  F_{\max}^2 N \left( \log^{-U+O(1)}N \right) \sum_{d,e} 
  \frac{1}{[d,e]} \nonumber \\
 &\ll F_{\max}^2 N \log^{-U+O(1)}N. \nonumber
\end{align}
for any $U \geq 1$. After rescaling $U$ here a bit we are done.

Now, we concentrate on the main term of $S_h$. It is convenient to put back $\widetilde{\beta}_h (t) = \beta_r (qt)$. We find

\begin{multline}\label{qq'}
\sum_{\substack{ p_1<\dots<p_{r-h} \\ q=[d_j, e_j]   }}   \sum_{\substack{ A_jN/q < t \leq 2A_jN/q }}  \beta_r (qt) = 
\sum_{\substack{ p_1<\dots<p_{r-h} \\ q=[d_j, e_j]   }}   \sum_{\substack{ A_jN/q < t \leq 2A_jN/q }}  
\sum_{\substack{ p_1'<\dots < p_h' \\ q'=t } } \beta_r (qt) \\
= \sum_{\substack{ p_1<\dots<p_{r-h} \\ q=[d_j, e_j]   }}    
\sum_{\substack{ p_1'<\dots < p_h' } }  \beta_r (qq') 
\mathbf{1}_{A_jN/q < q' \leq 2A_jN/q},  
\end{multline}
where $q'=\prod_{i=1}^h p_i'$. We wish to extract the greatest prime divisor of $qq'$ from the summations. We remember that it has to be larger than $R$, so the only possibility is $p'_h$. We set this index apart and name it $p$. Therefore, we put $q''=q'/p$ and the expression from (\ref{qq'}) equals
\begin{multline}
 \sum_{\substack{ p_1<\dots<p_{r-h} \\ q=[d_j, e_j]   }}    
\sum_{\substack{ p_1'<\dots < p_{h-1}' } }  
\sum_{p> \max(p_{h-1}',R) } \beta_r (qq''p)  \mathbf{1}_{A_jN/qq'' < p \leq 2A_jN/qq''} \\
=   \sum_{\substack{ p_1<\dots<p_{r-1} \\ [d_j, e_j]|q'''\\ \Omega ([d_i,e_i])=r-h  }}    
\sum_{p } W_r \left( \frac{\log p_1}{\log pq'''} , \dots , \frac{\log p_{r-1}}{\log pq'''} \right)  \mathbf{1}_{A_jN/q''' < p \leq 2A_jN/q'''},
\end{multline}
where $q''' = \prod_{i=1}^{r-1} p_i$. The condition $p>\max (p_{h-1}',R)$ could be erased due to the definition of $W_r$. Notice that the 'old' collection of $p_1, \dots , p_{r-h}$ is not necessarily equivalent to the 'new' collection -- there is a possibility that all of the $p_i$ and $p_i'$ were mixed with each other. From now on, we relabel $q'''$ as $q$ for the sake of convenience. We conclude that $M_0 + \dots + M_r$ equals
\begin{multline}\label{wniosek_prop5.2}
 \frac{1}{\varphi([A_j,W] \text{rad} \,A_j)} \sideset{}{'}\sum_{\substack{d_1, \dots , d_k \\  e_1, \dots,  e_k}}
 \frac{ \lambda_{d_1,\dots,d_k}  \lambda_{e_1,\dots,e_k}}{ \prod_{i\not=j} \varphi ([d_i,e_i])}
\sum_{\substack{ p_1<\dots<p_{r-1} \\ [d_j, e_j]|q   }}    
\sum_{ A_jN/q < p \leq 2A_jN/q } W_r \left( \frac{\log p_1}{\log pq} , \dots , \frac{\log p_{r-1}}{\log pq} \right)  \\
= \frac{1}{\varphi([A_j,W] \text{rad} \,A_j)}
  \sum_{\substack{ p_1<\dots<p_{r-1}   }}    T^{(j)}_q
\sum_{ A_jN/q < p \leq 2A_jN/q } W_r \left( \frac{\log p_1}{\log pq} , \dots , \frac{\log p_{r-1}}{\log pq} \right). 
\end{multline}
Combining  (\ref{E_h_error}), (\ref{wniosek_prop5.2}) with $(8.13)$ and $(8.14)$ from \cite{MaynardK}, we obtain for any $U \geq 1$
\begin{multline}
 \sum_{\substack{ N < n \leq 2N \\ n \equiv \nu_0 \bmod W}} \beta_r \left( L_j (n) \right)  \left( \sum_{\substack{d_1, \dots , d_k \\ \forall i ~ d_i | L_i(n) }} \lambda_{d_1,\dots,d_k} \right)^2  \\
= ~ \frac{ N}{\varphi(W) \log N}  \sum_{\substack{ p_1<\dots<p_{r-1}}} \frac{T^{(j)}_q}{q} 
  \alpha \left( \frac{\log p_1}{\log R} , \dots , \frac{\log p_{r-1}}{\log R} \right) \left( 1 + O\left( \frac{1}{\log N} \right) \right) + O_U(N \log^{-U} N),
\end{multline}
  where we noticed that
\begin{equation}
\frac{A_j}{\varphi([A_j,W] \text{rad} \,A_j)} = \frac{1}{\varphi (W)}.
\qedhere
\end{equation}

\end{proof}
 To finish the proof of Proposition \ref{5.2} we have to perform exactly the same reasoning which appears in the proof of Lemma 8.3 from \cite{MaynardK}. The only difference is the error term coming from transforming $T_q^{(j)}$ into the integral as in Lemma \ref{calkaI1} which, by the fact that $\alpha \ll 1$, can be easily calculated as follows:
  \begin{equation}
 \ll \frac{ F_{\max}^2 \varphi (W)^{k} N (\log R)^{k+1} }{W^{k+1} (\log N) D_0}   \sum_{R^\epsilon < p_1 < \dots < p_{r-1} < R } \frac{1}{q} \ll   \frac{ F_{\max}^2 \varphi (W)^{k} N (\log R)^{k} |\log \epsilon |^{r-1}}{W^{k+1}  D_0}
  \end{equation}
   where we added the restriction $R^\epsilon < p_1$ under the sum because of the support of $W_r$.

\subsection{Proof of Proposition \ref{5.3}}

\begin{proof}
The proof is the same regardless of whether we assume Hypothesis 1 or Hypothesis 2, so we describe only the first case. 
Define $A:= 2\max (A_1, \dots , A_k)$. Recall that $n \equiv \nu_0 \bmod W$ implies $(\mathcal{P}(n), W)=1$. We wish to estimate from above the following expression:

\begin{equation}\label{propozycja5.3}
\sum_{D_0 \leq p< AN^{1/2}} \sum_{\substack {N < n \leq 2N \\ n \equiv \nu_0 \bmod W \\ p^2|  L_j (n) }} \left(   \sum_{\substack{ d_1, \dots , d_k \\ \forall i ~ d_i|L_i(n)}} \lambda_{d_1,\dots, d_k}  \right)^2.
\end{equation}
We split the outer sum as follows:
\begin{equation}\label{3.23}
 \sum_{D_0 \leq p< AN^{1/2}}  = \sum_{D_0 \leq p< N^\eta}  +  \sum_{N^\eta \leq p< AN^{1/2}}. 
 \end{equation}
To calculate the second sum from (\ref{3.23}) we apply Lemma \ref{szacowanieSelbergow} and the divisor bound $\tau (n) \ll n^{o(1)} $ to obtain
\begin{align}
\sum_{N^\eta \leq p< AN^{1/2}} \sum_{\substack {N < n \leq 2N \\ n \equiv \nu_0 \bmod W \\ p^2|  L_j (n) }} \left(   \sum_{\substack{ d_1, \dots , d_k \\ \forall i ~ d_i|L_i(n)}} \lambda_{d_1,\dots, d_k}  \right)^2 \ll& ~ N^{o(1)}  \sum_{N^\eta \leq p< AN^{1/2}} \sum_{\substack {N < n \leq 2N \\ n \equiv \nu_0 \bmod W \\ p^2|  L_j (n) }} 1 \\
\ll& ~    N^{1+o(1)}  \sum_{p \geq N^{\eta}} \frac{1}{p^2} \ll N^{1-\eta + o(1)}.
\end{align}
In the case of the first sum from (\ref{3.23}) we can rearrange the order of summation and get
\begin{equation}\label{tutaj!}
\sum_{D_0 \leq p< N^\eta} 
  \sideset{}{'}\sum_{\substack{ d_1, \dots , d_k \\ e_1, \dots , e_k}} 
  \lambda_{d_1,\dots, d_k} \lambda_{e_1,\dots, e_k}  \sum_{\substack {N < n \leq 2N \\ n \equiv \nu_0 \bmod W \\ p^2|L_j (n)  \\ \forall i ~ [d_i,e_i] |L_i(n)}} 1.
  \end{equation}
As before in similar cases, we may add the $'$-condition freely, because the sum equals 0 otherwise.
 By Chinese remainder theorem the last summand equals
 \begin{equation}
 \frac{N}{W [d_j,e_j,p^2] \prod_{i \not= j}^k [d_i,e_i]} + O(1).
 \end{equation}
 Therefore, the expression from (\ref{tutaj!}) is
 \begin{equation}
  \frac{N}{W}\sum_{D_0 \leq p< N^\eta}  \sideset{}{'}\sum_{\substack{ d_1, \dots , d_k \\ e_1, \dots , e_k}}  \frac{  \lambda_{d_1,\dots, d_k} \lambda_{e_1,\dots, e_k}}{ [d_j,e_j,p^2] \prod_{i \not= j}^k [d_i,e_i]} + O\left( 
      \sum_{D_0 \leq p< N^\eta} 
   \sideset{}{'}\sum_{\substack{ d_1, \dots , d_k \\ e_1, \dots , e_k}} | \lambda_{d_1,\dots, d_k} \lambda_{e_1,\dots, e_k} | \right  ).
 \end{equation}
In order to calculate the error term we again apply Lemma \ref{szacowanieSelbergow} to estimate the Selberg weights. We also see that the product $d_1\dots d_ke_1 \dots e_k$ can take only values lower than $R^2$, so we obtain 
\begin{equation}
      \sum_{D_0 \leq p< N^\eta} 
   \sideset{}{'}\sum_{\substack{ d_1, \dots , d_k \\ e_1, \dots , e_k}} | \lambda_{d_1,\dots, d_k} \lambda_{e_1,\dots, e_k} | \ll
       y^2_{\max} \sum_{r < R^2N^\eta} \tau_{2k+1} (r) \ll   F^2_{\max} N^{\eta + o(1) }R^2 ,
\end{equation}
which is neglible for $2 \vartheta < 1- \eta$.

The main term equals
\begin{equation}
  \frac{N}{W}\sum_{D_0 \leq p< N^\eta}  \frac{T_{1, \dots , 1, p , 1, \dots , 1}}{p^2},
\end{equation}
where $p$ appears on the $j$-th coordinate. From Lemma \ref{Ttilde} we get
\begin{equation}\label{3.29}
\begin{gathered}
 \widetilde{T}_{1, \dots , 1, p , 1, \dots , 1} \ll  
  \frac{ \varphi (W)^k \log^k R}{W^k} 
 \int \limits_0^1 \dots \int \limits_0^1 G(t_1, \dots , t_k)\, dt_1 \, \dots \, dt_k 
 ~+ O \left( \frac{ F_{\max }^2 \varphi (W)^k (\log R)^{k} }{W^k D_0} \right)
 \end{gathered}
\end{equation}
 where
 \begin{equation}\label{3.30}
 G(t_1, \dots, t_k) = \left( F(t_1, \dots , t_k) - F (t_1, \dots , t_{j-1} , 
t_j + \frac{\log p}{\log R} , t_{j+1}, \dots , t_k ) \right)^2.
 \end{equation}
 Note that
 \begin{equation}\label{3.31}
G(t_1, \dots , t_k) \ll \frac{F_{\max}^2 \log^2 p}{\log^2 R}
 \end{equation}
for $t_j \leq 1- \frac{\log p}{ \log R}$, since 
 \begin{equation}\label{3.32}
 \frac{ \partial F}{\partial t_j} \ll F_{\max}.
 \end{equation}
For $t_j \geq  1- \frac{\log p}{ \log R}$ we obtain simply
\begin{equation}\label{3.33}
G(t_1, \dots , t_k) \ll F_{\max}^2.
\end{equation}
Combining (\ref{3.29})--(\ref{3.33}), we get
\begin{equation}
\widetilde{T}_{1, \dots , 1, p , 1, \dots , 1} \ll   \frac{ \varphi (W)^k (\log R)^{k}  }{W^k} 
 \left(  \frac{F_{\max}^2 \log^2 p}{\log^2 R} +  \frac{F_{\max}^2 \log p}{\log R} \right) \ll  
   \frac{ \varphi (W)^k \log^{k-1} R \log p}{W^k} .
\end{equation}
Thus,
\begin{multline}
\frac{N}{W}\sum_{D_0 \leq p< N^\eta} \frac{T_{1, \dots , 1, p , 1, \dots , 1}}{p^2} = 
\frac{N}{W}\sum_{D_0 \leq p< N^\eta} \frac{\widetilde{T}_{1, \dots , 1, p , 1, \dots , 1}}{p^2}  \\
+~ O \left( \frac{ F^2_{\max}\varphi (W)^k N (\log R)^k }{ W^{k+1} D_0 }
\sum_{D_0 \leq p< N^\eta} \frac{1}{p^2} \right).
\end{multline}
To sum up, 
\begin{multline}
\sum_{D_0 \leq p< N^\eta} \sum_{\substack {N < n \leq 2N \\ n \equiv \nu_0 \bmod W \\ p^2|  L_j (n) }} \left(   \sum_{\substack{ d_1, \dots , d_k \\ \forall i ~ d_i|L_i(n)}} \lambda_{d_1,\dots, d_k}  \right)^2   \\
\ll \frac{ F_{\max}^2  \varphi (W)^k N \log^{k-1} R}{W^{k+1}}\left(   \sum_{D_0 \leq p \leq N^\eta} \frac{\log p}{p^2}  + \frac{\log R}{D_0} \right) \ll \frac{ F_{\max}^2 \varphi (W)^k N (\log R)^{k} }{W^{k+1}D_0}.
\qedhere
\end{multline}

\end{proof}


\section{Proof of the main theorem}

\subsection{Setup}

We apply Propositions \ref{5.1}, \ref{5.2}, \ref{5.3} and \ref{5.4} to estimate $\mathcal{S}$ from below. We assume Hypothesis \ref{A} or Hypothesis \ref{B}. Fix $\epsilon > 0$ and put
\begin{equation}\label{4.1}
W_0 (x) =   1- \frac{\vartheta}{\vartheta_0}x ,
\end{equation}

\begin{equation}
W_{r,s} (x_1 , \dots , x_{r-1}) = \end{equation}
\[ \begin{cases}
       \frac{1}{\vartheta_0} - s - \frac{1}{\vartheta_0}\sum_{i=1}^{r-s} x_i,  & \text{if } \epsilon<x_1<\dots<x_{r-s} \leq \vartheta_0 < x_{r-s+1} < \dots < x_{r-1} \text{ and }   \\
& ~~   \sum_{i=1}^{r-1} x_i < 1- x_{r-1} , \\
        0 & \text{otherwise}
        \end{cases}
\]
for any $r,s \in \mathbf{N}$. 

By Proposition \ref{5.3} we have
\begin{multline}\label{4.4}
\mathcal{S}' \leq 
 \sum_{\substack {N < n \leq 2N \\ n \equiv \nu_0 \bmod W \\ \mu \left( \mathcal{P}(n) \right)^2 \not=1 }} \left( \sigma + \frac{\log \mathcal{P}(n)}{\log R_0}\right) \Lambda_{\text{Sel}}^2 (n)  \ll 
  \sum_{\substack {N < n \leq 2N \\ n \equiv \nu_0 \bmod W \\ \mu \left( \mathcal{P}(n) \right)^2 \not=1 }} \Lambda_{\text{Sel}}^2 (n)
\ll  \frac{ F_{\max}^2 \varphi (W)^k N (\log R)^{k}  }{W^{k+1}D_0}. 
\end{multline}

Proposition \ref{5.1} gives us 
\begin{equation}\label{4.5}
T_0 =  \frac{\varphi(W)^k N (\log R)^{k}  }{W^{k+1}} \sum_{j=1}^k J_0^{(j)}  + O\left( \frac{F_{\max}^2 \varphi(W)^k N (\log R)^{k} ( \epsilon + \frac{|\log \epsilon |}{D_0} ) }{W^{k+1} } \right).
\end{equation}

We apply Proposition \ref{5.2} with $\beta_r (n)=\chi_{r,s} (n)$, where $\chi_{r,s} (n)$ is defined as in (\ref{1.31}) and where $J_r$ is relabeled as $J_{r,s}$. We get
\begin{equation}\label{4.6}
 T_{r,s}^{(j)} =   \frac{\varphi(W)^k N (\log R)^{k+1} }{W^{k+1} \log N} J_{r,s}^{(j)} + O_\epsilon \left(\frac{F_{\max}^2 \varphi(W)^k N (\log R)^{k}  }{W^{k+1} D_0} \right).
 \end{equation}
 
 Proposition \ref{5.4} gives
 \begin{equation}\label{4.7}
 S_0 = 
  \frac{\varphi(W)^k N (\log R)^{k} }{W^{k+1}} J + O\left( \frac{F^2_{\max} \varphi(W)^k N (\log R)^{k} }{W^{k+1}D_0} \right).
 \end{equation}
 
 Combining (\ref{4.4}), (\ref{4.5}), (\ref{4.6}), and (\ref{4.7}) we obtain
 \begin{multline}
 \mathcal{S} \geq  \frac{\varphi(W)^k N (\log R)^{k} }{W^{k+1}}\left(  \sigma J - \sum_{j=1}^kJ_0^{(j)} + \vartheta \sum_{j=1}^k  \sum_{r=1}^h \sum_{s=1}^r J_{r,s}^{(j)}
  \right) - \\
  O \left(  \frac{\epsilon F_{\max}^2 \varphi(W)^k N (\log R)^{k} }{W^{k+1}} \right) -  O_\epsilon \left(\frac{F_{\max}^2 \varphi(W)^k N (\log R)^{k}  }{W^{k+1} D_0} \right).
  \end{multline}
 From this we can see that if
 \begin{equation}
   \sigma = \frac{ \sum_{j=1}^kJ_0^{(j)} - \vartheta \sum_{j=1}^k  \sum_{r=1}^h \sum_{s=1}^r J_{r,s}^{(j)} }{J} + \epsilon C_F ,
 \end{equation}
where $C_F$ is a sufficiently large constant depending only on $F$, then $\mathcal{S}>0$. We put
 \begin{equation}
 \widetilde{\Upsilon}_k (F; \vartheta, \vartheta_0, \epsilon )=
  \frac{ \sum_{j=1}^kJ_0^{(j)} - \vartheta \sum_{j=1}^k  \sum_{r=1}^h \sum_{s=1}^r J_{r,s}^{(j)} }{J} + \frac{k}{\vartheta_0}.
 \end{equation}
 Therefore, from (\ref{1.22}) and (\ref{1.26}) we have
 \begin{equation}\label{4.10}
 \Omega ( \mathcal{P} (n) ) \leq  \widetilde{\Upsilon}_k (F; \vartheta, \vartheta_0, \epsilon ) + O_F(\epsilon)
 \end{equation}
 infinitely often. We also note that $ \widetilde{\Upsilon}_k$ is actually well defined for every $\vartheta \in [0, \frac{1}{2}]$, $\vartheta_0 \in [ 0, 1]$, $\epsilon \in [0,1]$. We put
\begin{equation} 
\begin{split}
\widetilde{\Omega}_k ( \vartheta, \vartheta_0, \epsilon ) = \inf_{} \{ \widetilde{\Upsilon}_k (F ; \vartheta, \vartheta_0, \epsilon ) : F \mbox{ is smooth and supported on } \mathcal{R}_k \}, \\
 \widetilde{\Omega}_k^{\text{ext}} (\vartheta, \vartheta_0, \epsilon) = \inf_{} \{ \widetilde{\Upsilon}_k (F; \vartheta, \vartheta_0, \epsilon) : F \mbox{ is smooth and supported on } \mathcal{R}'_k \}.
\end{split}
\end{equation}
We wish to find good upper bounds for $\widetilde{\Omega}_k$ and $\widetilde{\Omega}_k^{\text{ext}}$ by suitable choices of functions $F$. We restrict our attention to multivariate polynomials. Moreover, if our polynomial is symmetric, then the main term from (\ref{4.10}) simplifies to
\begin{equation}\label{4.12}
 \frac{k(J_0 - \vartheta   \sum_{r=1}^h \sum_{s=1}^r J_{r,s} )}{J} + \frac{k}{\vartheta_0},
\end{equation}
where $J_0 = J_0^{(1)}$ and $J_{r,s} = J_{r,s}^{(1)}$ for every $r,s \in \mathbf{N}$. Observe that $\widetilde{\Upsilon}_k$ is continuous and differentiable with respect to each variable. This implies 
\begin{equation}\label{Oeps}
\begin{split}
\widetilde{\Upsilon}_k (F ; \vartheta, \vartheta_0, \epsilon ) &= \widetilde{\Upsilon}_k (F ; \vartheta, \vartheta_0, 0 ) + O_F (\epsilon).
 \end{split}
\end{equation}
We usually take $\epsilon$ very close to $0$, so (\ref{Oeps}) motivates the following definitions. 
\begin{equation}\label{continuity}
\begin{split}
 \Upsilon_k (F ; \vartheta, \vartheta_0 ) &=\widetilde{\Upsilon}_k (F ; \vartheta, \vartheta_0, 0 ) , \\
 \Omega_k^{\text{ext}} (\vartheta, \vartheta_0 ) &=   \widetilde{\Omega}_k^{\text{ext}} (\vartheta, \vartheta_0, 0 ), \\ 
 \Omega_k  (\vartheta, \vartheta_0 ) &=\widetilde{\Omega}_k  (\vartheta, \vartheta_0, 0).
\end{split}
\end{equation}
Therefore, we can summarize the above discussion in the following result. 
\begin{theorem}\label{tw_J} Let $k$ be an integer greater than or equal to $3$. Let also $\epsilon >0$, and $\mathcal{H}$ be an admissible $k$--tuple. Then, we have
\begin{enumerate}
\item $\Omega ( \mathcal{P} (n) ) \leqslant {\Omega}_k + O(\epsilon) $ infinitely often, if $\vartheta$ and $\vartheta_0$ satisfy Hypothesis $\ref{A}$;
\item $\Omega ( \mathcal{P} (n) ) \leqslant {\Omega}_k^{\text{\normalfont{ext}}} + O(\epsilon) $ infinitely often, if $\vartheta$ and $\vartheta_0$ satisfy Hypothesis $\ref{B}$.
\end{enumerate}
\end{theorem}

\subsection{The GEH case}

Let us begin this subsection with an observation concerning the integrals $J_0$ and $J_{r,s}$. The precise shape and the very existence of these integrals can be perceived as a consequence of the relations between (\ref{1.8}) and (\ref{4.10}).
\begin{align}\label{4.14}
\begin{split}
 \sum_{\substack{ p|n,~ p \leq R_0}} \left( 1 - \frac{\log p}{\log R_0} \right)~~ &\text{   transforms into   } ~~ J_0/J,\\ 
  \sum_{r=1}^\infty \sum_{s=1}^r  \chi_{r,s} (n) ~~&\text{   transforms into   } ~~ \left(  \vartheta \sum_{r=1}^\infty \sum_{s=1}^r  J_{r,s} \right) / J,\\
   \frac{\log n}{\log R_0} ~~&\text{   transforms into   } ~~ \frac{k}{\vartheta_0}.
\end{split}
\end{align} 
The $J$ integral is basically a normalising term. Notice that it is relatively easy to calculate ${\Omega}_k $ or ${\Omega}_k^{\text{ext}} $ upon $\vartheta_0=1$, because in such a case we would have $J_{r,s}=0$ for every permissible pair $r,s$. Thus, if $\vartheta_0=1$ and $F$ is symmetric, then (\ref{4.12}) reduces to
\begin{equation}\label{4.76}
 \frac{kJ_0 }{J} + \frac{k}{\vartheta_0},
\end{equation}
We can view $\sum_{r=1}^\infty \sum_{s=1}^r  \chi_{r,s} (n)$ as a correction term which appears to give back the contribution taken away by the $p \leq R_0$ restriction in the upper sum from (\ref{4.14}). Therefore, we can propose the following
\begin{conj}\label{CONJ} The function $\Omega_k (\vartheta, \vartheta_0)$ is constant with respect to $\vartheta_0$. The same applies to ${\Omega}_k^{\text{\normalfont{ext}}}$.
\end{conj}

In the $GEH$ case of Theorem \ref{MAIN} we apply the first part of Theorem \ref{tw_J} with $\vartheta=\vartheta_0=1/3$. Conjecture \ref{CONJ} is probably  not exactly correct the way it is stated, but it has been very useful for choosing good (conjecturally near-optimal) polynomials $F$ for various $k$. In general, it is just much easier to minimise the expression (\ref{4.76}) rather than (\ref{4.12}). Our choices are listed in Table C.

\begin{center}
\centering
\text{Table C.}
\vspace{1mm}
\\
  \begin{tabular}{ | C{0.7cm} | C{6cm}  | C{6cm}  |  }
    \hline
    $k$ & $F_1$ & $F_2$ \\ \hline
    $3$ & $529 - 877 P_1 + 567 P_1^2 - 189 P_1^3$  & $1846 - 3225 P_1 + 2203 P_1^2 - 727 P_1^3 + 228 P_2 - 223 P_1P_2$\\ \hline
    $4$ & $17950 - 36681 P_1 + 28786 P_1^2 - 9510 P_1^3$  & $20875 - 43615 P_1 + 33273 P_1^2 - 10000 P_1^3 + 4867 P_2 - 4649 P_1P_2$\\ \hline
    $5$ & $15566 - 35617 P_1 + 30136 P_1^2 - 9807 P_1^3$  &  $17195 - 40385 P_1 + 33413 P_1^2 - 10000 P_1^3 + 5366 P_2 - 5148 P_1P_2$ \\  \hline
    $6$ &  $12739 - 31508 P_1 + 28087 P_1^2 - 9178 P_1^3$ & $11908 - 30242 P_1 + 26486 P_1^2 - 8071 P_1^3 + 4310 P_2 - 4162 P_1 P_2$  \\  \hline
    $7$ &  $11754 - 30703 P_1 + 28386 P_1^2 - 9354 P_1^3$ & $11091 - 29705 P_1 + 27075 P_1^2 - 8420 P_1^3 + 4322 P_2 - 4197 P_1 P_2$  \\  \hline 
    $8$ & $11131 - 30235 P_1 + 28687 P_1^2 - 9531 P_1^3$  & $10523 - 29241 P_1 + 27419 P_1^2 - 8679 P_1^3 + 4232 P_2 - 4128 P_1 P_2$  \\  \hline
    $9$ & $6710 - 18690 P_1 + 18003 P_1^2 - 6001 P_1^3$  &  $9528 - 27175 P_1 + 26009 P_1^2 - 8351 P_1^3 + 3857 P_2 - 3775 P_1 P_2$ \\  \hline
    $10$ & $6573 - 18606 P_1 + 18072 P_1^2 - 6024 P_1^3$  &  
  $  9726 - 1513 P_1 + 712 P_1^2 - P_1^3 + 4548 P_2 - 3828 P_1 P_2$ \\   
      \hline
  \end{tabular}
\end{center}
Recall that $P_j = \sum_{i=1}^k t_i^j$. The $F_1$ polynomials are chosen to be almost best from all polynomials of the form
 \[ a_0 + a_1(1-P_1) + a_2(1-P_1)^2 + a_3(1-P_1)^3\] 
 with $a_0,a_1,a_2,a_3>0$. One may notice that polynomials listed in the $F_1$ column resemble $(1-P_1)^3$ quite much, up to constant multipliers. On the other hand, we can also take into account all other possible terms of order greater than or equal to $3$ constructed from $(1-P_1)$, $P_2$ and find the almost optimal polynomials $F_2$ of the form 
 \[b_0 + b_1(1-P_1) + b_2(1-P_1)^2 + b_3(1-P_1)^3 + b_4P_2 + b_5(1-P_1)P_2\] 
 with $b_0,b_1,b_2,b_3,b_4,b_5>0$. We are able to calculate the upper bound on $\Upsilon (F_1; \frac{1}{3}, \frac{1}{3})$ thanks to the specific shape of $F_1$ -- as pointed out in Remark \ref{1dimrem} we are able to convert the integrals $J_0$ and the $J_{r,s}$ into simpler ones.
 
   Considerations mentioned above lead us to the results listed in Table D. In the first column we present the upper bound on $\Omega_k$ given by calculating the contributions from $J_0$, $J_{1,1}$, $J_{2,1}$, $J_{3,1}$, $J_{2,2}$, and $J_{3,2}$ into $\Upsilon (F_1)$. In the second column we extract also the contribution incoming from $J_{4,1}$, which can be considered as a black box due to its complicated shape. The third column contains the possible true values of $\Upsilon (F_1)$ predicted by Conjecture \ref{CONJ}. In the fourth column we have the same thing for $F_2$. As we can see in Table D, the differences between $\Upsilon (F_1)$ and $\Upsilon (F_2)$ are probably minuscule. As numerical experiments suggest, we are also unable to win much by considering higher powers of $(1-P_1)$, $P_2$ or even the negative $a_i$, $b_i$.

\begin{center}\label{wyniki}
\centering
\text{Table D.}
\vspace{1mm}
\\
  \begin{tabular}{ | C{0.7cm} |  C{3.2cm}  | C{3.2cm}  | C{3.2cm}  | C{3.2cm} | }
    \hline
       $k$ &  \footnotesize{ Upper bound on $\Omega_k $} & \footnotesize{Upper bound on $\Omega_k$  (contribution from $J_{4,1}$ included)} & \footnotesize{True value of  $\Upsilon_k (F_1)$  foreseen by Conjecture \ref{CONJ} } &  \footnotesize{True value of  $\Upsilon_k (F_2)$  foreseen by Conjecture \ref{CONJ} } \\ \hline
    $3$ &  7.530\dots& \bf{7.415\dots} & 7.38120\dots & 7.38096\dots \\ \hline
    $4$ &  10.750\dots &   \bf{10.523\dots} & 10.44612\dots & 10.44486\dots \\ \hline
    $5$ &  14.192\dots &  \bf{13.828\dots}  & 13.68862\dots & 13.68492\dots \\  \hline
    $6$ &  17.822\dots  & \bf{17.301\dots}    & 17.07933\dots  & 17.07180\dots  \\ \hline
    $7$ &  21.615\dots  &  \bf{20.921\dots}   &  20.59761\dots & 20.58499\dots \\ \hline
    $8$ &  25.550\dots  &  \bf{24.672\dots}   & 24.22809\dots  & 24.20929\dots \\ \hline
    $9$ & 29.614\dots   &  \bf{28.541\dots}   &  27.95889\dots  &  27.93302\dots \\ \hline
    $10$ & 33.795\dots  &  \bf{32.519\dots}  &  31.78061\dots  & 31.74691\dots \\
      \hline
  \end{tabular}
\end{center}
 By (\ref{Oeps}) we have
\begin{equation}
 \widetilde{\Upsilon}_k \left( F_1; \frac{1}{3} - \epsilon, \frac{1}{3}, \epsilon \right) = 
 \Upsilon_k \left( F_1; \frac{1}{3}, \frac{1}{3} \right) + O(\epsilon),
\end{equation}
The choice $\vartheta = 1/3 - \epsilon$ and $\vartheta_0 = 1/3$ satisfies all assertions of Hypothesis \ref{A} under $GEH[2/3]$, so fixing $\epsilon$ sufficiently close to $0$ enables us to use the results from Table D to prove the conditional part of Theorem \ref{MAIN}. 

As we can see, there should be a possibilty to improve the result for $k=9$ and $k=10$ by considering the contributions from integrals like $J_{5,1}$ or $J_{4,2}$ (some experiments with Monte Carlo method suggest that $J_{5,1}$ should contribute much more than $J_{4,2}$; unfortunately, calculating this expression is extremaly onerous). We also have no improvement for $k=3$ in view of what is already proven unconditionally (see \cite{3-tuples}).

There is a problem with using the full strength of $GEH$. Namely, we are limited by the inequality $\vartheta \leq \vartheta_0$. If $\vartheta > \vartheta_0$, then we cannot apply Proposition \ref{5.2} to our choices of $W_{r,s}$ (which are forced by the specific shape of functions $\chi_{r,s} (n)$). On the other hand, we can perform a little trick to overcome this issue. Take a look at the following inequality: 

\begin{equation}\label{4.16'}
\Omega (n) \geq \sum_{\substack{ p|n \\ p \leq y}} \left( 1 - \frac{\log p}{\log y} \right) + \frac{\log n}{\log y} + \sum_{r=1}^\infty \varsigma_r (n),
\end{equation}
where
\begin{equation}
 \varsigma_r (n) = 
\begin{cases}
- \left( \frac{\log n}{\log y} - 1 - \sum_{i=1}^{r-1} \frac{\log p_i}{\log y} \right), & \text{if~} n=p_1 \dots p_r \text{~with~ }  \\
& ~~  p_1 < \dots < p_{r-1} \leq y< p_r < n^{\vartheta},\\
\\
0, & \text{otherwise. }
\end{cases}
\end{equation}
Put 
\begin{equation}
\begin{gathered}
W^{\flat}_{r,s} (x_1 , \dots , x_{r-1}) = \\
 \begin{cases}
       \frac{1}{\vartheta_0} - s - \frac{1}{\vartheta_0}\sum_{i=1}^{r-s} x_i,  & \text{if } \epsilon<x_1<\dots<x_{r-s} \leq \vartheta_0 < x_{r-s+1} < \dots < x_{r-1} \text{ and }   \\
& ~~   \sum_{i=1}^{r-1} x_i < \min ( 1- x_{r-1}, 1-\vartheta ), \\
        0 & \text{otherwise}
        \end{cases}
\end{gathered}
\end{equation}
for all $r,s \in \mathbf{N}$ in the place of $W_{r,s}$. Now, we repeat the reasoning from (\ref{1.22})--(\ref{1.31}) with $\varsigma_r$ instead of $\chi_r$ and an analogous choice of $y$. Notice that we are allowed to apply Propositions \ref{5.1}--\ref{5.4} assuming only points $1,3,4,5$ from Hypothesis \ref{A}. Proceeding like in (\ref{4.1})--(\ref{4.10}) we conclude that if $F$ is a symmetric polynomial of $k$ variables, then we have
 \begin{equation}
 \Omega ( \mathcal{P} (n) ) \leq    \frac{k(J_0 - \vartheta   \sum_{r=1}^h \sum_{s=1}^r J^{\flat}_{r,s} )}{J} + \frac{k}{\vartheta_0} + O(\epsilon),
 \end{equation}
where the $J^\flat_{r,s}$ are the same as $J_{r,s}$ but with $W^{\flat}_{r,s}$ instead of $W_{r,s}$. Now, the condition ${\vartheta \leq \vartheta_0}$ can be discarded, so we can try to use $GEH[ \theta ]$ with an exponent greater than $2/3$, without harming the condition $\vartheta_0 + 2\vartheta < 1$. Unfortunately, this possibility has its price. Note that (\ref{4.16'}) is not an equality like $(1.11)$. That difference implies that some part of contribution given previously by $J_{r,s}$ disappears (it is not surprising because $\mbox{supp} \, (W^\flat_{r,s}) \subset \mbox{supp} \, (W_{r,s})$). If $\vartheta>\vartheta_0$, then the bigger the difference $\vartheta - \vartheta_0$ is, the stronger that phenomenon is going to be. Up to some point we are able to make some little progress over the results listed in Table D. Put $\vartheta_0 = 1 - 2\vartheta$. Define $\Upsilon^\flat_k$, $\widetilde{\Omega}^\flat_k$ and $\Omega^\flat_k$ the same way as their `non-$\flat$' analogues, but with $J^\flat_{r,s}$ instead of $J_{r,s}$. Taking polynomials $F_1$ from Table C and considering the contribution from $J^\flat_{1,1}$, $J^\flat_{2,1}$, $J^\flat_{3,1}$, $J^\flat_{4,1}$, $J^\flat_{2,2}$, $J^\flat_{3,2}$, $J^\flat_{3,3}$ integrals we get 

\begin{center}\label{wyniki}
\centering
\text{Table E.}
\vspace{1mm}
\\
  \begin{tabular}{ | C{0.7cm} | C{1.5cm} |  C{3cm}  |  }
    \hline
      $k$ &  $\vartheta$ &  \footnotesize{Upper bound on $\Omega^\flat_k$} \\ \hline
    $3$ & 0.371 & \bf{7.278\dots}   \\ \hline
    $4$ & 0.365 & \bf{10.389\dots} \\ \hline
    $5$ & 0.355 & \bf{13.704\dots} \\ \hline
    $6$ & 0.355 & \bf{17.184\dots} \\ \hline
    $7$ & 0.352 & \bf{20.817\dots} \\ \hline
    $8$ & 0.350 & \bf{24.582\dots} \\ \hline
    $9$ & 0.347 & \bf{28.467\dots} \\ \hline
    $10$ & 0.345 & \bf{32.459\dots}   \\
      \hline
  \end{tabular}
\end{center}
The improvements over bounds listed in Table D are rather subtle. Numerical experiments suggest that taking greater $\vartheta$ worsens the results because the rising gap between $J_{r,s}$ and $J^\flat_{r,s}$ takes from us more contribution to $\Upsilon^\flat_{k}(F_1)$ than the increasing $\vartheta$ can possibly offer. In order to use the full power of $GEH$ one needs to find a way around this issue. 

Conjecture \ref{CONJ} can also be used to predict the limits of Maynard's original method from \cite{MaynardK}. He used the identity (\ref{1.8}) instead of (\ref{1.11}), so $\vartheta_0 = \frac{1}{2} $ and $\vartheta = \frac{1}{4}$ is the optimal choice in this situation. We find close to best polynomials $F_1$ and $F_2$ (which are not written explicitly here) in the same way as these listed in Table C: 
\newpage
\begin{center}\label{wyniki}
\centering
\text{Table F.}
\vspace{1mm}
\\
  \begin{tabular}{ | C{0.7cm} |  C{3.2cm}  | C{3.2cm}  |  }
    \hline
       $k$ &   \footnotesize{True value of  $\Upsilon_k (F_1)$  foreseen by Conjecture \ref{CONJ} } &  \footnotesize{True value of  $\Upsilon_k (F_2)$  foreseen by Conjecture \ref{CONJ} } \\ \hline
    $3$ & 8.15617\dots & 8.15570\dots \\ \hline
    $4$ & 11.49648\dots & 11.49460\dots \\ \hline
    $5$ & 15.01694\dots & 15.01210\dots \\ \hline
    $6$ & 18.68736\dots &18.67814\dots \\ \hline
    $7$ & 22.48660\dots & 22.47174\dots \\ \hline
    $8$ & 26.39899\dots & 26.37743\dots \\ \hline
    $9$ &  30.41245\dots & 30.38333\dots \\ \hline
    $10$ & 34.51747\dots  & 34.47996\dots \\
      \hline
  \end{tabular}
\end{center}
Our heuristic predicts that Maynard's results from \cite{MaynardK} cannot be improved purely by his method. Moreover, even using multidimensional variation of his techniques does not make the situation much better. It is a good question whether it is possible to cross the $15$ barrier in the case $k=5$ relying only on better choice of polynomial than $F_2$, but without using the extended sieve support like in the next Subsection. It turns out that even if we choose the optimal polynomial $F$ of the form:
 \begin{multline*} c_{0} + c_{1}(1-P_1) + c_{11}(1-P_1)^2 + c_{2}P_2 + c_{111}(1-P_1)^3 +  c_{21}(1-P_1)P_2 + c_{3}P_3 \\
 + c_{1111}(1-P_1)^4 + c_{211}(1-P_1)^2P_2 + c_{22}P_2^2+ c_{31}(1-P_1)P_3  + c_4P_4 
\end{multline*}
with all the coefficients being real (so we can also consider negative values) then by Conjecture \ref{CONJ} we can expect only $\Upsilon_5 (F) = 15.01185\dots$. This discussion shows that the extended sieve support was necessary to prove the unconditional part of Theorem \ref{MAIN}.

Conjecture \ref{CONJ} also allows us to predict the optimal upper bounds on  ${\Omega}_k^{\text{\normalfont{ext}}} \left( \frac{1}{4}, \frac{3}{8} \right)$ incoming from polynomials $F$ of the form $a(1-P_1)^2+b(1-P_1)+c$. The results are listed below:

\begin{center}\label{wyniki}
\centering
\text{Table G.}
\vspace{1mm}
\\
  \begin{tabular}{ | C{0.7cm} |  C{3.2cm}  |  }
    \hline
       $k$ &    \footnotesize{True value of  $\Upsilon_k (F)$  foreseen by Conjecture \ref{CONJ} } \\ \hline
    $3$ & 7.85039\dots  \\ \hline
    $4$ & 11.27780\dots \\ \hline
    $5$ & 14.84799\dots  \\ \hline
    $6$ & 18.55495\dots  \\ \hline
    $7$ & 22.38582\dots  \\ \hline
    $8$ & 26.32852\dots \\ \hline
    $9$ & 30.37274\dots   \\ \hline
    $10$ & 34.50978\dots   \\
      \hline
  \end{tabular}
\end{center}
We can foresee that the extending of sieve support allows us to prove (\ref{glowne_szacowanie}) with $\rho_5=14$ and reprove the main result from \cite{3-tuples}, i.e. that  (\ref{glowne_szacowanie}) is true with $\rho_3=7$. 

The discussion in this subsection leads us to the conclusion that Theorem \ref{MAIN} cannot be improved only by optimizing parameters except possibly the cases $k=9,10$ under $GEH$. One needs to rely on some sort of new ideas in order to set new records. Perhaps, the technology developed in \cite{Polymath8} can be used to expand the sieve support even further and obtain new results.


\subsection{The upper bound for $\Omega_5^{\text{ext}} \left( \frac{1}{4} , \frac{3}{8}  \right) $}
For $k=5$ we put 
\begin{equation} F(t_1,t_2,t_3,t_4,t_5) = 
\begin{cases}
11 + 85(1-P_1) + 170(1-P_1)^2, & \text{if~} (t_1,t_2,t_3,t_4,t_5) \in \mathcal{R}'_5,  \\
0, & \text{otherwise, }
\end{cases}
\end{equation}
and take
\begin{equation}
\vartheta = \frac{1}{4},~~~~~~~\vartheta_0 = \frac{3}{8}. 
\end{equation}
Note that we have $F(t_1, t_2,t_3,t_4,t_5) = f(t_1+t_2+t_3+t_4+t_5)$ for $(t_1,t_2,t_3,t_4,t_5) \in \mathcal{R}'_5$, where 
\begin{equation}
f(x)=11+85(1-x)+170(1-x)^2.
\end{equation}
In the unextended variation of multidimensional Selberg sieve (i.e. when $\text{supp} \, (F) \subset \mathcal{R}_k$) this symmetry transforms our sieve into its one-dimensional analogue (see Remark in \cite[Section 6]{Maynard}). However, our function $F$ does not satisfy it for all points $(t_1,t_2,t_3,t_4,t_5) \in \mathbf{R}_{\geq 0}^5$. For this reason it should not be surprising that we are able to make a progress compared to what is achievable by a one-dimensional sieve. 

With our choices of $k$, $F$, $\vartheta$, and $\vartheta_0$ we shall calculate (\ref{4.12}) to get the desired result. The precise values of integrals appearing in this subsection are found by Mathematica 11.

\subsubsection{Calculating $J_{1,1}$}

For the sake of generality, we perform the calculations for an arbitrary $k$. We also assume that the symmetry $F(t_1, \dots ,t_k) = f(t_1+ \dots +t_k)$ is satisfied for $(t_1, \dots ,t_k) \in \mathbf{R}_{\geq 0}^k$, where $f \colon \mathbf{R}\rightarrow \mathbf{R}$ is a piecewise smooth function. We have
\begin{equation}\label{4.16}
 J_{1,1}=\frac{1- \vartheta_0 }{ \vartheta_0 } \idotsint \limits_{\mathcal{R}_{k-1}}  \left( \int \limits_0^{\rho (t_2, \dots ,t_k)}  f(t_1 + \dots + t_k) \, dt_1 \right)^2\, dt_2 \, \dots\, dt_k
\end{equation}
where $ \rho (t_2, \dots ,t_k) = \sup \{ t_1 \in \mathbf{R} \colon (t_1, \dots , t_k) \in \mathcal{R}'_k \}$. We observe that any permutation of the variables $t_2, \dots , t_k$ does not change the integrand. We also notice that $0\leq t_2 \leq \dots \leq t_k$ implies
\begin{equation}
 \rho (t_2, \dots ,t_k ) = 1- t_3 - \dots - t_k.
\end{equation}
Therefore, the outer integral from (\ref{4.16}) equals
\begin{equation}\label{4.18}
 (k-1)! \idotsint \limits_{\substack{ \mathcal{R}_{k-1} \\ 0\leqslant t_2 \leq \dots \leq t_k }}  \left( \int \limits_{\sum_{i=2}^k t_i}^{1+t_2} f(x)\, dx \right)^2\, dt_2 \, \dots\, dt_k
\end{equation}
We make a substitution $t = t_2 + \dots + t_k$ and transform the integral over $\mathcal{R}_{k-1}$ from expression (\ref{4.18}) into
\begin{equation}\label{4.19}\int  \limits_0^1 \idotsint \limits_{ 0\leqslant t_2 \leq \dots \leq t_{k-1} \leq t - \sum_{i=2}^{k-1} t_i }  \left( \int \limits_t^{1+t_2} f(x)\, dx \right)^2\, dt_2 \, \dots\, dt_{k-1} \, dt
\end{equation}
For the sake of clarity, we put $s$ in the place of $t_2$. We can rewrite (\ref{4.19}) as 
\begin{equation}\label{4.20}
\int  \limits_0^1  \int \limits_0^{\frac{t}{k-1}}  \left( \int \limits_t^{1+s} f(x)\, dx \right)^2
\left(  \int \limits_s^{\frac{t-s}{k-2}}  \int \limits_{t_3}^{\frac{t-s-t_3}{k-3}}  \cdots \int \limits_{t_{k-2}}^{\frac{t-s-t_3-\dots-t_{k-2}}{2}}  
   \, dt_{k-1} \, \dots\, dt_3 \right)\, ds\, dt
\end{equation}
By induction we calculate that the expression in the right-hand side parentheses from (\ref{4.20}) equals
\begin{equation}\label{4.21}
\frac{(t-(k-1)s)^{k-3}}{(k-2)!(k-3)!}.
\end{equation}
Combining (\ref{4.16})--(\ref{4.21}) we conclude that 
\begin{equation}\label{4.22}
 J_{1,1}=\frac{1- \vartheta_0 }{ \vartheta_0 (k-3)! }  \int  \limits_0^1  \int \limits_0^t  \left( \int \limits_t^{1+\frac{s}{k-1}} f(x)\, dx \right)^2   (t-s)^{k-3}\, ds\, dt .
\end{equation}
In the $k=5$ case the expression from (\ref{4.22}) equals
\begin{equation}\label{J1,1_wynik}
\frac{5}{6} \int  \limits_0^1  \int \limits_0^t  \left( \int \limits_t^{1+\frac{s}{4}} f(x)\, dx \right)^2
(t-s)^2\, ds\, dt > 9.4661240888 .
\end{equation}
\begin{remark}\label{1dimrem} Note that in the $\Omega_k$ case, when the function $F$ is supported on $\mathcal{R}_k$, we have $1-t_2-\dots-t_k$ instead of $\rho(t_2,\dots,t_k)$ in (\ref{4.16}). By perfoming the same procedures as in (\ref{4.16})--(\ref{4.22}) we get the same integral as in (\ref{4.22}) but with $1+s$ replaced by $1$ in the limit of integration. That would lead to 
\[  J_{1,1}=\frac{1- \vartheta_0 }{ \vartheta_0 }  \int  \limits_0^1   \left( \int \limits_t^{1} f(x)\, dx \right)^2 
\frac{t^{k-2}}{(k-2)!}\, dt \] 
which is known from \cite[(5.30)]{MaynardK}. The same observation applies to $J_0$ and the other $J_{r,s}$ integrals.
\end{remark}

\subsubsection{Calculating $J$}

We proceed as in the previous subsubsection. We have
\begin{equation}\label{4.24}
 J =  \idotsint \limits_{\mathcal{R}'_{k}}    f(t_1 + \dots + t_k)^2\, dt_1 \, \dots\, dt_k= 
 \idotsint \limits_{\mathcal{R}_{k-1}}  \int \limits_0^{\rho (t_2, \dots ,t_k)}  f(t_1 + \dots + t_k) ^2\, dt_1 \, \dots\, dt_k.
\end{equation}
Taking the lower expression from (\ref{4.24}) and proceeding like in (\ref{4.16})--(\ref{4.21}) we can easily get that in the $k=5$ case we have
\begin{equation}\label{J_wynik}
J =\frac{1}{  2 }  \int  \limits_0^1  \int \limits_0^t  \int \limits_t^{1+\frac{s}{4}} f(x)^2    (t-s)^2 \, dx\, ds\, dt >
14.3115286045.
\end{equation}

\subsubsection{Calculating $J_0$}
The function $F$ is assumed to be symmetric, so we have
\begin{equation}\label{4.26}
 J_0 =  \int \limits_0^{\vartheta_0/\vartheta} \frac{\vartheta_0 - \vartheta y}{\vartheta_0 y} \idotsint \limits_{\mathcal{R}'_{k}}    \left( F(t_1, \dots  ,t_k) - 
 F(t_1+y,t_2, \dots , t_k) \right)^2\, dt_1 \, \dots\, dt_k\, dy .
\end{equation}
The inner integral equals
\begin{multline}\label{4.27}
 \idotsint \limits_{\mathcal{R}_{k-1}} \int \limits_{0}^{1}    \left( F(t_1, \dots  ,t_k) - 
 F(t_1+y,t_2, \dots , t_k) \right)^2\, dt_1\, dt_2 \, \dots\, dt_k \\
=  \idotsint \limits_{\mathcal{R}_{k-1}} \int \limits_{0}^{\varpi (y, t_2, \dots, t_k)}    \left( f(t_1 + \dots + t_k) - 
 f(y + t_1+ \dots + t_k) \right)^2\, dt_1\, dt_2 \, \dots\, dt_k \\
+  \idotsint \limits_{\mathcal{R}_{k-1}} \int \limits_{\varpi (y, t_2, \dots, t_k)}^{\rho (t_2, \dots ,t_k)}   f ( t_1 + \dots  + t_k )^2\, dt_1\, dt_2 \, \dots\, dt_k =: \int_1 + \int_2,
 \end{multline}
where $\varpi (y, t_2, \dots, t_k) = \max \left( 0 , \sup \{ t_1 \in \mathbf{R} \colon (t_1+y,t_2, \dots , t_k) \in \mathcal{R}'_k \} \right)$. We have $\int_1=0$ for $y > 1$. Note that if $0 \leq t_2 \leq \dots \leq t_k $, then
\begin{equation}\label{4.28}
\varpi (y, t_2, \dots, t_k) = \max \left( 0,  1 - y - t_3 - \dots - t_k  \right).
\end{equation}
We again use the fact, that any permutation of $t_2, \dots , t_k$ does not change the integrand, so for $0 \leq y \leq 1$ it is true that $\int_1$ times $1/(k-1)!$ equals
\begin{multline}\label{4.29}
 \idotsint \limits_{\substack{ \mathcal{R}_{k-1} \\ 0 \leq t_2 \leq \dots \leq t_k}} \int \limits_{0}^{ \max \left( 0,  1 - y - t_3 - \dots - t_k  \right)}    \left( f(t_1 + \dots + t_k) - 
 f(y + t_1+ \dots + t_k) \right)^2\, dt_1\, dt_2 \, \dots\, dt_k \\ =
 \int  \limits_0^1 \idotsint \limits_{ 0\leqslant t_2 \leq \dots \leq t_{k-1} \leq t - \sum_{i=2}^{k-1} t_i } 
 \int \limits_{t}^{ \max \left( t,  1 - y + t_2  \right)} 
   \left( f(x) - 
 f(x+y) \right)^2\, dx\, dt_2\, dt_3 \, \dots\, dt_{k-1} \, dt  \\ =
  \int  \limits_0^1 \int \limits_0^{\frac{t}{k-1}}
\left(   \int \limits_{t}^{ \max \left( t,  1 - y +  s  \right)}  \left( f(x) -  f(x+y) \right)^2\, dx \right)
  \int \limits_s^{\frac{t-s}{k-2}}  \int \limits_{t_3}^{\frac{t-s-t_3}{k-3}}  \cdots \int \limits_{t_{k-2}}^{\frac{t-s-t_3-\dots-t_{k-2}}{2}}  
 \, dt_{k-1} \, \dots\, dt_3 \, ds\, dt \\ =
  \frac{1}{(k-1)!(k-3)!}  \int  \limits_0^1 \int \limits_0^t
  \int \limits_{t}^{ \max \left( t,  1 - y + \frac{s}{k-1}  \right)}  \left( f(x) -  f(x+y) \right)^2  
  (t-s)^{k-3}\, dx\, ds\, dt.
 \end{multline}

Let us move on to the $\int_2$ case. By the same argument as above, this integral times $1/(k-1)!$ equals
\begin{multline}\label{4.30}
 \idotsint \limits_{\substack{ \mathcal{R}_{k-1} \\ 0 \leq t_2 \leq \dots \leq t_k}} \int \limits_{ \max \left( 0,  1 - y - t_3 - \dots - t_k  \right)}^{1-t_3-\dots - t_k}    f(t_1 + \dots + t_k )^2\, dt_1\, dt_2 \, \dots\, dt_k  \\ 
=  \int  \limits_0^1 \idotsint \limits_{ 0\leqslant t_2 \leq \dots \leq t_{k-1} \leq t - \sum_{i=2}^{k-1} t_i } 
 \int \limits_{ \max \left( t,  1 - y + t_2  \right)}^{1+t_2} 
   f(x)^2\, dx\, dt_2\, dt_3 \, \dots\, dt_{k-1} \, dt  \\ 
=   \frac{1}{(k-1)!(k-3)!}  \int  \limits_0^1 \int \limits_0^t
  \int \limits_{ \max \left( t,  1 - y + \frac{s}{k-1}  \right)}^{1+\frac{s}{k-1}} f(x)^2  
  (t-s)^{k-3}\, dx\, ds\, dt.
  \end{multline}
Note that the expression above equals $J$ for $y>1$. 

From (\ref{4.26}) and (\ref{4.27}) we have
\begin{equation}\label{4.31}
\begin{gathered}
 J_0 =  \int \limits_0^{\vartheta_0/\vartheta} \frac{\vartheta_0 - \vartheta y}{\vartheta_0 y} \left( \int_1 + \int_2 \right)\, dy. 
\end{gathered}
\end{equation}
In the $k=5$ case we can write that the integral in the expression above equals
\begin{equation}
\frac{1}{2}\left( J_{0;1} + J_{0;2} \right),
\end{equation}
where
\begin{multline}
J_{0;1} =   \left(  \int  \limits_{0}^{\frac{1}{4}}   \int  \limits_{0}^{1-y} \int  \limits_{0}^{t}   \int  \limits_{t}^{1-y+\frac{s}{4}} ~+~
 \int  \limits_{\frac{1}{4}}^{1}     \int  \limits_{0}^{1-y} \int  \limits_{0}^{t}   \int  \limits_{t}^{1-y+\frac{s}{4}}~+~
   \int  \limits_{0}^{\frac{1}{4}} \int  \limits_{1-y}^{1} \int  \limits_{4t+4y-4}^{t}    \int  \limits_{t}^{1-y+\frac{s}{4}}  \right. \\
+ 
 \left.
   \int  \limits_{\frac{1}{4}}^{1}  \int  \limits_{1-y}^{\frac{4-4y}{3} } \int   \limits_{4t+4y-4}^{t}   \int  \limits_{t}^{1-y+\frac{s}{4}} \right)
 \frac{3 - 2 y}{3 y}   \left( f(x) -  f(x+y) \right)^2  
  (t-s)^2\, dx\, ds\, dt\, dy<
  11.3104037062 ,
  \end{multline}
  and
  \begin{multline}
J_{0;2} =   \left(   
     \int  \limits_{0}^{\frac{1}{4}}   \int  \limits_{0}^{1-y} \int  \limits_{0}^{t}   \int  \limits_{1-y+\frac{s}{4}}^{1+\frac{s}{4}} ~+~
     \int  \limits_{\frac{1}{4}}^{1}    \int  \limits_{0}^{1-y} \int  \limits_{0}^{t}   \int  \limits_{1-y+\frac{s}{4}}^{1+\frac{s}{4}} ~+~
       \int  \limits_{0}^{\frac{1}{4}}   \int  \limits_{1-y}^{1} \int  \limits_{4t+4y-4}^{t}    \int  \limits_{1-y+\frac{s}{4}}^{1+\frac{s}{4}}  \right. \\
 \left. +
   \int  \limits_{\frac{1}{4}}^{1}  \int  \limits_{1-y}^{\frac{4-4y}{3} } \int   \limits_{4t+4y-4}^{t}   \int  \limits_{t}^{1-y+\frac{s}{4}} ~+~
      \int  \limits_{\frac{1}{4}}^{1} \int  \limits_{\frac{4-4y}{3}}^{1} \int  \limits_{0}^{t}   \int  \limits_{t}^{1+\frac{s}{4}}
~+~ \int  \limits_{0}^{\frac{1}{4}} \int  \limits_{1-y}^{1} \int  \limits_{0}^{4t+4y-4}   \int  \limits_{t}^{1+\frac{s}{4}}          \right. \\
          \left.  
      ~+~      \int  \limits_{\frac{1}{4}}^{1}    \int  \limits_{1-y}^{\frac{4-4y}{3}} \int   \limits_{0}^{4t+4y-4}   \int  \limits_{t}^{1+\frac{s}{4}} ~+~
            \int  \limits_{1}^{\frac{3}{2}} \int  \limits_{0}^{1} \int  \limits_{0}^{t}   \int  \limits_{t}^{1+\frac{s}{4}}
          \right)
 \frac{3 - 2 y}{3 y}   f \left( x \right)^2  
  (t-s)^2\, dx\, ds\, dt\, dy <  20.2508453206 .
  \end{multline}
 A direct calculation shows that 
\begin{equation}\label{J0_wynik}
J_0 < 15.7806245134. 
\end{equation}

\subsubsection{Calculating $J_{2,1}$}

We have
\begin{equation}\label{4.35}
\begin{gathered}
J_{2,1}= \int \limits_0^{\frac{3}{2}} \frac{1-\vartheta_0 - \vartheta y}{\vartheta_0 y(1- \vartheta y)} 
 \idotsint \limits_{\mathcal{R}_{k-1}} \left( \int \limits_0^1 \left(  F(t_1, \dots , t_k) - F(t_1+y,t_2, \dots , t_k) \right)\, dt_1 \right)^2 dt_2 \, \dots\, dt_k\, dy .
\end{gathered} 
\end{equation}
The inner integral equals
\begin{equation}
\int \limits_0^y F(t_1, \dots , t_k)\, dt_1 =  \int \limits_0^{\max (y, \rho (t_2, \dots , t_k))} f(t_1 + \dots + t_k)\, dt_1  .
\end{equation}
Proceeding very much like in the $J_{1,1}$ case, we deduce that the integral over $\mathcal{R}_{k-1}$ from (\ref{4.35}) equals
\begin{equation}
 (k-1)! \int  \limits_0^1 \idotsint \limits_{ 0\leqslant t_2 \leq \dots \leq t_{k-1} \leq t - \sum_{i=2}^{k-1} t_i }  \left( \int \limits_t^{\max( t+y , 1+t_2 )} f(x)\, dx \right)^2\, dt_2 \, \dots\, dt_{k-1} \, dt.
\end{equation}
We put
\begin{equation}
\mbox{Int}_a^{b}(u) =  \int \limits_{\min (a , 1+u) }^{\max( b , 1+u )} f(x)\, dx.
\end{equation}
Continuing the reasoning from the $J_{1,1}$ case we get
\begin{equation}\label{4.39}
J_{2,1} = \frac{1}{(k-3)!} \int \limits_0^{\frac{3}{2}}
 \int  \limits_0^1  \int \limits_0^t  \frac{1-\vartheta_0 - \vartheta y}{\vartheta_0 y(1- \vartheta y)}  \left( \mbox{Int}_t^{t+y} \left( \frac{s}{k-1} \right) \right)^2   (t-s)^{k-3}\, ds\, dt\, dy  .
\end{equation}
We decompose the outer integral from (\ref{4.39}) as follows: 
\begin{equation}
 \int \limits_0^{\frac{3}{2}}  \int  \limits_0^1  \int \limits_0^t = \int \limits_{R_{2;1}} + \int \limits_{R_{2;2}},
\end{equation}
where
\begin{align}\label{4.41}
\begin{split}
R_{2;1} &= \{ (y,t,s) \in \mathbf{R}^3 \colon 0<y<\frac{3}{2} ,~ 0<t<1,~ 0<s<t,~  t+y > 1+\frac{s}{k-1}  \},  \\
R_{2;2} &= \{ (y,t,s) \in \mathbf{R}^3 \colon 0<y<\frac{3}{2} ,~ 0<t<1,~ 0<s<t,~  t+y < 1+\frac{s}{k-1}  \}.
\end{split}
\end{align}
We see that 
\begin{align}\label{4.42}
\begin{split}
\mbox{Int}_t^{t+y} \left( \frac{s}{k-1} \right) =   
\begin{cases}
\int \limits_{t }^{1 + \frac{s}{k-1}} f(x)\, dx  &\text{if}~~  (y,t,s) \in R_{2;1}, \\
\int \limits_{t }^{t+y} f(x)\, dx ~~ &\text{if}~~ (y,t,s) \in R_{2;2}.
\end{cases}
\end{split}
\end{align}
For $k=5$ we have 
\begin{multline}\label{4.43}
\int \limits_{R_{2;1}} = \left(  
\int \limits_1^{\frac{3}{2}}  \int  \limits_0^{1}  \int \limits_0^t  ~ + ~
 \int \limits_{\frac{1}{4}}^1  \int  \limits_{\frac{4-4y}{3}}^{1}  \int \limits_0^t  ~ + ~
  \int \limits_0^{\frac{1}{4}}  \int  \limits_{1-y}^1  \int \limits_0^{4t+4y-4}  ~ + ~
   \int \limits_{\frac{1}{4}}^1  \int  \limits_{1-y}^{\frac{4-4y}{3}}  \int \limits_0^{4t+4y-4}    \right)
    \frac{1-\vartheta_0 - \vartheta y}{\vartheta_0 y(1- \vartheta y)}\\
\times    \left(   \int \limits_{t }^{1+\frac{s}{4}} f(x)\, dx  \right)^2   (t-s)^{k-3}\, ds\, dt\, dy 
     > 15.2749404974
\end{multline}
and
\begin{multline}\label{4.44}
\int \limits_{R_{2;2}} = \left(  
\int \limits_0^1  \int  \limits_0^{1-y}  \int \limits_0^t  ~ + ~
  \int \limits_0^{\frac{1}{4}}  \int  \limits_{1-y}^1  \int \limits_{4t+4y-4}^t  ~ + ~
   \int \limits_{\frac{1}{4}}^1  \int  \limits_{1-y}^{\frac{4-4y}{3}}  \int \limits_{4t+4y-4}^t    \right)
    \frac{1-\vartheta_0 - \vartheta y}{\vartheta_0 y(1- \vartheta y)}  \\
\times      \left(   \int \limits_{t }^{t+y} f(x)\, dx  \right)^2   (t-s)^{k-3}\, ds\, dt\, dy 
     > 16.5050961382.
\end{multline}
Combining (\ref{4.43}) and (\ref{4.44}) we get
\begin{equation}\label{J2,1_wynik}
J_{2,1} > 15.8900183178.
\end{equation}

\subsubsection{Calculating $J_{3,1}$}

Let us define 
\begin{equation}\label{4.46}
\mathcal{A}_r':= \left\{  x\in [0,4]^{r-1} : 0 < x_1 < \dots < x_{r-1}< 4\vartheta_0 , ~\sum_{i=1}^{r-1} x_i <  \min \left( 4(1 - \vartheta_0), 4 - x_{r-1} \right)     \right\}. 
\end{equation}
Perfoming calculations analogous to the $J_{2,1}$ case we get
\begin{equation}\label{4.47}
\begin{gathered}
J_{3,1} = \frac{1}{(k-3)!} \iint  \limits_{\mathcal{A}_3'}  \int  \limits_0^1  \int \limits_0^t  
 \frac{1-\vartheta_0 - \vartheta (y+z)}{\vartheta_0 y z (1- \vartheta (y+z))}
 \left( \mbox{Int}_t^{t+z}   - \mbox{Int}_{t+y}^{t+y+z}  \right)^2   (t-s)^{k-3}\, ds\, dt\, dz\, dy,
\end{gathered}
\end{equation}
where we write $\mbox{Int}_a^{b}(u) $ instead of $\mbox{Int}_a^{b}\left( \frac{s-1}{k} \right)$ for the sake of clarity. We decompose the integral as follows:
\begin{equation}\label{4.48}
\begin{gathered}
\iint  \limits_{\mathcal{A}_3'}  \int  \limits_0^1  \int \limits_0^t  = \int \limits_{R_{3;1}} + \int \limits_{R_{3;2}} + \int \limits_{R_{3;3}} + \int \limits_{R_{3;4}},
\end{gathered}
\end{equation}
where
\begin{align} 
\begin{split}
R_{3;1} &= R_3 ~ \cap ~ \{ (y,z,t,s) \in \mathbf{R}^4 \colon  t < 1+\frac{s}{k-1} < t+z \},  \\
R_{3;2} &= R_3 ~ \cap ~ \{ (y,z,t,s) \in \mathbf{R}^4 \colon  t+z < 1+\frac{s}{k-1} < t+y \},  \\
R_{3;3} &= R_3 ~ \cap ~ \{ (y,z,t,s) \in \mathbf{R}^4 \colon  t+y < 1+\frac{s}{k-1} < t+y+z \},  \\
R_{3;4} &= R_3 ~ \cap ~ \{ (y,z,t,s) \in \mathbf{R}^4 \colon  t+y+z < 1+\frac{s}{k-1} \} 
\end{split}
\end{align}
with
\begin{equation}
R_3 =  \{ (y,z,t,s) \in \mathbf{R}^4 \colon 0<z<y<\frac{3}{2}, ~y+z<\frac{5}{2},~ 0<t<1,~ 0<s<t \}.
\end{equation}
We have
\begin{equation}\label{4.51}
 \mbox{Int}_t^{t+z}   - \mbox{Int}_{t+y}^{t+y+z} = 
\begin{cases}
  \int \limits_{t }^{1+\frac{s}{k-1}} f(x)\, dx, & \text{if~} (y,z,t,s)  \in R_{3;1}, \\
    \int \limits_{t }^{t+z} f(x)\, dx, & \text{if~} (y,z,t,s)  \in R_{3;2}, \\
    \left(  \int \limits_{t }^{t+z} - \int \limits_{t+y}^{1+\frac{s}{k-1}} \right) f(x)\, dx, & \text{if~} (y,z,t,s)  \in R_{3;3}, \\
    \left(  \int \limits_{t }^{t+z} - \int \limits_{t+y}^{t+y+z} \right) f(x)\, dx, & \text{if~} (y,z,t,s)  \in R_{3;4}. 
\end{cases}
\end{equation}
With (\ref{4.51}) we are able to decompose the integrals $\int_{R_{3;i}}$ into integrals with explicit limits. For $k=5$ and different cases $i=1,2,3,4$ we get $12, 24, 29, 10$ such integrals respectively which gives 75 of them in total. The results are
\begin{align}
\begin{split}
\int \limits_{R_{3;1}} &> 4.968 ,~~~~\int \limits_{R_{3;2}} > 12.158,   \\
\int \limits_{R_{3;3}} &> 4.227 ,~~~~\int \limits_{R_{3;4}} > 1.831 ,     
\end{split}
\end{align}
which implies
\begin{equation}\label{J3,1_wynik}
J_{3,1} > 11.592.
\end{equation}

\subsubsection{Calculating $J_{4,1}$}

We follow the steps from the last subsubsection. We have 
\begin{multline}\label{}
J_{4,1} = \frac{1}{(k-3)!} \iiint  \limits_{\mathcal{A}_4'}  \int  \limits_0^1  \int \limits_0^t  
 \frac{1-\vartheta_0 - \vartheta (y+z+w)}{\vartheta_0 y z w (1- \vartheta (y+z+w))} \\
\times  \left( \mbox{Int}_t^{t+w}   - \mbox{Int}_{t+z}^{t+z+w} 
- \mbox{Int}_{t+y}^{t+y+w}   + \mbox{Int}_{t+y+z}^{t+y+z+w}  \right)^2   (t-s)^{k-3}\, ds\, dt\, dw\, dz\, dy.
\end{multline}
We decompose 
\begin{equation}
\begin{gathered}
\iiint  \limits_{\mathcal{A}_4'}  \int  \limits_0^1  \int \limits_0^t  = \sum_{i=1}^{11} \int \limits_{R_{4;i}},
\end{gathered}
\end{equation}
where 
\begin{align} 
\begin{split}
R_{4;1} &= R_4 ~ \cap ~ \{ (y,z,w,t,s) \in \mathbf{R}^5 \colon  t < 1+\frac{s}{k-1} < t+w \},  \\
R_{4;2} &= R_4 ~ \cap ~ \{ (y,z,w,t,s) \in \mathbf{R}^5 \colon  t+w < 1+\frac{s}{k-1} < t+z \},  \\
R_{4;3} &= R_4 ~ \cap ~ \{ (y,z,w,t,s) \in \mathbf{R}^5 \colon  t+z < 1+\frac{s}{k-1} < t+z+w , ~ y>z+w \},  \\
R_{4;4} &= R_4 ~ \cap ~ \{ (y,z,w,t,s) \in \mathbf{R}^5 \colon  t+z+w < 1+\frac{s}{k-1} < t+y , ~ y>z+w \},  \\
\end{split}
\end{align}
\begin{align} 
\begin{split}
R_{4;5} &= R_4 ~ \cap ~ \{ (y,z,w,t,s) \in \mathbf{R}^5 \colon  t+y < 1+\frac{s}{k-1} < t+y+w , ~ y>z+w \},  \\
R_{4;6} &= R_4 ~ \cap ~ \{ (y,z,w,t,s) \in \mathbf{R}^5 \colon  t+z < 1+\frac{s}{k-1} < t+y  , ~ y<z+w \}, \\
R_{4;7} &= R_4 ~ \cap ~ \{ (y,z,w,t,s) \in \mathbf{R}^5 \colon  t+y < 1+\frac{s}{k-1} < t+z+w , ~ y<z+w \},  \\
R_{4;8} &= R_4 ~ \cap ~ \{ (y,z,w,t,s) \in \mathbf{R}^5 \colon  t+z+w < 1+\frac{s}{k-1} < t+y+w , ~ y<z+w \},  \\
R_{4;9} &= R_4 ~ \cap ~ \{ (y,z,w,t,s) \in \mathbf{R}^5 \colon  t+y+w < 1+\frac{s}{k-1} < t+y+z \},  \\
R_{4;10} &= R_4 ~ \cap ~ \{ (y,z,w,t,s) \in \mathbf{R}^5 \colon t+y+z < 1+\frac{s}{k-1} < t+y+z+w \},  \\
R_{4;11} &= R_4 ~ \cap ~ \{ (y,z,w,t,s) \in \mathbf{R}^5 \colon t+y+z+w < 1+\frac{s}{k-1} \}
\end{split}
\end{align}
with
\begin{equation}
R_4 =  \{ (y,z,w,t,s) \in \mathbf{R}^5 \colon 0<w<z<y<\frac{3}{2}, ~y+z+w<\frac{5}{2},~ 0<t<1,~ 0<s<t \}.
\end{equation}
The expression $ \left( \mbox{Int}_t^{t+w}   - \mbox{Int}_{t+z}^{t+z+w} - \mbox{Int}_{t+y}^{t+y+w}   + \mbox{Int}_{t+y+z}^{t+y+z+w} \right)$ equals
\begin{align*}
 &\int \limits_{t }^{1+\frac{s}{k-1}} f(x)\, dx,  ~~~~ &
& \int \limits_{t }^{t+w} f(x)\, dx , ~~~~ \\
&\left(  \int \limits_{t }^{t+w} - \int \limits_{t+z}^{1+\frac{s}{k-1}} \right) f(x)\, dx, ~~~~&
&\left(  \int \limits_{t }^{t+w} - \int \limits_{t+z}^{t+z+w} \right) f(x)\, dx, \\
&\left(  \int \limits_{t }^{t+w} - \int \limits_{t+z}^{t+z+w} - \int \limits_{t+y}^{1+\frac{s}{k-1}} \right) f(x)\, dx,~~~~&
&\left(  \int \limits_{t }^{t+w} - \int \limits_{t+z}^{1+\frac{s}{k-1}}  \right) f(x)\, dx, \\
&\left(  \int \limits_{t }^{t+w} - \int \limits_{t+z}^{1+\frac{s}{k-1}} - \int \limits_{t+y}^{1+\frac{s}{k-1}} \right) f(x)\, dx,~~~~ &
&\left(  \int \limits_{t }^{t+w} - \int \limits_{t+z}^{t+z+w} - \int \limits_{t+y}^{1+\frac{s}{k-1}} \right) f(x)\, dx, \\
&\left(  \int \limits_{t }^{t+w} - \int \limits_{t+z}^{t+z+w} - \int \limits_{t+y}^{t+y+w} \right) f(x)\, dx, ~~~~&
&\left(  \int \limits_{t }^{t+w} - \int \limits_{t+z}^{t+z+w} - \int \limits_{t+y}^{t+y+w} + \int \limits_{t+y+z}^{1+\frac{s}{k-1}} \right) f(x)\, dx, \\
&\left(  \int \limits_{t }^{t+w} - \int \limits_{t+z}^{t+z+w} - \int \limits_{t+y}^{t+y+w} +  \int \limits_{t+y+z}^{t+y+z+w} \right) f(x)\, dx  \\
\end{align*}
for $(y,z,w,t,s)$ belonging to $R_{4;1}, \dots , R_{4;11}$ respectively. For $k=5$ this allows us to decompose the integrals $\int_{R_{4;i}}$ into integrals with explicit limits and there are $1337$ of them in total. The results are

\begin{align} 
\begin{split}
\int \limits_{R_{4;1}} &> 0.392 ,~~~~\int \limits_{R_{4;2}} > 1.538 , ~~~~\int \limits_{R_{4;3}} > 0.851 ,\\
\int \limits_{R_{4;4}} &> 0.608 ,~~~~\int \limits_{R_{4;5}} > 0.093 ,~~~~ \int \limits_{R_{4;6}} > 0.410 , \\
\int \limits_{R_{4;7}} &> 0.307 ,~~~~\int \limits_{R_{4;8}} > 0.073 ,~~~~\int \limits_{R_{4;9}} > 0.066 ,   \\
\int \limits_{R_{4;10}} &> 0.017 , ~~~\int \limits_{R_{4;11}} > 0.010 
\end{split}
\end{align}
so the overall contribution satisfies
\begin{equation}\label{J4,1_wynik}
J_{4,1} > 4.365.
\end{equation}

\subsubsection{Calculating $J_{2,2}$ and $J_{3,2}$}

Proceeding as in the $J_{1,1}$ case we calculate
\begin{align}\label{J2,2_wynik}
J_{2,2} &= \int \limits_{\frac{3}{2}}^2 \frac{1-2\vartheta_0 }{\vartheta_0 y(1- \vartheta y)}
 \idotsint \limits_{\substack{ \mathcal{R}_{k-1}}} \left( \int \limits_0^1 \left(  F(t_1, \dots , t_k)  \right)\, dt_1 \right)^2\, dt_2 \, \dots\, dt_k\, dy  \\
&=  \frac{1}{(k-3)!} \int \limits_{\frac{3}{2}}^2
 \int  \limits_0^1  \int \limits_0^t  \frac{1-2\vartheta_0 }{\vartheta_0 y(1- \vartheta y)}  \left( \int_t^{1+\frac{s}{k-1}} f(x)\, dx \right)^2   (t-s)^{k-3}\, ds\, dt\, dy  > 
 1.9342154969. \nonumber
\end{align}

Let us move our attention to the $J_{3,2}$ case. Define
\begin{align}
\begin{split}
\mathcal{A}_r'' = 
 \Bigg\{  x\in [0,4]^{r-1} \colon 0<x_1<\dots<x_{r-s} \leq 4\vartheta_0 < x_{r-s+1} < \dots \\
 < x_{r-1}, ~\sum_{i=1}^{r-1} x_i < \min \left( 4( 1 -  \vartheta_0 ), 4 - x_{r-1} \right)     \Bigg\}.
\end{split}
\end{align}
We have
\begin{equation}
\begin{gathered}
J_{3,2} = \frac{1}{(k-3)!} \iint  \limits_{\mathcal{A}_3''}  \int  \limits_0^1  \int \limits_0^t  
 \frac{1- 2\vartheta_0 - \vartheta z}{\vartheta_0 y z (1- \vartheta (y+z))}
 \left( \mbox{Int}_t^{t+z}   - \mbox{Int}_{t+y}^{t+y+z}  \right)^2   (t-s)^{k-3}\, ds\, dt\, dz\, dy ,
\end{gathered} 
\end{equation}

Note that $\mbox{Int}_{t+y}^{t+y+z}$ vanishes because $y>3/2$. We again decompose the integral in order to simplify the integrand. 

\begin{equation}\label{4.57'}
\begin{gathered}
\iint  \limits_{\mathcal{A}_3''}  \int  \limits_0^1  \int \limits_0^t  = \int \limits_{R_{3;1}} + \int \limits_{R_{3;2}},
\end{gathered}
\end{equation}
where
\begin{align} 
\begin{split}
R_{3;1}' &= R_3' ~ \cap ~ \{ (y,z,t,s) \in \mathbf{R}^4 \colon  t < 1+\frac{s}{k-1} < t+z \},  \\
R_{3;2}' &= R_3' ~ \cap ~ \{ (y,z,t,s) \in \mathbf{R}^4 \colon  t+z < 1+\frac{s}{k-1}  \}
\end{split}
\end{align}
with
\begin{equation}
R_3' =  \{ (y,z,t,s) \in \mathbf{R}^4 \colon 0<z<\frac{3}{2}<y, ~y+z<\frac{5}{2},~2y+z<4,~ 0<t<1,~ 0<s<t \}.
\end{equation}
We get
\begin{equation}\label{4.59'}
 \mbox{Int}_t^{t+z}   - \mbox{Int}_{t+y}^{t+y+z} = 
\begin{cases}
  \int \limits_{t }^{1+\frac{s}{k-1}} f(x)\, dx, & \text{if~} (y,z,t,s)  \in R_{3;1}', \\
    \int \limits_{t }^{t+z} f(x)\, dx, & \text{if~} (y,z,t,s)  \in R_{3;2}', \\
\end{cases}
\end{equation}
Therefore, for $k=5$ we have
\begin{multline}
\int \limits_{R_{3;1}'} = 
 \left(  
\int \limits_{\frac{3}{2}}^{\frac{15}{8}} \int \limits_{\frac{1}{4}}^{4-2y} \int  \limits_{\frac{4-4z}{3}}^{1}  \int \limits_0^t  ~ + ~
\int \limits_{\frac{3}{2}}^{\frac{15}{8}} \int \limits_{0}^{\frac{1}{4}} \int  \limits_{1-z}^{1}  \int \limits_{0}^{4t+4z-4}  ~ + ~
\int \limits_{\frac{15}{8}}^{2} \int \limits_{0}^{4-2y} \int  \limits_{1-z}^{1}  \int \limits_{0}^{4t+4z-4}  ~ + ~
\int \limits_{\frac{3}{2}}^{\frac{15}{8}} \int \limits_{\frac{1}{4}}^{4-2y} \int  \limits_{1-z}^{\frac{4-4z}{3}}  \int \limits_0^{4t+4z-4}    \right)  \\
     \frac{1- 2\vartheta_0 - \vartheta z}{\vartheta_0 y z (1- \vartheta (y+z))} 
     \left(   \int \limits_{t }^{1+\frac{s}{4}} f(x)\, dx  \right)^2   (t-s)^{k-3}\, ds\, dt\, dz\, dy 
      > 0.347 , 
      \end{multline}
\begin{multline}
\int \limits_{R_{3;2}'} = \left( 
\int \limits_{\frac{3}{2}}^{\frac{15}{8}} \int \limits_{0}^{4-2y} \int  \limits_{0}^{1-z}  \int \limits_0^t  ~ + ~
\int \limits_{\frac{15}{8}}^{2} \int \limits_{0}^{4-2y} \int  \limits_{0}^{1-z}  \int \limits_0^t  ~ + ~ 
\int \limits_{\frac{3}{2}}^{\frac{15}{8}} \int \limits_{0}^{\frac{1}{4}} \int  \limits_{1-z}^{1}  \int \limits_0^{4t+4z-4}   \right.
 \\  
\left. +~
\int \limits_{\frac{15}{8}}^{2} \int \limits_{0}^{4-2y} \int  \limits_{1-z}^{1}  \int \limits_0^{4t+4z-4}  ~ + ~
\int \limits_{\frac{3}{2}}^{\frac{15}{8}} \int \limits_{\frac{1}{4}}^{4-2y} \int  \limits_{1-z}^{\frac{4-4z}{3}}  \int \limits_0^{4t+4z-4} 
 \right)   \frac{1- 2\vartheta_0 - \vartheta z}{\vartheta_0 y z (1- \vartheta (y+z))}  \\
\times      \left(   \int \limits_{t }^{t+z} f(x)\, dx  \right)^2   (t-s)^{k-3}\, ds\, dt\, dz\, dy 
> 1.564 ,
\end{multline}
which gives
\begin{equation}\label{J3,2_wynik}
J_{3,2} > 0.9555 .
\end{equation}

\subsection{Proof of the unconditional part of Theorem \ref{MAIN}}

We combine (\ref{4.12}), (\ref{J1,1_wynik}), (\ref{J_wynik}), (\ref{J0_wynik}), (\ref{J2,1_wynik}), (\ref{J3,1_wynik}), (\ref{J4,1_wynik}), (\ref{J2,2_wynik}), (\ref{J3,2_wynik}) and prove that
\begin{equation}\label{4.73'}
\begin{gathered}
 \Upsilon_5^{\text{ext}} \left( F; \frac{1}{4}, \frac{3}{8} \right) \leq 
 \frac{5 \left( J_0 - \frac{1}{4} \left( J_{1,1} + J_{2,1}+J_{3,1}+J_{4,1}+J_{2,2}+J_{3,2} \right)  \right)}{J} + \frac{40}{3} < 14.98582.
\end{gathered}
\end{equation}
Hence, equation (\ref{Oeps}) implies
\begin{equation}
\widetilde{\Upsilon}_5^{\text{ext}} \left(F; \frac{1}{4} - \epsilon, \frac{3}{8}, \epsilon \right) < 14.98582 + O (\epsilon ).
\end{equation}
The inequatility above combined with Theorem \ref{tw_J} implies the first part of Theorem \ref{MAIN}. 

\begin{remark}
Taking $k=3$ and 
\begin{equation*} F(t_1,t_2,t_3) = 
\begin{cases}
256 + 819(1-P_1) + 833(1-P_1)^2, & \text{if~} (t_1,t_2,t_3) \in \mathcal{R}'_3,  \\
0, & \text{otherwise, }
\end{cases}
\end{equation*}
 we can perform analogous calculations in order to show that
\[
\widetilde{\Upsilon}_3^{\text{ext}} \left( F; \frac{1}{4} - \epsilon, \frac{3}{8}, \epsilon \right) < 7.928 + O (\epsilon ). \]
This reproves the main theorem from \cite{3-tuples}.
\end{remark}

\bibliographystyle{amsplain}


\end{document}